\documentclass[final]{amsart}
\usepackage{graphicx} 
\usepackage{amsmath}
\usepackage{mathtools}
\usepackage{amssymb}
\usepackage{amsthm}
\usepackage[colorlinks]{hyperref}
\usepackage[notref,notcite,color]{showkeys}
\usepackage{enumitem}
\usepackage{bbm}
\usepackage{mathrsfs}
\usepackage{hyperref}
\usepackage[nameinlink]{cleveref}

\newtheorem{ass}{Assumption}[]

\newtheorem{thm}{Theorem} 
\newtheorem{deff}[thm]{Definition}
\newtheorem{lem}[thm]{Lemma}
\newtheorem{cor}[thm]{Corollary}

\newtheorem{rem}[thm]{Remark}

\numberwithin{thm}{section} 
\numberwithin{equation}{section} 

\newcommand{\E}{\mathbb{E}}
\newcommand{\U}{\mathcal{U}}
\newcommand{\Cc}{\mathcal{C}_c}
\newcommand{\Prb}{\mathbb{P}}
\newcommand{\mcM}{\mathcal{M}}
\newcommand{\ncN}{\mathcal{N}}
\newcommand{\Pp}{\mathcal{P}}
\newcommand{\N}{\mathbb{N}}
\newcommand{\R}{\mathbb{R}}
\newcommand{\Id}{\mathbbm{1}}
\newcommand{\inv}{\mathcal{M}_+}

\newcommand{\D}{\mathbb{D}}
\newcommand{\Dm}{\mathcal{D}}
\newcommand{\Dme}{\mathcal{D}^{ext}}
\newcommand{\K}{\mathcal{K}}
\newcommand{\C}{\mathcal{C}}
\newcommand{\PP}{\mathscr{P}}
\newcommand{\UU}{\bar{U}}

\title[Stochastic Extinction]{Stochastic Extinction \\
\large An Average Lyapunov Function Approach}
\author[F\"oldes and Stacy]{Juraj F\"oldes and Declan Stacy}

\address{Juraj F\"oldes
\footnote{Juraj F\"oldes was partly supported under the grant NSF DMS-1816408}
\newline 
\indent Department of Mathematics, University of Virginia \indent 
\newline \indent   Kerchof Hall, Charlottesville, VA 22904-4137,\indent }
\email{foldes@virginia.edu}

\address{Declan Stacy
\newline 
\indent Department of Mathematics, University of Virginia (Undergraduate student) \indent 
\newline \indent   Kerchof Hall, Charlottesville, VA 22904-4137,\indent }
\email{fmw3cs@virginia.edu}

\begin{document}

\begin{abstract}
We study the stability of $\mathcal{M}_0$, an invariant subset
of a Markov process $(X_t)_{t\geq 0}$ on a metric space $\mathcal{M}$. By building the theory of average Lyapunov functions, we formulate general criteria based on the signs of Lyapunov exponents that guarantee extinction ($X_t \to \mathcal{M}_0$ as $t \to \infty$). Additionally, we provide applications to a stochastic SIS epidemic model on a network with regime-switching, a stochastic differential equation version of the Lorenz system, a general class of discrete-time ecological models, and stochastic Kolmogorov systems. In many examples we improve  existing results by removing unnecessary assumptions or providing sharper criteria for the extinction.



\end{abstract}

\maketitle

\section{Introduction}
The stability or instability of an invariant subset of a state space is a central problem in the analysis of stochastic and deterministic dynamical systems. 
Specifically, suppose a Markov process $(X_t)_{t\geq 0}$ takes values in a metric space $(\mcM, d)$ and there is a closed subset $\mcM_0 \subset \mcM$, called the extinction set, which is invariant: $X_t \in \mcM_0$ for all $t \geq 0$ if $X_0 \in \mcM_0$.
We investigate criteria ensuring that $\mcM_0$ is stable, meaning (roughly) that $d(X_0, \mcM_0) \ll 1$ implies $d(X_t,\mcM_0) \to 0$ as $t \to \infty$ (extinction).
In addition, we assume that
$X_t$ can only approach $\mcM_0$ ``asymptotically," meaning $X_t$ cannot enter $\mcM_0$ in finite time. For the theory of extinction in finite time we refer the reader to \cite{FiniteTime} and for a detailed discussion on the finite versus asymptotic extinction see \cite{persistence, SchSurvey}.
The stability and closely related questions, such as the rate of 
convergence of $d(X_t,\mcM_0) \to 0$ and the dependence on certain parameters, have been extensively investigated in many contexts.

In this manuscript we consider stochastic processes $(X_t)_{t\geq 0}$, and therefore the stability and the rate of convergence can in general depend on the realization of the process. Also, we allow $(X_t)_{t\geq 0}$ to have jump discontinuities, and in particular our results also apply to discrete-time Markov chains. 
Some processes which can be analyzed using our methods include:

 \textit{Ecological models}, including stochastic replicator dynamics in evolutionary game theory and more traditional predator-prey type systems (see \cite{ecologicalDiscrete, ecologicalContinuous, rps, Lot-Vol-Degenerate, Pred-Prey-Regime, ecologicalGeneral2, oldec1, oldec2, oldec3, oldec4, replicator, wierdSwitching1, ecologicalGeneral, wierdSwitching2}), where points in $\mcM$ represent the population densities of some interacting species and $\mcM_0$ is the subset where certain species are extinct.
 
 \textit{Epidemiological models} (see \cite{SIS2, SIRS,SIRS2, SIRS3, SIRS4, SIS, SIRV}), where points in $\mcM$ represent the prevalence of some disease and $\mcM_0$ (usually a single point) is the subset where the disease is eradicated.

 \textit{Chemostat models} (see \cite{Chemostat-general, Chemostat-1}), where points in $\mcM$ represent the concentration levels of bacteria and substrate and $\mcM_0$ is the subset where the concentration of bacteria is $0$.

 \textit{Piecewise Deterministic Markov Processes (PDMP)}
 (see \cite{pdmp1, pdmp2, pdmp3}), where $\mcM$ is the product of some phase space and a finite set of ``switches," which can be thought of as different environments or regimes. For a fixed regime, the process induces a deterministic flow on the phase space with a fixed point $0$. Then $\mcM_0$ consists of points which have $0$ in the phase space component.
 
\textit{Models for turbulence} of fluid flow (\cite{Lorenz, NS2, NS3, stochasticPDE}), where $\mcM_0$ consists of non-turbulent states. 

In this paper we introduce a general framework that provides stability criteria, rates of convergence, and dependence on the parameters with wide applicability, in particular, to the settings mentioned above. We leverage the powerful average Lyapunov function technique, substantially extending the ideas from \cite{Hofb1, AvgLyap, 71, 32}, where average Lyapunov functions were used to analyze deterministic ecological models.

The present work is partly motivated by \cite{persistence}, where the question of instability of $\mcM_0$, or so called persistence, was investigated. We remark that although our general setting and technical assumptions are quite similar and our results are in some sense dual to those in \cite{persistence}, the methods involved in the proofs and applications of our results are quite different, as discussed in detail below.

\subsection{Average Lyapunov Function.}
Lyapunov functions are a powerful tool for analyzing the stability of dynamical systems. Lyapunov functions are functions $f$ which have certain properties, such as $f \geq 0$ is nonnegative or $f(x) \to \infty$ as $x \to \mcM_0$. Also, a bound on the evolution of $f(X_t)$ is assumed,  for example $f(X_{t+1}) - f(X_t) \leq 0$, so that $f(X_t)$ is decreasing along positive integer time steps. For continuous-time processes, it is more common to postulate bounds on $\mathcal{L}f$, where $\mathcal{L}f$ encodes the average rate of change of $f(X_t)$. For example, if $X_t$ solves a deterministic differential equation $\frac{d}{dt}X_t = F(X_t)$, then $\mathcal{L}f = \nabla f \cdot F$ since $\frac{d}{dt}f(X_t) = \nabla f(X_t) \cdot \frac{d}{dt} X_t = (\nabla f \cdot F)(X_t)$. For random $X_t$, $\mathcal{L}f$ is something like $\frac{d}{dt}\E[f(X_t)]$. For precise definitions and properties $\mathcal{L}f$ see the main body of the manuscript below. After finding a Lyapunov function $f$ with suitable properties and bounds on $\mathcal{L} f$, one can derive various conclusions about the long-term behavior of $X_t$,
such as convergence to some point or staying inside some set forever.
However, even for deterministic $X_t$, Lyapunov functions are often extremely difficult to construct, usually because  $\mathcal{L} f$ needs to satisfy bounds on the entirety of $\mcM$ or $\mcM \backslash \mcM_0$.

Unlike a traditional Lyapunov function, an average Lyapunov function $V$ does not require bounds on $\mathcal{L} V$ everywhere. Instead, one considers the average value of $\mathcal{L} V$ with respect to certain measures $\mu$, meaning one only needs to estimate 
$$
\mu \mathcal{L} V := \int \mathcal{L} V(x) d\mu(x) \,,
$$ 
which we refer to as an (average) Lyapunov exponent. For example, in our paper we consider $\mu \mathcal{L} V$ for all  invariant measures $\mu$ on $\mcM_0$. 
Due to their weaker assumptions, 
average Lyanpunov functions are a lot easier to construct than traditional Lyapunov functions, but using them to derive properties of $X_t$ is
significantly harder. Further complications arise in the stochastic setting because $\mathcal{L} V$ only gives information about the expectation of the derivative of $V(X_t)$, and so additional requirements on $V$ are needed to show that, with a high probability, the expectation accurately captures the behavior of $X_t$. For example, often sufficient additional information is an estimate on the growth of $Var(V(X_t))$, quantified by $\Gamma V := \mathcal{L} V^2 - 2V \mathcal{L} V$. To summarize, traditional Lyapunov functions are difficult to construct and are problem-specific, 
 whereas average Lyapunov functions are often obvious to construct and instead the challenge lies in computing the Lyapunov exponents (which is problem-specific) and in using their properties to determine the long term behavior of $X_t$, which should not be problem-specific. We remark that both approaches are usually more involved in the stochastic setting. 

Most of the previous attempts to use an average Lyapunov function $V$ to analyze the long-term behavior of $X_t$ involve constructing a more traditional Lyapunov function which is some sort of combination of $e^{cV}$ for some $c \in \R$ and an additional Lyapunov function $W$ which prevents $X_t \to \infty$ in finite time. The combination of $e^{cV}$ and $W$ is delicate, and moreover it is quite nontrivial to deduce bounds on $\mathcal{L}e^{cV}$ or $\E[e^{cV(X_{t+T})} - e^{cV(X_t)}]$ for some $T > 0$ large enough assuming only some bounds on $\mu \mathcal{L} V$ for certain measures $\mu$. As a consequence, these approaches are unsuitable for generalization and require additional unnecessary assumptions. 

There are two notable arguments we found in the literature that avoid this sort of construction. The first one in \cite{Khas} only applies to linear stochastic differential equations, and therefore it cannot be easily generalized. However, its use of a change of coordinates is crucial, and we often leverage this technique
 when applying our theory to examples.
 The second is the proof by contradiction method presented in \cite{persistence} for solving the problem of persistence (coexistence), which cannot be applied to our problem of extinction. However, \cite{persistence} and the present work share the use of several tools such as certain martingales and empirical occupation measures.
 Both of these arguments are discussed in greater detail below in the introduction.

\subsection{Our Method}

In this paper we use an entirely new approach in order to prove our general results, which, in particular, apply to all of the 
examples listed above, see below for more details. 
Although the proof of our main result \Cref{main2} is somewhat involved, the result itself is a universal tool that can be easily applied.
We show that the long-term behavior of many processes is completely determined by Lyapunov exponents depending only on the behavior on an arbitrarily small neighborhood of the extinction set $\mcM_0$. In practice, this allows one to linearize the system about $\mcM_0$ and study that instead of the original system. Since our results apply to many different types of processes, each of which has a different notion of linearization, we do not phrase our results in terms of linearization, but
we discuss this connection to linearization later in the introduction.

We formulate sufficient conditions that guarantee $X_t \to \mcM_0$ as $t \to \infty$ (extinction) in terms of 
an average Lyapunov function $V$ defined on some subset of $\mcM \backslash \mcM_0$ and a corresponding Lyapunov exponent $\alpha$ which depends only on the behavior of $X_t$ near $\mcM_0$. The central idea of the proof is to construct a clever sequence of stopping times $\tau_n$ so that the discrete-time process $V(Y_n) \coloneqq V(X_{\tau_n})$ behaves similarly to a
Brownian motion with drift  (the mean and variance of $V(Y_n)$ grow linearly in $n$). This allows us to conclude that $X_t$ stays close to $\mcM_0$ for all times (stability in probability), and consequently  ergodic theory-type arguments imply that $X_t$ must approach $\mcM_0$ as $t \to \infty$ (asymptotic stability). If $\mcM_0$ satisfies an accessibility criterion,
then a renewal-type argument yields that the extinction occurs with probability $1$, regardless of the initial condition (global asymptotic stability). Additionally, the speed of the convergence is given explicitly as a function of the Lyapunov exponent, for example, if $V(x) = -\log{d(x,\mcM_0)}$, then $d(X_t^x,\mcM_0) = O(e^{-\alpha t})$.

The main difficulties and novelties of the proof occur in the construction of the stopping times $\tau_n$, which can be split into two phases. In the first phase we run the process until  $V(X_t)$ or its approximation $\int_0^t \mathcal{L} V(X_s)ds$ increases enough, and otherwise we stop the process before they decrease too much.
Obtaining the appropriate bounds is quite nontrivial
because our average Lyapunov function $V$ is not defined on $\mcM_0$, the set on which we have the most information about the long-term behavior of $X_t$. However, we assume that $\mathcal{L} V$ extends continuously to $\mcM_0$, and so for initial conditions $X_0 \in \mcM_0$ we can use ergodic arguments to control $\int_0^t \mathcal{L} V (X_s) ds$, a proxy for $V(X_t) - V(X_0)$.
To extend this control to initial conditions $X_0$ outside of $\mcM_0$
we require a strong continuous dependence of $X_t$ on $X_0$, namely continuity of the law of $\{X_t\}_{t \geq 0}$ viewed as a random element of the Skorokhod space $\D_{[0,\infty)}(\mcM)$.  It is well-known that such property holds when $X_t$ is Feller, but, 
motivated by applications, we need to extend it to 
 ``$C_b$-Feller" processes, which requires new proofs and additional stability assumptions. In the second phase, we let the process (quickly) return to a fixed compact set. The compactness is essential since the success of the first phase relies on having uniform bounds on $\E[V(X_{\tau_1}) - V(X_0)]$ and $\E[(V(X_{\tau_1}) - V(X_0))^2]$, but for each point in $\mcM_0$ we can only obtain bounds which hold 
  in a small neighborhood. Additionally, we must use clever manipulations of certain martingales to show that the second phase does not diminish the increase of $V(X_t)$ from the first phase.

\subsection{A Brief Summary of Persistence/Coexistence Theory}
Before further discussing the details of our manuscript, we summarize
the persistence result proved in \cite{persistence} in order to introduce important notation and underlying ideas for our manuscript in a technically easier setting.
Since the assumptions of the persistence result are 
almost complementary to ours, together \cite{persistence} and our paper completely characterize the long-term behavior of many Markov Processes. However, we again stress that our paper solves a different problem than \cite{persistence} does, and thus the methods and proofs are quite different.

In some sense the opposite to extinction,  $d(X_t,\mcM_0) \to 0$ as $t \to \infty$, is persistence, which heuristically means that $X_t$ spends an arbitrarily large fraction of time in compact subsets of $\mcM \setminus \mcM_0$. To formalize this definition, we introduce 
the empirical occupation measures on $\mcM$ as 
    \begin{equation}\label{intro-occ}
    \mu_t = \frac{1}{t}\int_0^t \delta_{X_s}ds,    
    \end{equation}
which are random measures tracking the time-averaged behavior of $X_t$. Then under mild stability assumptions it is known (see \Cref{lim-is-inv} and \Cref{lem:bhcon}) that $(\mu_t)_{t \geq 0}$ is tight and all limiting measures as $t \to \infty$ lie in the set $P_{inv}(\mcM)$ of invariant measures on $\mcM$, which represent the possible stationary behaviors of $X_t$. In other words, the invariant measures completely characterize the possible long-term time-averaged behaviors of $X_t$. While extinction forces all limiting measures of $(\mu_t)_{t \geq 0}$ as $t \to \infty$ to be supported on $\mcM_0$, persistence means that they are all supported on $\mcM \setminus \mcM_0$. Thus, in the context of empirical occupation measures, the notions of persistence and extinction are in some sense dual.

The crucial assumption that guarantees persistence in  \cite{persistence} is the existence of an average Lyapunov function 
$V: \mcM \setminus \mcM_0 \to \R$ such that $\mathcal{L} V: \mcM \setminus \mcM_0 \to \R$ extends to a continuous function $H$ on $\mcM$ and 
 the ``$H$-exponent" or ``average Lyapunov exponent"
\begin{equation}\label{other-exponent}
    \Lambda \coloneqq \sup_{\mu \in P_{inv}(\mcM_0)} \mu H \coloneqq \sup_{\mu \in P_{inv}(\mcM_0)} \int H(x) d\mu(x)
\end{equation}
is strictly negative. In our paper we instead assume that a similar quantity
\begin{equation}\label{our-exponent}
\alpha \coloneqq \inf_{\mu \in P_{inv}(\mcM_0)} \mu H
\end{equation}
is strictly positive.
Heuristically, $\Lambda < 0$ (resp. $\alpha > 0$) implies an average decrease (resp. increase) of $t \mapsto V(X_t)$ when $X_t$ is close to $\mcM_0$. 
In the case of $\Lambda < 0$, it is further assumed that $V \geq 0$, and so we obtain that $X_t$ cannot spend much time close to $\mcM_0$, while for $\alpha > 0$ we further assume that $V(x) \to \infty$ as $x \to \mcM_0$, and so $X_t$ must approach $\mcM_0$.

To provide more justifications, we summarize the argument given in \cite{persistence} for the easier case of $\Lambda < 0$ (persistence). Along with the empirical occupation measures \eqref{intro-occ}, a key tool in the proof is 
\begin{equation}\label{other-mg}
    M^V_t \coloneqq V(X_t) - V(X_0) - \int_0^t H(X_s)ds \,,
\end{equation}
 a martingale which allows us to relate the behavior of $V(X_t)$, which is not even defined if $X_0 \in \mcM_0$, to that of $\int_0^t H(X_s)ds$, which can be easily analyzed using ergodic theory.

Specifically, for any initial condition $X_0 \in \mcM \setminus \mcM_0$ and all large times $t$ it holds that
\begin{equation}\label{eq:iswap}
\frac{V(X_t)}{t} \approx \frac{1}{t} \int_0^t H(X_s)ds = \mu_t H \approx \int H(x) d\mu(x) \,,
\end{equation}
where $\mu \in P_{inv}(\mcM)$ is a limit point of $\mu_t$.
Hence,  $V \geq 0$ implies $\int H(x) d\mu(x) \geq 0$. Suppose for contradiction that $\mu(\mcM_0) > 0$. Since it is assumed that $\mcM_0$ and $\mcM \setminus \mcM_0$ are both invariant sets, we can decompose $\mu$ as 
$$
\mu = \mu(\mcM_0) \mu_1 + (1 - \mu(\mcM_0)) \mu_2 \,,
$$ 
where $\mu_1 \in P_{inv}(\mcM_0), \mu_2 \in P_{inv}(\mcM \setminus \mcM_0)$. If $\mu_2$ is ergodic and  
$X_0$ is distributed as $\mu_2$, then 
$$
\mu_2 H = \lim_{t \to \infty} \mu_t H = \lim_{t \to \infty} \frac{V(X_t)}{t} = 0 \,,
$$
where in the last equality we used Birkhoff's ergodic theorem to conclude that $X_t$ enters a compact subset of $\mcM \setminus \mcM_0$ infinitely often ($V$ is bounded on such a compact subset). Then by ergodic decomposition $\mu_2 H = 0$ for any 
$\mu_2 \in P_{inv}(\mcM \setminus \mcM_0)$. However,  we assumed $\mu_1 H \leq \Lambda < 0$, and consequently we obtain
$$0 \leq \mu H = \mu(\mcM_0) \mu_1 H + (1 - \mu(\mcM_0)) \mu_2 H \leq \mu(\mcM_0) \Lambda < 0 \,,$$
a contradiction.

Rigorously justifying \eqref{eq:iswap} requires verifying
two technical details. The first is $\frac{V(X_t)}{t} \approx \frac{1}{t} \int_0^t H(X_s)ds$, which is
equivalent to the strong law of large numbers for \eqref{other-mg}: $\lim_{t \to \infty} \frac{ M^V_t}{t} = 0$, which is assumed to hold in \cite{persistence}.
The second is $\mu_t H \approx \int H(x) d\mu(x)$, which by the ergodic theorems is valid for $\mu$-a.e. initial condition. However, the proof requires the validity for any initial condition in $\mcM \setminus \mcM_0$, which is deduced via the existence of a traditional Lyapunov function $W$ with certain properties (\Cref{as3}). In practice, $W$ can be almost always constructed. The existence of $W$ also yields the correct stability criteria that ensure the tightness and the invariance of limit points of the empirical occupation measures, which are also key to the proof.

The assumptions (besides $\alpha > 0$ versus $\Lambda < 0$) in our work and in  \cite{persistence} are similar and
let us highlight, mostly technical, differences. For example, instead of assuming that martingales satisfy the strong law of large numbers, in our manuscript we assume the stronger condition given in our \Cref{as5}, which involves another traditional Lyapunov function $U$. Although \Cref{as5} is stronger, in all examples we were able to locate in the literature, the strong law of large numbers is verified via \Cref{as5}, and so for practical purposes there is no difference in the assumptions. In addition, for the special cases of pure jump processes and SDEs (see \Cref{discretization} and \Cref{cts-paths}), we show that $U$ can be frequently constructed as a function of $W$, so  \Cref{as5} is automatically satisfied. Additionally,  \Cref{as5} is satisfied if $\mcM$ is compact (in fact almost all of the technical assumptions are trivial in this case, see \Cref{compact}).

Other differences in our assumptions and \cite{persistence} are intrinsic to the problem of extinction. For extinction we need to assume that $V(X_t) \to \infty$ implies that $d(X_t,\mcM_0) \to 0$ since by our techniques we cannot directly estimate $d(X_t,\mcM_0)$, but we can obtain a lot of information about $V(X_t)$. On the other hand, compared to \cite{persistence} we do not need to assume that $V \geq 0$ or that $V$ is defined on all of $\mcM \setminus \mcM_0$ (we only require $V$ to be defined on an open dense subset), which is essential for certain applications. 

The similarities of our assumptions and the ones in \cite{persistence} have significant practical consequences; for a large class of Markov processes, it suffices to investigate the sign of the $H$-exponents $\Lambda, \alpha$ to
determine persistence or extinction. 
 This is an improvement over the existing literature, where two completely separate arguments were needed, one for persistence and one for extinction. 
  In the examples below we focus on extinction, but with little additional extra work it is possible to formulate assumptions guaranteeing persistence. Similarly, the examples in \cite{persistence} have extinction analogues based on our theorems.

\subsection{Our Results as Linearization}

Before applying our techniques to complicated systems, we discuss a more classical example (linear SDEs) which illustrates the connection between the $H$-exponent $\alpha$ and the long-term behavior of $X_t$. We highlight \cite{Khas}'s use of a change of coordinates, a tool which is also applicable to more complicated examples considered in this manuscript. Finally, we remark that in the context of many (nonlinear) SDEs with a fixed point $0$, our \Cref{main} yields that $0$ is stable for the nonlinear SDE if $0$ is stable for its corresponding linearized system, and a similar principle also applies to SDEs with more complicated extinction sets.
 
In \cite{Khas} the authors  discuss the long term behavior of the linear stochastic system
\begin{equation}\label{linear-SDE}
    dx_t = Ax_tdt + \Sigma x_t dW_t \,,
\end{equation}  
where $A, \Sigma \in \R^{m \times m}$ are fixed $m \times m$ matrices, $x_t \in \mcM := \R^m$, and $W_t$ is a Brownian motion. In this case $\mcM_0 \coloneqq \{0\}$ contains only the origin, and we would like to find conditions that guarantee $x_t \to 0$ as $t \to \infty$ and the rate of convergence. 

Since the system is linear, we expect an exponential convergence (if it happens). Thus we set the average Lyapunov function to be $V(x) = - \log|x|$, where $|x| \coloneqq \Big(\sum_{i=1}^m x_i^2\Big)^{1/2}$ denotes the Euclidean norm, defined for all $x \notin \mcM_0$. Since this function only depends on $|x|$, it is natural to transform the problem to polar coordinates $(v,r) \in S^{m-1} \times [0,\infty) \eqqcolon \ncN$, where $S^{m-1} \coloneqq \{v \in \R^m \mid \|v\| = 1\}$ is a sphere.

Using It\^ o formula one can easily show that $\mathcal{L}$ is a differential operator and compute $\mathcal{L}V$
(see \Cref{verify-assumptions} below for details). Then one  notices that $\lim_{x \to 0} \mathcal{L}V(x)$ depends on the direction that $x$ approaches $0$, that is, $\lim_{t \to 0} \mathcal{L}V(tv)$ with $v \in S^{m-1}$ depends on $v$, and in particular $\mathcal{L}V$ cannot be continuously extended to $\mcM$.
 However, in polar coordinates the origin can be interpreted as $\ncN_0 \coloneqq S^{m-1} \times \{0\}$, that is, $\mcM_0$ becomes a sphere, and then $\mathcal{L} V$ indeed extends to a continuous function on $\ncN$, as detailed in the argument below.
Such approach is common in our examples, where  one establishes a ``boundary" of the possible trajectories of $X_t$ by expanding $\mcM_0$ to account for all the directions along which $X_t$ could approach $\mcM_0$. We formalize this change of variables  in \Cref{change-of-variables} below.
 
To provide more details, we set $x_t = r_tv_t$ and then It\^ o's formula yields that $(v_t,r_t)$ satisfies an SDE 
\begin{align*}
    dv_t &= f(v_t)dt + \sigma(v_t)dW_t \\
    dr_t &= r_tg(v_t)dt + r_t\eta(v_t)dW_t
\end{align*}
for continuous functions $f,\sigma: S^{m-1} \to \R^m$, $g,\eta: S^{m-1} \to \R$, and $V(v,r) \coloneqq -\log{r}$ satisfies $$dV_t = \Big[-g(v_t) + \frac{\eta(v_t)^2}{2}\Big]dt - \eta(v_t) dW_t \,.
$$
For
$\mathcal{L}V(v,r) \coloneqq -g(v) + \frac{\eta(v)^2}{2}$ we have that 
$$
M^V_t \coloneqq V(v_t,r_t) - V(v_0,r_0) - \int_0^t \mathcal{L}V(v_s,r_s)ds = -\int_0^t \eta(v_s)dW_s
$$ 
is a square-integrable martingale whose quadratic variation $\int_0^t \eta^2(v_s)ds$ is bounded by $Ct$ for some constant $C > 0$. Applying the strong law of large numbers for martingales we have that $\lim_{t \to \infty} \frac{M_t^V}{t} = 0$ a.s. Thus, if   $P_{inv}(S^{m-1})$ denotes the set of invariant measures of $v_t$ on $S^{m-1}$ and 
$$
\alpha \coloneqq \inf_{\mu \in P_{inv}(S^{m-1})} \int \mathcal{L}V(v,0) d\mu(v)
$$
is the Lyapunov exponent, 
then Birkhoff's ergodic theorem implies that for most initial conditions
\begin{equation}\label{birkhoff-application}
\liminf_{t \to \infty} \frac{-\log r_t}{t} \geq \alpha \,.
\end{equation}
In particular, if $\alpha > 0$ then $x_t \to 0$ exponentially fast with rate $\alpha$. Of course, determining the positivity of  $\alpha$ may be involved and possible only in special cases.  For example in the deterministic case of $\Sigma = 0$, we may explicitly solve \eqref{linear-SDE} to see that $-\alpha$ is simply the largest real part of eigenvalues of $A$.

Observe that this argument, in particular the conclusion \eqref{birkhoff-application}, crucially depends on the fact that both $\mathcal{L}V$ and the right hand side of $dv_t$ are only functions of $v_t$ as opposed to $(v_t,r_t)$. This does not necessarily hold for the more general system
\begin{equation}\label{nonlinear-SDE}
    dx_t = \hat{A}(x_t)dt + \hat{\Sigma}(x_t) dW_t \,,
\end{equation}  
where $\hat{A}, \hat{\Sigma} : \R^m \to \R^m$ are continuously differentiable functions with $\hat{A}(0) = \hat{\Sigma}(0) = 0$. A natural question is whether \eqref{nonlinear-SDE} has the same  long term behavior as the 
 ``linearized" system \eqref{linear-SDE} with 
 $A = \nabla \hat{A}(0)$ and $\Sigma = \nabla \hat{\Sigma}(0)$. After the  polar change of coordinates as above, we obtain that \eqref{linear-SDE} and \eqref{nonlinear-SDE} have the same behavior when the initial distribution is supported on the set $\{r = 0\}$, and also that $\mathcal{L}V(v,0)$ and $\alpha$ are the same for both systems. Since in this paper we show (see, for example, \Cref{main}) that the stability of $\mcM_0$ is determined by $\alpha$ (determined by behavior on $\mcM_0$ only), we conclude that if $0$ is attractive for the linearized system \eqref{linear-SDE}, meaning $\alpha > 0$, then the nonlinear system \eqref{nonlinear-SDE} is stable as well (compare to \cite[Theorem 7.1]{Khas}). Additionally, our \Cref{robust} immediately implies that $\alpha$ depends continuously on $A,\Sigma$, and thus the stability is preserved if the coefficients of the equation are perturbed slightly.
 
The reasoning presented above is a general principle that extends beyond the case where $\mcM_0$ is a singleton. Although it is not formally stated in our manuscript, our main results imply that to determine whether $\mcM_0$ is stable (meaning extinction occurs), it suffices to linearize the system about $\mcM_0$ and determine whether $\mcM_0$ (or a suitable ``blown up" version $\ncN_0$ after a change of coordinates) is stable for the linearized system.
Indeed, we show that 
the stability of $\ncN_0$ for the nonlinear and linearized systems occurs when the corresponding $H$-exponent $\alpha$ is positive, and since $\alpha$ only depends on the behavior of $X_t$ near the extinction set, both values of $\alpha$ are exactly the same for the linearized and nonlinear systems. This principle of linearization is one of the main advantages of our technique; we only need to analyze $X_t$ and $\mathcal{L}V$ near $\mcM_0$, as opposed to on the entirety of $\mcM$, greatly simplifying our analysis of the important examples discussed in further detail below.

\subsection{Discussion of our Examples}
Since we consider general Markov processes including jump processes and SDEs with Markovian switching, our results have a wide range of applications to both discrete-time and continuous-time Markov chains. Due to the multitude of applications and also because our focus is on the development of universal tools, we decided to present 
  representative examples that  motivated the development of our general theory. Notably, we omit a discussion of chemostat models, but we remark that our results can be used to simplify the proof of the general result \cite[Theorem 2.2]{Chemostat-general} in the case of $\lambda < 0$.

\subsubsection{SIS Model}
The first example is the SIS (susceptible-infected-susceptible) epidemic model on a network with Markovian switching analyzed in  \cite{SIS}.
Each node $i$ of the network can be infected at time $t$
with probability $x_i(t)$ based on the interactions given by the network. 
 The network, the rate of transmission, and the rate of recovery are changing in time according to a Markov process on a finite state space, where each state represents a different environment or regime. In \cite{SIS}, conditions ensuring that $x(t) \to 0$ exponentially fast are established by examining the Lyapunov function $V(x) = -\frac{1}{2}\log{\|x\|^2}$. Specifically, using their assumptions and It\^ o's formula \cite{SIS} proves that $\mathcal{L} V$ is positive and then the rate of change of the quadratic variation of the martingale 
 \begin{equation}\label{V-mg}
    V(x(t)) - V(x(0)) - \int_0^t \mathcal{L}V(x(s)) ds
\end{equation} 
is bounded and thus \eqref{V-mg} satisfies the strong law of large numbers. Unlike our approach, the authors of \cite{SIS} needed to bound $\mathcal{L} V$ for all possible inputs $x \neq 0$, not just for $x$ approaching $0$. Consequently, the assumptions posited in \cite{SIS} to guarantee extinction include an extra term (which they denote by $K$) that is related to the strength of the noise on the entire state space.

The theory developed in the present paper uses simpler analysis, only requiring us to compute $\mathcal{L} V$ as $x$ approaches $0$. This allows us to remove the unnecessary $K$ term in our improved conditions that guarantee extinction. Using the corresponding persistence theory \cite{persistence}, one can show that our extinction conditions are optimal if the network topology remains constant and only the rates of transmission and recovery are switching.

\subsubsection{Lorenz System}
In the second example, we consider the Lorenz system, which is 
a well-studied simplified (deterministic) model of fluid dynamics and a prototypical model for chaos in three variables $(X,Y,Z) \in \R^3$. The stochastic Lorenz system with additive white noise of strength $\hat{\alpha}$ in the $Z$ component was analyzed in \cite{Lorenz}. Here we provide a simplified analysis of the stochastic Lorenz system, focusing on the range of parameters for which the
  $Z$-axis is a global attractor for the deterministic system. In this case, it is shown in \cite{Lorenz} that for small $\hat{\alpha}$, the solutions converge to the $Z$-axis almost surely, while for large $\hat{\alpha}$, persistence occurs.

  The proof in \cite{Lorenz} is divided into multiple steps which connect the persistence or extinction to the sign of a Lyapunov-type exponent $\lambda_{\hat{\alpha}}$ (\cite[Theorem 4.1]{Lorenz}), and then the sign of $\lambda_{\hat{\alpha}}$ is analyzed (\cite[Theorem 5.2]{Lorenz}). To connect the persistence or extinction to the sign of $\lambda_{\hat{\alpha}}$,
  first the known Lyapunov function $V_1$ for the deterministic system is used to control the process at spatial infinity.
  Next, after a cylindrical-like change of coordinates from $(X,Y,Z)$ to $(r,\theta,z)$, where $r = -\infty$ (a cylinder) corresponds to $X=Y=0$ (the $Z$-axis), the authors use the average Lyapunov function $V(r,\theta,z) = -r$ to construct a more traditional Lyapunov function $V_0 : = e^{- \kappa r}(1 - \kappa g_{\hat{\alpha}} + V_2)$, where $V_2$ is similar to $V_1$, the sign of $\kappa$ is the same as the sign of $\lambda_{\hat{\alpha}}$, and $g_{\hat{\alpha}}$ (not given explicitly) is the solution to a PDE related to the linearization of the system about the $Z$-axis. Then they show that $V := V_0 + V_1$ satisfies $\mathcal{L} V \leq K - cV$ for some constants $K,c > 0$ and $V$ blows-up at infinity. If $\lambda_{\hat{\alpha}} > 0$, $V$ also blows up near the $Z$-axis, which provides persistence. If $\lambda_{\hat{\alpha}} < 0$, then $V  \approx 0$ near the $Z$-axis and satisfies $\mathcal{L} V \lesssim - \Gamma{V}$ in a small neighborhood around the $Z$-axis, where $\Gamma{V}$ is the average rate of change of the quadratic variation. By a renewal-type argument and a comparison to Brownian motion with drift, the authors obtain extinction, that is, convergence to the $Z$-axis.

On the other hand, our theory immediately implies \cite[Theorem 4.1]{Lorenz} using only the standard Lyapunov function $V_1$ and the average Lyapunov function $V(r,\theta,z) = -r$, bypassing all of the difficult steps outlined in the previous paragraph. 
We are also able to greatly simplify the arguments made in \cite[Section 5]{Lorenz}, which analyses the behavior of the Lyapunov exponent $\lambda_{\hat{\alpha}}$ for small and large $\hat{\alpha}$. 
Specifically, \Cref{robust} below implies that if we perturb the coefficients of an SDE, then the average Lyapunov exponent $\alpha$ from \eqref{our-exponent} does not change a lot (recall that $P_{inv}(\mcM_0)$ is the set of invariant measures supported on $\mcM_0$). If $\hat{\alpha} = 0$, then $P_{inv}(\mcM_0)$ may have multiple ergodic measures $\mu_0$ which are each supported on a point or a circle, but they can all be easily characterized and thus $\alpha$ can be computed and shown to be positive. However, when $\hat{\alpha} > 0$ then $P_{inv}(\mcM_0)$ has exactly one measure $\mu_{\hat{\alpha}}$, which is supported on a cylinder, and both $\mu_{\hat{\alpha}}$ and $\mu_{\hat{\alpha}} H = -\lambda_{\hat{\alpha}}$
are difficult to analyze or estimate.
Yet, the continuity from \Cref{robust} implies that $\mu_{\hat{\alpha}} H > 0$ for small enough $\hat{\alpha}$, and thus extinction occurs. Since our paper is focused on extinction rather than persistence, we 
do not discuss the other half of \cite[Theorem 5.2]{Lorenz} (large $\hat{\alpha}$ implies persistence), but it would be interesting to see if that analysis can be simplified as well.

\subsubsection{Ecological Models}

Historically, the average Lyapunov function technique was primarily developed for deterministic ecological models by 
mathematical biologists and ecologists \cite{Hofb1, AvgLyap, 71, 32} who recognized the importance of the so-called ``invasion" rates $r_i(\mu)$ associated to each species $i$ and invariant measure $\mu$. Intuitively, $r_i(\mu)$ measures how quickly a 
small population of the $i$ species will grow (invade the environment) when the populations of all the species are distributed according to $\mu$.

Conditions ensuring whether or not a collection of species can coexist in terms of $(r_i(\mu))_{i, \mu}$ were known mainly for special cases until the recent works of \cite{ecologicalContinuous} (for the case of SDEs), \cite{ecologicalDiscrete} (for discrete time systems with a globally attractive compact set), \cite{ecologicalGeneral} (for general discrete time systems and also SDEs), and
\cite{Functional-1, Functional-2} (for stochastic equations with delay). Since the invasion rates are essentially average Lyapunov exponents, average Lyapunov functions are fundamental to their proofs.

However, instead of working with the average Lyapunov function $V$, the arguments in cited papers typically use tools like the log-laplace transform to instead analyze a more traditional Lyapunov function $\hat{V}$ which is constructed using $V$.
The analysis of $\hat{V}$ is quite involved since it came from an average Lyapunov function and thus a sufficient time needs to pass before $\hat{V}$ starts behaving like a traditional Lyapunov function. In other words, instead of investigating $\mathcal{L}\hat{V}$, one has to analyze $\E[\hat{V}(X(t+T)) - \hat{V}(X(t))]$ for all $T$ in some interval, which is typically harder and often involves rather strong conditions.
The theory developed in this manuscript allows us to easily recover many of these results while also weakening their assumptions.

For example, under our weaker assumptions we show that \cite[Theorem 2.4]{ecologicalGeneral} is a consequence of our \Cref{main3} and \cite[Theorem 2.1]{ecologicalGeneral}  follows from the corresponding theory developed in \cite{persistence}. We are also able to recover \cite[Theorem 2.5]{ecologicalGeneral}. 
Since the continuous- and discrete-time ecological models share 
the same main ideas, we decided to provide details only for the latter one. This is partly because an application to SDEs is already illustrated via the noisy Lorenz system above, and partly because the calculations involved in the analysis of the discrete-time models cannot be done explicitly and thus require more work. In \Cref{example-ecological-continuous} we formulate the proper (weaker) assumptions so that the continuous-time analogues of the results mentioned above (studied in \cite{ecologicalContinuous}) hold, and briefly discuss how they can be proven using our techniques.

We emphasize that while almost all of the applications discussed above are direct consequences of the theory developed in this manuscript and require very little additional work to analyze, the proof of our improved version of \cite[Theorem 2.5]{ecologicalGeneral} is quite nontrivial and illustrates the power of using our robustness result \Cref{robust} in conjunction with the change of variables technique from \Cref{change-of-variables}. We construct an average Lyapunov function $V$, but $\mathcal{L}V$ does not extend continuously to $\mcM$ without some change of variables, the choice of which is not obvious. Since
$\mcM_0$ has a ``corner", we consider changes of variables $\pi_p$ for $1 > p > 0$ which approximate $\mcM_0$ by a smooth 
surface which resembles the unit ball in $L^p$. Each $\pi_p$ maps a simplex-like set $\ncN_0$ to $\mcM_0$ differently, and thus each $p$ induces a different Markov process $Y^p$ (which can be thought of as $\pi_p^{-1}(X)$) on $\ncN_0$, a different set of invariant measures $P_{inv}^p(\ncN_0)$ for $Y^p$, and also a different $H_p$, which is defined as the continuous extension of $\mathcal{L}V \circ \pi_p$ to $\ncN_0$. Each of these objects is quite hard to analyze for any individual $p > 0$. However, we show that, as $p \to 0$, $H_p$ converges to a (simpler) function $H_0$ and also the dynamics governing $Y^p$ converge in some sense to those of a simple Markov process $Y^0$, making $\inf_{\mu \in P_{inv}^0(\ncN_0)} \mu  H_0 > 0$ easy to calculate. Using the continuity result \Cref{robust}, we conclude that if $p$ is small enough, then $\inf_{\mu \in P_{inv}^p(\ncN_0)} \mu H_p  > 0$, so that $\pi_p$ gives a change of variables for which \Cref{change-of-variables} can be applied to, allowing us to conclude our improved version of \cite[Theorem 2.5]{ecologicalGeneral}.

\subsection{Organization of Paper}
 In \Cref{notation} we introduce our notation and key assumptions. In \Cref{main-results} we give precise formulations of our main results and techniques. 
 \Cref{general-facts} is devoted to proving fundamental estimates on important martingales related to our Lyapunov functions, some ergodic-type facts about empirical occupation measures, and a result concerning the long-term behavior of a class of discrete-time semimartingales. These results are the backbone of the proof of the crucial \Cref{main2}. In \Cref{feller-stuff} we generalize some results known for $C_0$-Feller processes to the more general class of ``$C_b$-Feller" processes. Here we also discuss the Skorokhod topology and the continuity of relevant functions on these Skorokhod spaces. 
The main theoretical novelties occur in \Cref{main-arguments}, where we prove \Cref{main2}. 
Another main \Cref{main3} is proved in 
\Cref{global}. In \Cref{generator-bs} we include technical results relating the extension of the generator $\mathcal{L}$ from the domains $\Dm(\mcM),\Dm_2(\mcM)$ (\Cref{gen}, \Cref{Gamma}) to the extended domains $\Dme_+(\mcM),\Dme_2(\mcM)$ (\Cref{D+}, \Cref{D2}). 
In \Cref{verify-assumptions} we provide sufficient conditions that verify the bulk of our technical assumptions in special cases: if the Markov process is a switching diffusion or an SDE driven by Brownian motion. Also, we show how our theory can be applied to discrete-time Markov chains.
Finally, the applications of our theory to the SIS model, Lorenz system, and ecological models are given 
in \Cref{examples}.

For the reader primarily interested in applications, we recommend first reading \Cref{notation} and \Cref{main-results} (skipping the proofs), and then for each example in \Cref{examples} first reading the corresponding section of \Cref{verify-assumptions} (skipping the proofs). For example, one should be familiar with the statements in \Cref{discrete-time} before reading \Cref{example-ecological-discrete}.

\section{Notation and Assumptions} \label{notation}
Let $(\mcM,d)$ be a locally compact Polish (complete and separable) metric space, $\mcM_0 \subset \mcM$ be a closed set, and $\inv \subset \mcM_0^c$ be open and dense in $\mcM$ (where $\mcM_0^c$ denotes the complement of $\mcM_0$). The set $\mcM_0$ can be viewed as an ``extinction" set and $\inv$ as the set of possible initial conditions. Without loss of generality we assume $d \leq 1$, otherwise we replace $d$ by $\min(d, 1)$. We endow $\mcM$ with the Borel $\sigma$- algebra. For each $x \in \mcM$, let $X_t^x$ be a homogeneous Markov process defined on $\mcM$ which has cadlag sample paths. By this we mean that there is a filtered probability space $(\Omega, \mathcal{F}, \{\mathcal{F}_t\}_{t \geq 0}, \Prb)$ (which we assume is complete and right-continuous) and a family of $\mcM-$valued random variables $\{X_t^x\}_{x \in \mcM, t \geq 0}$ such that:

\begin{itemize}
    \item $X_0^x = x$ and $t \mapsto X_t^x$ is cadlag (right continuous with left limits) a.s.
    \item $X^x_\cdot$ is adapted to $\{\mathcal{F}_t\}_{t \geq 0}$, meaning $X_t^x$ is $\mathcal{F}_t$ measurable for each $t \geq 0$.
    \item \label{semigroup} For all bounded measurable functions $f:\mcM \to \mathbb{R}$, the map
    $$
    [0,\infty) \times \mcM \ni (t,x) \mapsto \Pp_t f(x) \coloneqq \E[f(X_t^x)] 
    $$ 
    is measurable and for any $s, t \geq 0$ we assume (homogeneity)
    that
    $$
    \Pp_s f(X_t^x) = \E[f(X_{t + s}^x)|\mathcal{F}_t] \,.
    $$    
\end{itemize}

It is standard to prove that $(\mathcal{P}_t)_{t \geq 0}$ defines 
a semigroup: for any $s, t \geq 0$ it holds that $\Pp_{s+t}f = \Pp_s \Pp_tf$. Also, the definition of $\Pp_tf$ makes sense as long as $f: \mcM \to \R$ is measurable and bounded from below (or above), possibly taking on the value $\infty$ (or $-\infty$). 

Our setup is summarized in the following definition.

\begin{deff}\label{quadruple}
For a Polish space $\mcM$, $\mcM_0 \subset \mcM$ closed, $\inv \subset \mcM_0^c$ open and dense, and $\{X_t^x\}_{x \in \mcM, t \geq 0}$ a Markov process on $\mcM$ with cadlag sample paths as defined above, we call $(\mcM, \mcM_0, \inv, \{X_t^x\}_{x \in \mcM, t \geq 0})$ a Markov quadruple.
\end{deff}

For the rest of the section, fix a Markov quadruple $(\mcM, \mcM_0, \inv, \{X_t^x\}_{x \in \mcM, t \geq 0})$.

\subsection{Invariant Sets and the Feller Property }
Next we list two basic assumptions on our process $X_t^x$ and the sets $\mcM_0, \inv$. First, we suppose that if an initial condition is in $\mcM_0$ (respectively $\inv$) then $X_t^x$ belongs to $\mcM_0$ (respectively $\inv$) for all $t \geq 0$. Second, we assume a standard continuity on the law of $X_t^x$ as a function of the initial condition $x$ and time $t$. Such a condition is commonly known as the ``Feller" property.

\begin{deff}
\label{inv-set}
    A measurable set $A \subset \mcM$ is invariant if $x \in A$ implies that almost surely $X_t^x \in A$ for all $t \geq 0$.
\end{deff}

\begin{rem}
    Some authors use the expression ``forward invariant set" instead of our invariant set, 
    since we do not require that $X_t^x \in A$ for some $t > 0$ implies $x \in A$. However, there should not be any confusion, since we do not use backward evolution in the present manuscript. 
\end{rem}

\begin{ass} \label{as1}
    $\mcM_0$ and $\mcM_+$ are invariant.
\end{ass}

Let $C_b(\mcM)$ denote the space of bounded continuous functions on $\mcM$ endowed with the supremum norm $\|f\| = \sup_{x \in \mcM} |f(x)|$. The following continuity assumption is vital for our analysis:

\begin{ass} \label{as2}
    Assume that $X_t^x$ is Feller in the sense that for any $f \in C_b(\mcM)$ and $t \geq 0$, $\Pp_tf \in C_b(\mcM)$, and also $\Pp_tf \to f$ pointwise as $t \downarrow 0$.
\end{ass}

\begin{rem}
    Note that we use ``$C_b$" Feller continuity as opposed to the usual ``$C_0$" Feller continuity (just replace $C_b$ by $C_0$ in \Cref{as2}, where $C_0$ is the space of functions vanishing at infinity) because in many examples (particularly ecological ones) the semigroup is not $C_0$ Feller. See \cite[Remark 1 and Example 1]{persistence} for more details. We also note that the requirement of right-continuity in time ($\Pp_tf \to f$ pointwise as $t \downarrow 0$) is automatically satisfied since our process is assumed to have cadlag paths.
\end{rem}

\begin{rem}
    Under an additional stability assumption which we define below in \Cref{as3}, our definition of Feller is equivalent to the seemingly stronger requirements of \cite[Hypothesis 2]{persistence}. This is a consequence of the quasi-left continuity showed in \Cref{quasi-left} and (a slight modification of the proof of) \Cref{unif-in-t}, which is similar to the proof of strong continuity for $C_0$ Feller semigroups.
\end{rem}

\begin{deff}
    \label{feller-quadruple}
If $(\mcM, \mcM_0, \inv, \{X_t^x\}_{x \in \mcM, t \geq 0})$ is a Markov quadruple such that \Cref{as1} and \Cref{as2} are satisfied, then  $(\mcM, \mcM_0, \inv, \{X_t^x\}_{x \in \mcM, t \geq 0})$ is called a Feller quadruple.
\end{deff}

 \subsection{Two Important Martingales} \label{generator-stuff}
In this section we introduce important definitions that are used in \Cref{lyap}, as well as two martingales that are central to our  analysis.

A central object in the study of Feller processes is the generator $\mathcal{L}$ of the Markov semigroup $\Pp_t$ defined on its domain $\Dm(\mcM) \subset C_b(\mcM)$ as follows:

\begin{deff}
\label{gen}
    Suppose that $\{X_t^x\}_{x \in \mcM, t \geq 0}$ is a Feller process (see \Cref{as2}) with the Markov semigroup $\Pp_s$ (see \Cref{semigroup}). Then 
    the domain $\Dm(\mcM) \subset C_b(\mcM)$ of
    the generator $\mathcal{L}$ is the set of all $f \in C_b(\mcM)$ such that:
    \begin{enumerate}[label=(\roman*)]
        \item\label{first} for all $x \in \mcM$,  $\lim_{s \downarrow 0} \frac{\Pp_sf(x) - f(x)}{s}$ exists and we set it equal to $\mathcal{L}f(x)$.
        \item\label{second} $\mathcal{L}f \in C_b(\mcM)$.
        \item\label{third} $\sup_{s > 0} \Big\|\frac{\Pp_sf - f}{s}\Big\| < \infty$.
    \end{enumerate}
\end{deff}

Heuristically, since $\mathcal{L}f(x)$ is the average rate of change of $\E[f(X_t^x)]$ at $t = 0$, then
\begin{equation} \label{Mtf}
     M_t^f(x) \coloneqq f(X_t^x) - f(x) - \int_0^t \mathcal{L}f(X_s^x)ds   
    \end{equation}
    should be a martingale, and this is proven in \Cref{martingale} below for $f \in \Dm(\mcM)$.
    
    The variance of \eqref{Mtf} is related to the Carre du Champ operator $\Gamma$ defined on $\Dm_2(\mcM) = \{f \in \Dm(\mcM) \mid f^2 \in \Dm(\mcM)\}$ as $\Gamma f = \mathcal{L}f^2 - 2f\mathcal{L}f$. It is a folklore result that $\Gamma$ gives the (predictable) quadratic variation of $M_t^f(x)$ as $\langle M^f(x) \rangle_t = \int_0^t \Gamma f(X_s^x)ds$, or equivalently that
    \begin{equation}\label{Mtf-quad-var}
        (M^f_t(x))^2 - \int_0^t \Gamma f(X_s^x)ds
    \end{equation}
is a martingale. We make this precise in \Cref{def:sqrv} and \Cref{Gamma-is-quad-var} below.

\begin{rem}
Informally, 
in the context of SDEs, $\mathcal{L}f$ (resp. $\Gamma f$) is the
 coefficient of the ``$dt$" part of ``$df$" (resp. ``$(df)^2$"). Then the formula for $\Gamma$ is given by It\^ o's formula: $df^2 = 2fdf + (df)^2$ so that $(df)^2 = df^2 - 2fdf$.
\end{rem}

We wish to apply the operators $\mathcal{L}$ and $\Gamma$ to Lyapunov functions introduced below that are unbounded and possibly defined only on a subset $A$ of $\mcM$. This motivates the definitions of $\Dme_+(A)$ and $\Dme_2(A)$ below, which are general classes of functions for which \eqref{Mtf} and possibly \eqref{Mtf-quad-var} are (local) martingales for suitable functions $\mathcal{L}f, \Gamma f$. In \Cref{generator-bs} and \Cref{verify-assumptions} we show that $\Dme_+(A)$ and $\Dme_2(A)$ are in some sense the closures of $\{f \in \Dm(\mcM) \mid f \geq 0\}$ and $\Dm_2(\mcM)$, respectively, and we  give easily verifiable general conditions for ensuring that functions are in $\Dme_+(A)$ or $\Dme_2(A)$.

\begin{deff}
\label{D+}
    Let $A \subset \mcM$ be an open invariant set (see \Cref{inv-set}). We let $\Dme_+(A)$ be the set of all continuous $f:A \to [0,\infty)$ such that there exists a continuous function $\mathcal{L}f: A \to \R$ such that $M^f_{\cdot}(x)$
   in \eqref{Mtf} is a cadlag local martingale for all $x \in A$.
\end{deff}

\begin{deff}
    \label{D2}
   Let $A \subset \mcM$ be an open invariant set (see \Cref{inv-set}). We let $\Dme_2(A)$ be set of all continuous $f:A \to \R$ such that there exist continuous functions $\mathcal{L}f: A \to \R$ and $\Gamma f: A \to [0,\infty)$ such that $M^f_{\cdot}(x)$
   in \eqref{Mtf} is a cadlag square integrable martingale and the stochastic process in \eqref{Mtf-quad-var} is a martingale for all $x \in A$.
\end{deff}

Next, we formulate a sufficient condition which ensures that \eqref{Mtf} satisfies the strong law for martingales:

\begin{deff} \label{lin-bdd}
    Let $A \subset \mcM$ be an open invariant set and $f \in \Dme_2(A)$. Then we say $f$ has linearly bounded quadratic variation if for all $x \in A$
    it holds that 
    $$
     \sup_{t \geq 1} \frac{1}{t}\int_0^t \Pp_s\Gamma f(x) ds \leq C(x)
     $$ for some continuous function $C: \overline{A} \to \mathbb{R}$, where $\overline{A}$ denotes the closure of $A$.
\end{deff}

In the following lemma we formulate a law of large numbers for the martingales $(M_t^f(x))_{t\geq 0}$. 

\begin{lem}
    \label{strong-law}
    Assume that $A \subset \mcM$ is an open invariant set and $f \in \Dme_2(A)$ has linearly bounded quadratic variation (see \Cref{D2}, \Cref{lin-bdd}). Then for all $x \in A$, $\lim_{t \to \infty} \frac{M_t^f(x)}{t} = 0$ a.s. (recall that $M_t^f(x)$ was defined in \eqref{Mtf}).
\end{lem}

\begin{proof}
The proof closely follows the argument in \cite[page 76]{persistence} and we repeat it here for completeness. 

Since $f$ has linearly bounded quadratic variation, we have 
$\E[(M_t^f(x))^2] \leq tC(x)$ for $t \geq 1$. For any integer $n$ and any $\epsilon > 0$, Doob's inequality for right continuous martingales implies that 
\begin{align*}
\Prb\Big(\sup_{2^n \leq t \leq 2^{n+1}} \frac{|M_t^f(x)|}{t} \geq \epsilon\Big) &\leq \Prb\Big(\sup_{t \leq 2^{n+1}} |M_t^f(x)| \geq 2^n\epsilon\Big) \leq \frac{1}{2^{2n}\epsilon^2}2^{n+1}C(x) \\
&= \frac{2C(x)}{2^n\epsilon^2}
\end{align*}
 and the assertion follows from Borel-Cantelli lemma.
\end{proof}

 \subsection{Lyapunov Functions} \label{lyap}
In this section we provide the assumptions necessary for our main results. In \Cref{as3} we suppose that there is a Lyapunov function which is large near spatial infinity and we use it to prove tightness of appropriate measures and to estimate return times to compact sets. \Cref{as4} is crucial as it gives the existence of an ``average Lyapunov function" that forces $X_t^x$ to (on average) move towards $\mcM_0$ if $X_t^x$ is close to $\mcM_0$.
\Cref{as5} guarantees that the variances of the mentioned Lyapunov functions evaluated at $X_t^x$ have a controlled growth in time.

\begin{deff}
\label{proper}
   For $f:\mcM \to [0,\infty)$, we call $f$ proper if $f$ is continuous and has compact sublevel sets, that is for each $m \in [0, \infty)$ the set
   $\{f \leq m\} \coloneqq \{x: f(x) \leq m\}$ is compact. 
\end{deff}

\begin{ass} \label{as3}
    There are proper maps $W, W': \mcM \to [1,\infty)$ and a constant $K > 0$ such that:
    \begin{enumerate}[label=(\roman*)]
        \item \label{3.1} $W \in \Dme_{2}(\mcM)$ (see \Cref{D2}).
        \item \label{3.2} $\mathcal{L}W \leq K - W'$.
    \end{enumerate}  
\end{ass}

\begin{rem}
    \label{sigma-compact}
       Since $\mcM$ is a locally compact Polish space, if $A \subset \mcM$ is open, then there are compact sets $K_n \subset \mcM$ such that $K_n \subset K_{n+1}^\circ$ (where $K^\circ$ denotes the interior of $K$) and $\cup_n K_n = A$. When $A = \mcM$, the fact that $W$ is proper means that we may take $K_n = \{W \leq n\}$ (the sublevel sets).
\end{rem}

\begin{deff}
\label{meas}
    Let $P(\mcM)$ denote the set of all Borel probability measures on $\mcM$. 
    For $\mu \in P(\mcM)$, $S \subset \mcM$ measurable and $f:S \to \mathbb{R}$ measurable, we use the shorthand
$$
\mu f := \int_S f d\mu \,.
$$
    We endow $P(\mcM)$ with topology of weak convergence, that is, $\mu_n \to \mu$ if for all $f \in C_b(\mcM)$, $\mu_nf \to \mu f$. 
\end{deff}

\begin{deff} \label{inv-meas}
    For $F \subset \mcM$ a closed invariant set (see \Cref{inv-set}), let $P_{inv}(F) \subset P(\mcM)$ denote the set of invariant measures on $F$, that is $\mu \in P_{inv}(F)$ if $\mu(F) = 1$ and $\mu \Pp_tf = \mu f$ for all $f \in C_b(\mcM)$ and $t \geq 0$.
\end{deff}

\begin{rem}\label{sub-inv}
    By the Tietze extension theorem, $\mu \in P_{inv}(F)$ is equivalent to $\mu \in P(F)$ and $\mu \Pp_t f = \mu f$ for all $f \in C_b(F)$ and $t \geq 0$. Thus, if we view $\{X_t^x\}_{x \in F, t \geq 0}$ as a Feller Markov process on the Polish space $F$ there is no ambiguity in writing $P_{inv}(F)$.
\end{rem}

\begin{deff}
    \label{vanish}
    We say a function $f: A \to \R$ (where $A \subset \mcM$) vanishes over a function $g: \mcM \to (0,\infty)$ if there is a sequence of compact sets $(V_n)_{n \geq 1}$ such that $\cup_{n} V_n = \mcM$ and 
    $$
    \lim_{n \to \infty} \sup_{x \in A \setminus V_n} \frac{|f(x)|}{g(x)} \leq 0 \,.
    $$
Note that we have $\leq 0$ instead of $=0$ only because of the convention $\sup \emptyset = -\infty$. If $A \setminus V_n \neq \emptyset$ then the limit is equal to $0$.
\end{deff}

\begin{ass} \label{as4}
For $W'$ as in \Cref{as3}, 
    there is $V \in \Dme_2(\inv)$ such that
    \begin{enumerate}[label=(\roman*)]
        \item \label{4.1} For $x_n \in \inv$, $V(x_n) \to \infty$ implies $x_n \to \mcM_0$ (meaning $d(x_n,\mcM_0) \to 0$).
        \item \label{4.2} $\mathcal{L}V$ vanishes over $W'$.
        \item \label{4.3} $\mathcal{L}V$ extends to a continuous function $H: \mcM \to \mathbb{R}$ and there is a constant $\alpha > 0$ such that $\mu H \geq \alpha$ for all $\mu \in P_{inv}(\mcM_0)$.
    \end{enumerate}
\end{ass}

\begin{rem}\label{W'-vanish}
    In view of \Cref{vanish}, the density of $\inv$ implies that $H$, the continuous extension of $\mathcal{L}V$, also vanishes over $W'$. 
\end{rem}

\begin{rem}
    It is shown below in \Cref{inv-is-compact} that under our assumptions $H$ is $\mu$-integrable for all $\mu \in P_{inv}(\mcM)$, so the condition $\mu H \geq \alpha$ in \Cref{as4} \ref{4.3} makes sense.
\end{rem}

The function $V$ from \Cref{as4} could be called an ``average Lyapunov function." Heuristically, under our continuity assumptions, to assess if $\mcM_0$ is an ``attractor" in some sense, it should be enough to examine the dynamics of $X_t^x$ on $\mcM_0$.  \Cref{as4} \ref{4.3} asserts that on $\mcM_0$, the time-averaged value of $H$ should be at least $\alpha$, so if $X_t^x$ stays near $\mcM_0$ for a long time then $V(X_t^x)$ should grow at least as $t \mapsto \alpha t$, and consequently $V(X_t^x) \to \infty$ as $t \to \infty$. Then, by \Cref{as4} \ref{4.1},  $X_t^x$  approaches $\mcM_0$. To show that $X_t^x$ stays near $\mcM_0$ most of the time is not straightforward 
 and it is one of the major challenges of the present paper.  

The next assumption is to ensure that $V$ and $W$ have linearly bounded quadratic variation (see \Cref{lin-bdd}):

\begin{ass} \label{as5}
Let $W$ and $V$ be respectively as in \Cref{as3} and \Cref{as4}. Assume that 
there are continuous maps $U,U': \mcM \to [0,\infty)$ and a constant $K > 0$ such that:
   \begin{enumerate}[label=(\roman*)]
       \item \label{5.1} $U \in \Dme_+(\mcM)$ (see \Cref{D+}).
       \item \label{5.2} $\mathcal{L}U \leq K - U'$.
       \item \label{5.3} $\Gamma W \leq KU'$.
       \item \label{5.4} $\Gamma V \leq KU'$.
   \end{enumerate}
\end{ass}
\begin{rem}
    Note that we may without loss of generality assume that the constants $K$ 
    in \Cref{as3} and \Cref{as5} are the same, since increasing $K$ keeps the relevant inequalities valid.
\end{rem}

Finally, we list some consequences of our assumptions that will be useful for examples and also needed for the proof of \Cref{robust}.

\begin{lem}\label{inv-is-compact}
Suppose \Cref{as3} and \Cref{as5} \ref{5.1}-\ref{5.3} hold true and let 
$W', K$ be as in \Cref{as3}. Then for any $\mu \in P_{inv}(\mcM)$ 
we have $\mu W' \leq K$, and thus $P_{inv}(\mcM)$ is compact. 
Consequently, if $H: \mcM \to \R$ is a continuous function which 
vanishes over $W'$, then $H$ is $\mu$-integrable for all $\mu \in 
P_{inv}(\mcM)$. In addition, $\mu \mapsto \mu H$ is a continuous 
function on $P_{inv}(\mcM)$ and thus $\inf_{\mu \in P_{inv}
(\mcM_0)}\mu H > 0$ is equivalent to $\mu H > 0$ for all ergodic 
$\mu \in P_{inv}(\mcM_0)$.
\end{lem}

\begin{proof}
    For the proof of $\mu W' \leq K$ we refer an interested reader to 
    \cite[Theorem 2.2ii]{persistence}, and we remark that the argument follows from \Cref{lim-is-inv}, \Cref{occ-is-tight}, Birkhoff's ergodic theorem, and ergodic decomposition theorem. 

    The compactness of $P_{inv}(\mcM)$ is a consequence of the tightness of $P_{inv}(\mcM)$ which follows from Chebyshev inequality since $\mu W' \leq K$ and $W'$ is proper (see \Cref{proper}). By \cite[Proposition 4.15]{BenaimHurth22} (detailed in \Cref{lem:bhcon}), $\mu \mapsto \mu H$ is a continuous function of $\mu \in P_{inv}(\mcM)$, and so $\inf_{\mu \in P_{inv}(\mcM_0)}\mu H > 0$ is equivalent to $\mu H > 0$ for all $\mu \in P_{inv}(\mcM_0)$. It suffices to consider only ergodic $\mu$ since $P_{inv}(\mcM_0)$ is convex with extreme points being ergodic measures. 
\end{proof}

\section{Main Results} \label{main-results}
In this section we introduce the main theorems that are proved in the remainder of the paper. In particular, we claim that under 
\Cref{as1} -- \ref{as5}, 
 on the Markov process $X_t^x$,  the set $\mcM_0$ is an ``attractor" in the following sense. If an initial condition $y$ is close to $\mcM_0$, then with high probability $X_t^y$ approaches $\mcM_0$ as $t \to \infty$ with rate determined by $V$ and $\alpha$. For precise statements see \Cref{main} and \Cref{main2}. Since \Cref{main} is a direct consequence of \Cref{main2} we only prove the latter in \Cref{main-arguments} after proving preliminary lemmas in \Cref{general-facts} and \Cref{feller-stuff}.

The next main result, \Cref{main3},  asserts that if $\mcM_0$ satisfies some accessbility conditions, then almost surely $X_t^y$ approaches $\mcM_0$ as $t \to \infty$ for any $y \in \mcM_+$. 
 The proof that \Cref{main2} is a consequence of \Cref{main3} is given in \Cref{global}.

In many cases, we need to enlarge $\mcM_0$ in order to fully capture the dynamics of $X_t^x$ as $x \to \mcM_0$. For example, if $\mcM = \R^n$ and $\mcM_0$ is a single point, then it is natural to blow up $\mcM_0$ into a sphere $S^{n-1}$ and $\mcM$ into $S^{n-1} \times [0,\infty)$, essentially transforming the problem to polar coordinates. Such transformation is formalized in \Cref{quadruple-map}. The ``blow up'' procedure is used if  $\mathcal{L}V$ from \Cref{as4} cannot be extended continuously to all of $\mcM$ (in particular to $\mcM_0$), but it can be extended continuously to a larger ``blown up" space. Then \Cref{change-of-variables} applies and asserts that the conclusions of \Cref{main}, \Cref{main2}, and \Cref{main3} still hold.

The most crucial of all of our assumptions is \Cref{as4} \ref{4.3}, which states that $\inf_{\mu \in P_{inv}(\mcM_0)} \mu H > 0$. Then, informally,  \Cref{robust} claims that $(H, X_t^x) \mapsto \inf_{\mu \in P_{inv}(\mcM_0)} \mu H$ is (lower semi-)continuous.

Finally, in \Cref{compact} we state that several assumptions of the main results are immediately satisfied if  $\mcM$ is compact.
Let us proceed with the rigorous statements of the main results.

\begin{thm}
\label{main}
Suppose that \Cref{as1} -- \ref{as5} are valid for the Markov quadruple $(\mcM, \mcM_0, \inv, \{X_t^x\}_{x \in \mcM, t \geq 0})$.
    If $x \in \mcM_0$ is such that $\lim_{y \to x} V(y) = \infty$, then
    $$\lim_{y \to x, y \in \inv}\Prb\Big(\liminf_{t \to \infty} \frac{V(X_t^y)}{t} \geq \alpha\Big) = 1.$$
\end{thm}

\begin{thm}\label{main2} 
    Suppose that \Cref{as1} -- \ref{as5} are valid for the Markov quadruple $(\mcM, \mcM_0, \inv, \{X_t^x\}_{x \in \mcM, t \geq 0})$.
    Then for every $M, \delta > 0$ there is $D > 0$ such that for any $y \in \inv \cap \{V \geq D\} \cap \{W \leq M\}$ we have 
    $$
    \Prb\Big(\liminf_{t \to \infty} \frac{V(X_t^y)}{t} \geq \alpha\Big) \geq 1 - \delta.
    $$
\end{thm}

\begin{rem}
    Note that \Cref{main2} implies \Cref{main} by setting $M = W(x) + 1$. Also, by \Cref{as4} \ref{4.1}, $V(X_t^y) \to \infty$ implies $X_t^y \to \mcM_0$ as $t \to \infty$, and therefore $\liminf_{t \to \infty} \frac{V(X_t^y)}{t} \geq \alpha$ could be replaced with $X_t^y \to \mcM_0$ in \Cref{main} and \Cref{main2}. Alternatively, using results in \Cref{empirical-facts} one could replace $\liminf_{t \to \infty} \frac{V(X_t^y)}{t} \geq \alpha$ by the statement that all limit points of the empirical occupation measures $\mu_t^y$ (defined in \Cref{occ-meas}) as $t \to \infty$ lie in $P_{inv}(\mcM_0)$.
\end{rem}

For our next result, we need to review the definition of accessibility. The equivalence of the listed conditions in 
\Cref{accessible}
is proved in \Cref{equal-acc}.

\begin{deff}\label{accessible}
      For $x \in \mcM$ and open $\U \subset \mcM$ we say $\U$ is accessible from $x$ if any of the following equivalent conditions hold:
      \begin{enumerate}
        \item[(1)]     $\int_0^\infty e^{-t}\Prb(X_t^x \in \U)dt > 0$.
        \item[(2)] There exists $t \geq 0$ such that $\Prb(X_t^x \in \U) > 0$.
        \item[(3)] $\Prb(\exists t \geq 0 \text{ such that } X_t^x \in \U) > 0$.
    \end{enumerate}
\end{deff}

\begin{thm}\label{main3}
Suppose that \Cref{as1} -- \ref{as5} are valid for the Markov quadruple $(\mcM, \mcM_0, \inv, \{X_t^x\}_{x \in \mcM, t \geq 0})$.
Assume additionally that $\mcM = \inv \cup \mcM_0$ and that every point $x \in \inv$ satisfies the following accessibility condition:
 \begin{equation}\label{acc}
        \exists M > 0 \text{ such that } \forall D > 0, \{V > D\} \cap \{W < M\} \text{ is accessible from } x \,.
    \end{equation}
    Then $$\Prb\Big(\liminf_{t \to \infty} \frac{V(X_t^x)}{t} \geq \alpha\Big) = 1$$ for all $x \in \inv$.
\end{thm}

\begin{rem}\label{suff-acc}
A sufficient condition for \eqref{acc} to hold for $x \in \inv$ is that there is $y \in \mcM_0$ such that $\lim_{z \to y, z \in \inv} V(z) = \infty$ and $y$ is accessible from $x$ in the sense that for all open sets $\U$ containing $y$, $\U$ is accessible from $x$. Indeed, set $M = W(y) + 1$ and for any $D > 0$ choose $\epsilon > 0$ small enough that $\U \coloneqq \{d(y,\cdot) < \epsilon\}$, the ball of radius $\epsilon$ around $y$, satisfies $\U \cap \inv \subset \{V > D\} \cap \{W < M\}$. We have by assumption that $\U$ is accessible from $x$. Since $x \in \inv$ and $\inv$ is invariant by \Cref{as1}, we have $\Prb(X_t^x \in \U) \leq \Prb(X_t^x \in \{V > D\} \cap \{W < M\})$ for any $t > 0$, and thus $\{V > D\} \cap \{W < M\}$ is also accessible from $x$. 
\end{rem}

\begin{deff}\label{quadruple-map}
    Let $(\ncN, \ncN_0, \ncN_+, \{Y_t^y\}_{y \in \ncN, t \geq 0})$ and $(\mcM, \mcM_0, \inv, \{X_t^x\}_{x \in \mcM, t \geq 0})$
    be Markov quadruples (see \Cref{quadruple}). Then
     $\pi: \ncN \to \mcM$ is called a quadruple map if $\pi$ is a continuous surjection such that:
    \begin{enumerate}
        \item $\ncN_0 = \pi^{-1}(\mcM_0)$ and $\ncN_+ = \pi^{-1}(\inv)$.
        \item $\pi^{-1}(\K)$ is compact for all $\K \subset \mcM$ compact.
        \item For $y_n \in \ncN_+$, $\pi(y_n) \to \mcM_0$ implies $y_n \to \ncN_0$.
        \item For all $y \in \ncN$, $\pi(Y_\cdot^y) = X_\cdot^{\pi(y)}$ almost surely.
    \end{enumerate}
\end{deff}

\begin{rem}\label{pushforward-is-inv}
        It follows from \Cref{quadruple-map} that if $(\ncN, \ncN_0, \ncN_+, \{Y_t^y\}_{y \in \ncN, t \geq 0})$ is a Feller quadruple (\Cref{feller-quadruple}) then so is $(\mcM, \mcM_0, \inv, \{X_t^x\}_{x \in \mcM, t \geq 0})$, but we do not use this fact. In this case, if $\mu \in P_{inv}(\ncN)$ (respectively $P_{inv}(\ncN_0)$) then it standard to prove from \Cref{quadruple-map} that the pushforward measure $\pi^* \mu$ (given by $\pi^* \mu f = \mu f \circ \pi$ for $f \in C_b(\mcM)$) is in $P_{inv}(\mcM)$ (respectively in $P_{inv}(\mcM_0)$).
\end{rem}

\begin{thm}\label{change-of-variables}
    Let $(\ncN, \ncN_0, \ncN_+, \{Y_t^y\}_{y \in \ncN, t \geq 0})$ and $(\mcM, \mcM_0, \inv, \{X_t^x\}_{x \in \mcM, t \geq 0})$ be  Feller quadruples (see \Cref{feller-quadruple}) and $\pi: \ncN \to \mcM$ is a quadruple map.
    Then  \Cref{main2} (and thus its corollaries \Cref{main} and \Cref{main3}) remain valid with \Cref{as4} \ref{4.3} replaced with the assumption that:
    \begin{equation}\label{as4.3-sub}
    \parbox{0.8\textwidth}{
    $\mathcal{L}V \circ \pi$ extends to a continuous function $H: \ncN \to \mathbb{R}$ and there is a constant $\alpha > 0$ such that $\mu H \geq \alpha$ for all $\mu \in P_{inv}(\ncN_0)$.}    
    \end{equation}

\end{thm}

\begin{rem}
    The assumptions of \Cref{change-of-variables} are somewhat redundant, since by \Cref{feller-quadruple}, \Cref{as1} -- \ref{as2} are automatically satisfied.
\end{rem}

\begin{proof}
As an consequence of  \Cref{D+}, \Cref{D2},  and \Cref{quadruple-map} we immediately obtain that for any $f: \mcM \to \R$ it holds that
    \begin{enumerate}
        \item[(1)] If $f \in \Dme_+(\mcM)$, then $f \circ \pi \in \Dme_+(\ncN)$ and we may take $\mathcal{L}(f \circ \pi) = \mathcal{L}f \circ \pi$.
        \item[(2)] If $f \in \Dme_{2}(\mcM)$, then $f \circ \pi \in \Dme_{2}(\ncN)$ and we may take $\mathcal{L}(f \circ \pi) = \mathcal{L}f \circ \pi$, $\Gamma (f \circ \pi) = \Gamma f \circ \pi$.
    \end{enumerate}
    
Also, (1) and (2) hold with $\mcM,\ncN$ replaced with $\inv,\ncN_+$.

Suppose \Cref{as1} -- \ref{as5} are valid for $(\mcM, \mcM_0, \inv, \{X_t^x\}_{x \in \mcM, t \geq 0})$,  where we replace \Cref{as4} \ref{4.3} with \eqref{as4.3-sub}. Then it follows easily from \Cref{feller-quadruple} and \Cref{quadruple-map} that
\Cref{as1} -- \ref{as5} are satisfied for $(\ncN, \ncN_0, \ncN_+, \{Y_t^y\}_{y \in \ncN, t \geq 0})$ with $W,W',U,U'$ from \Cref{as3} and \Cref{as5} being replaced respectively by $W \circ \pi,W' \circ \pi,U \circ \pi,U' \circ \pi,V \circ \pi$.

Thus, by \Cref{main2} for every $M > 0, \delta > 0$ there is a $D > 0$ such that for any $y \in \ncN_+ \cap \{V \circ \pi \geq D\} \cap \{W \circ \pi \leq M\}$ we have 
    $$
    \Prb\Big(\liminf_{t \to \infty} \frac{V(\pi(Y_t^y))}{t} \geq \alpha\Big) \geq 1 - \delta.
    $$
    If $x \in \inv \cap \{V \geq D\} \cap \{W \leq M\}$ then by surjectivity of $\pi$ and $\ncN_+ = \pi^{-1}(\inv)$ there is some $y \in \ncN_+ \cap \{V \circ \pi \geq D\} \cap \{W \circ \pi \leq M\}$ such that $\pi(y) = x$ and so the claim follows by noting that $\pi(Y_\cdot^y) = X_\cdot^x$.
\end{proof}

\begin{thm}\label{robust}
    Let $\Theta$ be a compact metric space and let 
    $\{\{X_{\theta,t}^x\}_{x \in \mcM, t \geq 0}\}_{\theta \in \Theta}$ be a collection of Markov processes. For each $\theta \in \Theta$ denote $P_{inv}^\theta(\mcM)$ the set of invariant probability measures  (see \Cref{inv-meas}) for $X_{\theta,t}^x$. Define the Markov process $Y_t^{(\theta,x)}$ on $\Theta \times \mcM$ by $Y_t^{(\theta,x)} = (\theta, X_{\theta,t}^x)$.
    
    Suppose that $Y_t^{(\theta,x)}$ satisfies \Cref{as2} and that each $\{X_{\theta,t}^x\}_{x \in \mcM, t \geq 0}$ satisfies \Cref{as3} and \Cref{as5} \ref{5.1}-\ref{5.3} with the same $W,W',K,U,U'$ (independent of $\theta$).
    Then if $\theta_n, \theta_\infty \in \Theta$ and $\mu_n \in P^{\theta_n}_{inv}(\mcM)$ are such that $\theta_n \to \theta_\infty$, then there is a subsequence of $\mu_n$ converging to some $\mu \in P_{inv}^{\theta_\infty}(\mcM)$.
    
    Thus, if $H: \Theta \times \mcM \to \R$ is a continuous function which vanishes over $(\theta,x) \mapsto W'(x)$ (see \Cref{vanish}), then for all $\theta_\infty \in \Theta$ we have
    \begin{equation}\label{eq:lgbm}
      \liminf_{\theta \to \theta_\infty} \inf_{\mu \in P^\theta_{inv}(\mcM)} \mu H_\theta \geq \inf_{\mu \in P^{\theta_\infty}_{inv}(\mcM)} \mu H_{\theta_\infty} \,,  
    \end{equation}
    where $H_\theta(x) \coloneqq H(\theta,x)$. In other words, the function $\theta \mapsto \inf_{\mu \in P^\theta_{inv}(\mcM)} \mu H_\theta$ is lower semicontinuous.
\end{thm}

\begin{rem}
    The condition that $\{Y_t^{(\theta,x)}\}_{(\theta,x) \in \Theta \times \mcM}$ satisfies \Cref{as2} is equivalent to the following convergence of the semigroups $\Pp_t^\theta$ of $\{X_{\theta,t}^x\}_{x \in \mcM, t \geq 0}$: if $f \in C_b(\mcM)$ and $\Theta \ni \theta_n \to \theta \in \Theta$, then $\Pp_t^{\theta_n} f \to \Pp_t^{\theta_{\infty}} f$ uniformly on compact subsets of $\mcM$.
\end{rem}

\begin{proof}
    We use $\tilde{W}$, $\tilde{W'}$, $\tilde{U}$, $\tilde{U'}$ to denote the maps defined by precomposing $W$, $W'$, $U$, $U'$ with $(\theta,x) \mapsto x$ so that by the definition $Y_t^{(\theta,x)} = (\theta, X_{\theta,t}^x)$ it follows that \Cref{as3} and \Cref{as5} \ref{5.1}-\ref{5.3} are satisfied for $Y_t^{(\theta,x)}$ with $\tilde{W}$, $\tilde{W'}$, $\tilde{U}$, $\tilde{U'}$ in place of $W$, $W'$, $U$, $U'$.
    Let $\theta_n, \theta_\infty \in \Theta$ and $\mu_n \in P^{\theta_n}_{inv}(\mcM)$ be such that $\theta_n \to \theta_\infty$. Let $P_{inv}(\Theta \times \mcM)$ denote the set of invariant measures for $\{Y_t^{(\theta,x)}\}_{(\theta,x) \in \Theta \times \mcM}$. It readily follows from \Cref{inv-meas} and $Y_t^{(\theta,x)} = (\theta, X_{\theta,t}^x)$ that $\nu_n \coloneqq \delta_{\theta_n} \otimes \mu_n \in P_{inv}(\Theta \times \mcM)$, where $\delta_{\theta_n}$ is the dirac delta measure at $\theta_n$ and $\otimes$ denotes the usual product of measures.
    By \Cref{inv-is-compact} we have
    that there is a subsequence of $\nu_n$ converging to some $\nu \in P(\Theta \times \mcM)$. Without loss of generality we  assume $\nu = \lim_{n \to \infty} \nu_n$ (otherwise pass to a sub-sequence). By Portmanteau theorem, $\nu = \delta_{\theta_\infty} \otimes \mu$ for some $\mu \in P(\mcM)$ such that $\mu = \lim_{n \to \infty} \mu_n$. We conclude the proof by showing that $\mu \in P_{inv}^{\theta_\infty}(\mcM)$, or equivalently that $\nu \in P_{inv}(\Theta \times \mcM)$. Indeed,  if $f \in C_b(\Theta \times \mcM)$ and $t \geq 0$ then $$\nu f = \lim_{n \to \infty} \nu_n f = \lim_{n \to \infty} \nu_n \Pp^Y_t f = \nu \Pp^Y_t f \,,$$ where $\Pp^Y_t$ denotes the semigroup of $\{Y_t^{(\theta,x)}\}_{(\theta,x) \in \Theta \times \mcM}$.

    The claim \eqref{eq:lgbm} is a consequence of the first claim and \Cref{inv-is-compact}.
\end{proof}

\begin{lem}\label{compact}
For compact $\mcM$,  \Cref{as3}, \Cref{as5}, and \Cref{as4} \ref{4.2} are  satisfied if $\sup_{x \in \inv} \Gamma V(x) < \infty$.
\end{lem}

\begin{proof}
    If $\mcM$ is compact then constant functions $g(x) = c > 0$ are proper (as defined in \Cref{proper}) and all functions vanish over $g$ (take $V_n = \mcM$ in \Cref{vanish}). Thus, we may simply take $W,U,W',U' \equiv 1$ and $K = 1 + \sup_{x \in \inv} \Gamma V(x)$.
\end{proof}

\section{General Facts}\label{general-facts}

This section contains statements and proofs of some technical results related to the martingales $M_t^W,M_t^V,M_t^U$ (see \eqref{Mtf}), the empirical occupation measures (see \Cref{occ-meas}), and discrete semimartingales. Recall that functions 
$W, V$, and $U$ were defined in \Cref{as3}--\ref{as5}. In \Cref{martingale-inequalities} we prove estimates on stopping times that are frequently used in \Cref{feller-stuff} and \Cref{main-arguments}. In \Cref{empirical-facts} we show that almost surely the empirical occupation measures are tight and that all of their limit points (with time approaching to infinity) are invariant measures. We also  obtain a sufficient conditions that imply  \Cref{main2} and these conditions are verified below in \Cref{part1}. 
 Finally, in \Cref{semimg} we recall and prove some basic facts about discrete-time (semi-)martingales which serve as the motivation for the proof of \Cref{part1}. In the entire section we fix a Markov quadruple $(\mcM, \mcM_0, \inv, \{X_t^x\}_{x \in \mcM, t \geq 0})$ satisfying \Cref{as1}--\ref{as5}.

\subsection{Optional Stopping Inequalities} \label{martingale-inequalities}

\begin{lem} \label{ineq1}
Let $W, W'$ and $U, U'$ be as in \Cref{as3} and \Cref{as5} respectively. 
    For any $x \in \mcM$ and a stopping time $\tau$ such that $\E[\tau] < \infty$, we have
    $$
    \E\Big[W(X_{\tau}^x) + \int_0^{\tau} W'(X_s^x)ds\Big] 
    \leq W(x) + K\E[\tau].
    $$ 
    In particular, $\Pp_tW + \int_0^t \Pp_sW'ds \leq W + Kt$. The same inequalities hold with $W,W'$ replaced by $U,U'$ respectively.
\end{lem}

\begin{proof}
 Let $M_t \coloneqq M_t^W(x)$ be the local martingale defined in \eqref{Mtf}. Let $(\tau_n)_{n \in \N}$ being a localizing sequence, meaning $\tau_n$ is an increasing sequence of bounded stopping times such that $\tau_n \uparrow \infty$ almost surely and $M_{t \wedge \tau_n}$ is a martingale for all $n$. Then, by the 
optional stopping theorem and the definition \eqref{Mtf}, for each $n \geq 1$ we have 
$$
0 = \E[M_{\tau_n \wedge \tau}] = \E\Big[W(X_{\tau_n \wedge \tau}^x) - W(x) - \int_0^{\tau_n \wedge \tau} \mathcal{L}W(X_s^x)ds\Big].
$$
By \Cref{as3} we have $W \geq 1$ a d $\mathcal{L}W \leq K - W' \leq K$, and therefore
$$
 W(X_{\tau_n \wedge \tau}^x) - W(x) - \int_0^{\tau_n \wedge \tau} \mathcal{L}W(X_s^x)ds \geq -W(x) - K(\tau_n \wedge \tau) \geq -W(x) - K\tau\,.
$$
Since $\E[W(x) + K\tau] < \infty$,  Fatou's lemma and 
$X^x_{\tau_n \wedge \tau} \to X^x_{\tau}$ almost surely as $n \to \infty$ give 
$$
0 \geq \E\Big[W(X_{\tau}^x) - W(x) - \int_0^{\tau} \mathcal{L}W(X_s^x)ds\Big]
\,.
$$ Using $\mathcal{L}W \leq K - W'$ again, we have
$$0 \geq \E\Big[W(X_{\tau}^x) - W(x) -K\tau + \int_0^{\tau} W'(X_s^x)ds\Big]
\,.$$
Adding $W(x) + K\E[\tau] < \infty$ to both sides proves the first claim.

Setting $\tau =t$ gives 
    $$
    \E\Big[W(X_t^x) + \int_0^t W'(X_s^x)ds\Big] \leq W(x) + Kt.
    $$ 
    By definition of $\Pp_t$ we have
    $\E[W(X_t^x)] = \Pp_tW(x)$,  and consequently
    $$
    \E\Big[\int_0^t W'(X_s^x)ds\Big] = \int_0^t \E[W'(X_s^x)]ds = \int_0^t \Pp_sW'(x)ds
    $$ 
    by Tonelli's theorem and the definition of $\Pp_s$, proving the second claim. The proof for $U,U'$ is analogous with \Cref{as3} replaced by \Cref{as5}.
\end{proof}

\begin{cor}
\label{V-W-linearly-bdd}
For any $x \in \mcM$ and any stopping time $\tau$ with $\E[\tau] < \infty$ we have 
    \begin{equation}\label{eq:sqbom}
    \E[M^W_{\tau}(x)^2] = \E\Big[\int_0^{\tau} \Gamma W(X_s^x)ds\Big] \leq K(U(x) + K\E[\tau]) \,,
    \end{equation}
    where the martingale $M_t^f(x)$ was defined in \eqref{Mtf}. Similarly, for any $x \in \inv$ and any stopping time $\tau$ with $\E[\tau] < \infty$ we have 
    $$
    \E[M^V_{\tau}(x)^2] = \E\Big[\int_0^{\tau} \Gamma V(X_s^x)ds\Big] \leq K(U(x) + K\E[\tau]).
    $$
Specifically,  $V$ and $W$ have linearly bounded quadratic variation (see \Cref{lin-bdd}).
\end{cor}
\begin{proof}
We only prove the claim for $W$ as the proof is analogous for $V$. 

Since $W \in \Dme_2(\mcM)$ (see \Cref{as3}),  $M_t^W(x)^2 - \int_0^t \Gamma W(X_s^x) ds$ is a martingale by \Cref{D2}. By the Optional stopping theorem 
$$
\E\Big[(M_{t \wedge \tau}^W(x))^2 - \int_0^{t \wedge \tau} \Gamma W(X_s^x) ds\Big] = 0.
$$ 
By \Cref{as5}, $\Gamma W \leq KU'$ and so \Cref{ineq1} gives \begin{equation} \label{eq:gamma-W-bound}
\begin{aligned}
\E[(M_{t \wedge \tau}^W(x))^2] &= \E\Big[\int_0^{t \wedge \tau} \Gamma W(X_s^x) ds\Big]  \\
&\leq
K\E\Big[\int_0^{t \wedge \tau}  U'(X_s^x) ds\Big] \leq
K(U(x) + K\E[t \wedge \tau]).
\end{aligned}
\end{equation}
Thus, 
$$
\sup_{t \geq 0} \E\Big[(M_{t \wedge \tau}^W(x))^2\Big] \leq K(U(x) + K\E[\tau]) < \infty
$$ 
and by the Martingale convergence theorem, $M_{t \wedge \tau}^W(x) \to M_{\tau}^W(x)$ in $L^2$. Passing $t \to \infty$ in \eqref{eq:gamma-W-bound} implies that 
$$
\E[M^W_{\tau}(x)^2] = \E\Big[\int_0^{\tau} \Gamma W(X_s^x)ds\Big] \leq K(U(x) + K\E[\tau])
$$ 
and \eqref{eq:sqbom} follows.
To show that $W$ has linearly bounded quadratic variation (see \Cref{lin-bdd}), we note that for $x \in \mcM$, by \eqref{eq:gamma-W-bound} with $\tau = t$ and Tonelli's theorem 
$$
\sup_{t \geq 1} \frac{1}{t}\int_0^t \Pp_s\Gamma W(x) ds \leq \sup_{t \geq 1} \frac{K(U(x) + Kt)}{t} \leq KU(x) + K^2 \,,
$$
as desired.
\end{proof}

\begin{cor}
    \label{V-is-integral-of-H}
    If $x \in \inv$ and $\tau$ is a stopping time such that $\E[\tau] < \infty$, then  
    $$
    \E\Big[\int_0^{\tau} |H(X_s^x)|ds\Big] < \infty \qquad \textrm{and} \quad 
    \E[V(X_{\tau}^x) - V(x)] = \E\Big[\int_0^{\tau} H(X_s^x)ds\Big].$$
\end{cor}

\begin{proof}
Recall from \Cref{as4} that $V \in \Dme_2(\inv)$ and $\mathcal{L}V$ is the restriction of $H$ to $\inv$. Hence, by \Cref{D2} and \Cref{V-W-linearly-bdd}, for each $t \geq 0$ we have that
  $$
  M_t \coloneqq M_{\tau \wedge t}^V(x) = V(X_{\tau \wedge t}^x) - V(x) - \int_0^{\tau \wedge t} H(X_s^x)ds
  $$ 
  is a square-integrable martingale with  
 $$
 \sup_{t \geq 0} \E[M_t^2] \leq K(U(x) + K\E[\tau]).
 $$ 

By Martingale convergence theorem, $M_t \to V(X_{\tau}^x) - V(x) - \int_0^{\tau} H(X_s^x)ds$ in $L^2$ as $t \to \infty$ and thus

 \begin{equation}\label{eq:smim}
 0 = \E\Big[V(X_{\tau}^x) - V(x) - \int_0^{\tau} H(X_s^x)ds \Big] \,.    
 \end{equation}
  Since $H$ vanishes over $W'$ (see \Cref{W'-vanish}, \Cref{vanish}), there are constants $A > 0, b > 0$ such that $|H| \leq A + bW'$. Consequently, \Cref{ineq1} gives 
 \begin{equation}\label{eq:puboh}
 \begin{aligned}
      \E\Big[\int_0^{\tau} |H(X_s^x)|ds\Big] &\leq A\E[\tau] + b\E\Big[\int_0^{\tau} W'(X_s^x)ds\Big] \\
 &\leq bW(x) + (A+bK)\E[\tau] < \infty,
 \end{aligned}
 \end{equation}
 which proves the claim after adding $\E[\int_0^{\tau} H(X_s^x)ds]$ to both sides of \eqref{eq:smim}.
\end{proof}

\subsection{Empirical Occupation Measures} \label{empirical-facts}

The empirical occupation measures as defined below in \Cref{occ-meas} capture the behavior of $X_t^x$ averaged over time, and are central to the proof of \Cref{main2}.

\begin{deff}
\label{occ-meas}
    For any $x \in \mcM$ and $t > 0$ let $\mu_t^x$ denote the empirical occupation measure 
    \begin{equation}\label{eq:docm}
    \mu_t^x(\omega) \coloneqq \frac{1}{t}\int_0^t \delta_{X_s^x(\omega)}ds,    
    \end{equation}
    where $\delta_y$ is the Dirac  measure concentrated at $y$, that is, $\delta_y f = f(y)$. Note that $\omega \in \Omega$, and therefore $\mu_t^x$ is a random measure which depends on $\omega$, but often this dependence is suppressed.
\end{deff}

\begin{rem}
    Since the paths $t \mapsto X_t^x$ are cadlag a.s., almost surely all continuous functions on $\mcM$ (in particular $W'$ and $H$) are $\mu_t^x$-integrable for all $t > 0$.
\end{rem}

For what follows, recall \Cref{meas} and  \Cref{inv-meas}.

\begin{lem} \label{lim-is-inv}
   If $x \in \mcM$ (respectively $x \in \mcM_0$), then almost surely the following holds. For every sequence $t_n \to \infty$ such that $\mu_{t_n}^x(\omega)$ converges to some $\mu \in P(\mcM)$, then $\mu \in P_{inv}(\mcM)$ (respectively $\mu \in P_{inv}(\mcM_0)$).\end{lem}
   \begin{proof}
   The claim is proven for $\mcM$ in \cite[Theorem 2.2ii]{persistence}.

For any $x \in \mcM_0$, additionally note that  Portmanteau theorem and \Cref{as1} imply that if $\mu_{t_n}^x \to \mu$, then $\mu(\mcM_0) \geq \limsup_{n \to \infty} \mu_{t_n}^x(\mcM_0) = 1$ a.s.
\end{proof}

\begin{lem} \label{lem:bhcon} \cite[Proposition 4.15]{BenaimHurth22}
    If $f:\mcM \to \mathbb{R}_+$ is proper and $\mu_n \in P(\mcM)$ is such that $\limsup_{n \to \infty} \mu_nf < \infty$, then $(\mu_n)_{n \in \N}$ is tight. If furthermore $g:\mcM \to \mathbb{R}$ is a continuous function which vanishes over $f$, then $\mu_n \to \mu$ implies $\mu_ng \to \mu g$. 
\end{lem}

\begin{cor} \label{occ-is-tight}
For $\mu_t^x$ as in \Cref{occ-meas} we have 
$\limsup_{t \to \infty} \mu_t^x(W') \leq K$ a.s. Thus, almost surely $(\mu_{t_n}^x)_{n \in \N}$ is tight for all $t_n \to \infty$. 
\end{cor}

\begin{proof}
   A similar proof can be found in \cite[Theorem 2.2ii]{persistence} but we include the details here for completeness.
    
    Since $W$ has linearly bounded quadratic variation (\Cref{V-W-linearly-bdd}), $M_t^W(x)$ satisfies the strong law for martingales (\Cref{strong-law}), and thus $\frac{-W(X_t^x)}{t} + \mu_t^x(\mathcal{L}W) \to 0$ a.s. as $t \to \infty$. Moreover, from $W \geq 0$ and $\mathcal{L}W \leq K - W'$ (see \Cref{as3}) it follows that
    \begin{align*}
        0 &\leq \liminf_{t \to \infty} \frac{W(X_t^x)}{t} = \liminf_{t \to \infty} \mu_t^x(\mathcal{L}W) 
        \leq \liminf_{t \to \infty} K - \mu_t^x W'     \,,   
    \end{align*}
    and therefore 
    $$
    \limsup_{t \to \infty} \mu_t^x W' \leq K \,.
    $$
    Hence, since $W'$ is proper, by \Cref{lem:bhcon} almost surely $(\mu_{t_n}^x(\omega))_{n \in \N}$ is tight for all $t_n \to \infty$.
\end{proof}

\begin{lem}
\label{muH-big}
    For all $x \in \inv \cup \mcM_0$ and $\mu_t^x$ defined in \eqref{eq:docm}, $\Prb$-a.s the following are equivalent:
    \begin{enumerate}[label=(\roman*)]
        \item \label{H} $\liminf_{t \to \infty} \mu_t^xH \geq \alpha$.
        \item \label{V} $\liminf_{t \to \infty} \frac{V(X_t^x)}{t} \geq \alpha$ or $x \in \mcM_0$.
        \item \label{occ-limit} If $t_n \to \infty$ with  $\mu_{t_n}^x \to \mu$ for some $\mu \in P(\mcM)$, then $\mu(\mcM_0) = 1$.
    \end{enumerate}
\end{lem}

\begin{proof}
    To show \ref{H} implies \ref{V} we assume $x \in \inv$ and show that a.s.
    $$
    \liminf_{t \to \infty} \mu_t^xH \geq \alpha \quad 
    \textrm{ implies }
    \quad \liminf_{t \to \infty} \frac{V(X_t^x)}{t} \geq \alpha\,.
    $$ By the definition of $M^V_t$ in \eqref{Mtf} and $\mathcal{L}V$ being the restriction of $H$ to $\inv$, we have 
    $$
    \frac{M^V_t(x)}{t} =  \frac{V(X_t^x)}{t} -  \frac{V(x)}{t} - 
    \frac{1}{t} \int_0^t H(X_s^x) ds =  \frac{V(X_t^x)}{t} -  \frac{V(x)}{t} 
    - \mu_t^xH \,.
    $$
    By \Cref{V-W-linearly-bdd}, $V$ has linearly bounded quadratic variation, so by \Cref{strong-law} $\frac{M^V_t(x)}{t} \to 0$ as $t \to \infty$, proving the claim.

    To show \ref{V} implies \ref{occ-limit}, we first note that if $x \in \inv$ and $\liminf_{t \to \infty} \frac{V(X_t^x)}{t} \geq \alpha$ a.s., then by \Cref{as4} \ref{4.1} it follows that $d(X_t^x,\mcM_0) \to 0$ a.s. as $t \to \infty$. Also, if $x \in \mcM_0$, then $d(X_t^x,\mcM_0) = 0$ for each $t \geq 0$ by \Cref{as1} (see \Cref{inv-set}). In either case, the Portmanteau Theorem implies
    $$
    \mu(\mcM_0) = \lim_{\epsilon \downarrow 0} \mu(\{d(\cdot,\mcM_0) \leq \epsilon\}) \geq \lim_{\epsilon \downarrow 0} \limsup_{n \to \infty} \mu_{t_n}^x(\{d(\cdot,\mcM_0) \leq \epsilon\}) = 1 \,.
    $$    
   
    To show \ref{occ-limit} implies \ref{H}, first note that by \Cref{occ-is-tight}, almost surely for all $t_n \to \infty$, the sequence $(\mu_{t_n}^x)_{n \in \N}$ is tight and thus has a further subsequence $\mu_{t_{n_k}}^x$ which converges to some $\mu$. Applying \Cref{lim-is-inv} and \ref{occ-limit}, $\mu \in P_{inv}(\mcM_0)$, so \Cref{as4} \ref{4.3} implies $\mu H \geq \alpha$. Since $H$ vanishes over $W'$ (\Cref{W'-vanish}) and $\limsup_{t \to \infty}  \mu_t^x(W') \leq K$ (\Cref{occ-is-tight}), then $\mu_{t_{n_k}}^xH \to \mu H \geq \alpha$ by \Cref{lem:bhcon}. Thus, $\liminf_{t \to \infty} \mu_t^xH \geq \alpha$. 
\end{proof}

\begin{cor}
    \label{distance-small-makes-V-big}
    For $y \in \inv$, we have a.s. 
    $$
    \lim_{M \to \infty} \limsup_{t \to \infty}d(X_t^y,\mcM_0) \Id_{W(X_t^y) < M} = 0 \quad \textrm{implies} \quad  \liminf_{t \to \infty} \frac{V(X_t^y)}{t} \geq \alpha.
    $$ 
\end{cor}

\begin{proof}
We claim that a.s.  \begin{equation}\label{eq:srem}
\lim_{M \to \infty} \limsup_{t \to \infty}d(X_t^y,\mcM_0)\Id_{W(X_t^y) < M} = 0    
\end{equation}
implies

\begin{equation}\label{eq:eolm}
       \lim_{\epsilon \downarrow 0} \limsup_{t \to \infty} \mu_t^y(\{d(\cdot,\mcM_0) > \epsilon\}) = 0 \,.
    \end{equation}
    
Indeed, suppose \eqref{eq:srem} holds and fix any $\epsilon > 0$. By Chebyshev inequality and 
\Cref{occ-is-tight}, there is a $K_\epsilon = \frac{2K}{\epsilon}$ such that a.s. \begin{equation} \label{W'-big}
    \limsup_{t\to\infty} \mu_t^y(\{W' > K_\epsilon\}) < \epsilon \,.
\end{equation}

 Since $W'$ is proper (\Cref{as3}), the set $\{W' \leq K_\epsilon\}$ is compact, so by continuity of $W$ there is $M > 0$ such that $\{W \geq M\} \subset \{W' > K_\epsilon\}$. Combining this with \eqref{W'-big} gives
 \begin{equation} \label{eq:spcop}
\limsup_{t \to \infty} \mu_t^y(\{W \geq M\}) \leq \limsup_{t \to \infty} \mu_t^y (\{W' > K_\epsilon \}) < \epsilon \,.    
\end{equation}
 Using \eqref{eq:srem}, by making $M$ large enough we may also assume
$$\limsup_{t \to \infty}d(X_t^y,\mcM_0)\Id_{W(X_t^y) < M} < \epsilon \,.$$
In particular, there is a (random) $T > 0$ such that $d(X_t^y,\mcM_0)\Id_{W(X_t^y) < M} < \epsilon$ for all $t \geq T$. Thus, 
\begin{equation} \label{eq:distance-small}
    \limsup_{t \to \infty} \mu_t^y(\{W < M\} \cap \{d(\cdot,\mcM_0) > \epsilon\}) = 0 \,. 
\end{equation} 
Since 
$$
\{d(\cdot,\mcM_0) > \epsilon\} \subset \{W \geq M\} \cup (\{W < M\} \cap \{d(\cdot,\mcM_0) > \epsilon\}),
$$
then \eqref{eq:spcop} and \eqref{eq:distance-small}  show that 
$$
\limsup_{t \to \infty} \mu_t^y(\{d(\cdot,\mcM_0) > \epsilon\}) < \epsilon
$$
and since $\epsilon > 0$ is arbitrary, we have proven the claim that almost surely \eqref{eq:srem} implies \eqref{eq:eolm}.

To finish the proof, we assume \eqref{eq:eolm} holds and show that $\liminf_{t \to \infty} \frac{V(X_t^y)}{t} \geq \alpha$. By \Cref{muH-big}, it suffices to show that for all $t_n \to \infty$, every limit point $\mu$ of $(\mu_{t_n}^y)_{n \in \N}$ satisfies $\mu(\mcM_0) = 1$. Due to \eqref{eq:eolm} and the Portmanteau theorem,

$$\mu(\mcM_0^c) = \lim_{\epsilon \downarrow 0} \mu(\{d(\cdot,\mcM_0) > \epsilon\}) \leq \lim_{\epsilon \downarrow 0} \liminf_{n \to \infty} \mu^y_{t_n}(\{d(\cdot,\mcM_0) > \epsilon\}) = 0 \,,$$
as desired.
\end{proof}

\subsection{Discrete-Time Semimartingales} \label{semimg}

\begin{lem}
    \label{discrete-martingale-fact}
    Let $\{M_n\}_{n \in \N}$ be a (discrete-time) martingale with respect to some filtration $\{\mathcal{G}_n\}_{n \in \N}$ with $M_0 = 0$ and $\E[(M_n - M_{n-1})^2 \mid \mathcal{G}_{n-1}] \leq B$ for some constant $B > 0$. Then for every $b > 0$, $\delta > 0$ there is some $C > 0$ such that $$\Prb(|M_n| \leq bn + C \text{ for all } n \geq 0) > 1 - \delta,$$ and $C$ depends only on $\delta$, $B$, and $b$. 
\end{lem}
\begin{proof}
    The proof is similar to the proof of \Cref{strong-law}. First note that since $\E[M_n \mid \mathcal{G}_{n-1}] = M_{n-1}$,$$\E[(M_n-M_{n-1})^2 \mid \mathcal{G}_{n-1}] = \E[M_n^2 \mid \mathcal{G}_{n-1}] - M_{n-1}^2.$$ 
    Thus $\E[M_n^2 \mid \mathcal{G}_{n-1}] \leq B + M_{n-1}^2$, and by induction $\E[M_n^2] \leq Bn$. By Doob's inequality, for any integer $m \geq 0$
    $$
    \Prb\Big(\max_{2^m \leq k \leq 2^{m+1}} \frac{|M_k|}{k} \geq \frac{b}{2}\Big) \leq \Prb\Big(\max_{k \leq 2^{m+1}} |M_k| \geq 2^m\frac{b}{2}\Big) \leq \frac{4}{2^{2m}b^2}2^{m+1}B = \frac{8B}{2^mb^2} \,,
    $$ 
    and so, for any integer $n \geq 0$, 
    $$
    \Prb\Big(\sup_{2^n \leq k} \frac{|M_k|}{k} \geq b\Big) \leq \sum_{m=n}^{\infty} \Prb\Big(\max_{2^m \leq k \leq 2^{m+1}} \frac{|M_k|}{k} \geq \frac{b}{2}\Big) \leq \frac{64B}{2^nb^2}.
    $$ 
    Choose $n$ large enough so that $\frac{64B}{2^nb^2} < \frac{\delta}{2}$ (note that this $n$ only depends on $\delta,B$ and $b$). On the other hand, for any $C > 0$, applying Doob's inequality again gives 
    $$
    \Prb(\max_{0 \leq k \leq 2^n} |M_k| > C) \leq \frac{B2^n}{C^2}.
    $$ 
    For $C$ large enough, only depending on $\delta,B,$ and $n$, we have $\frac{B2^n}{C^2} < \frac{\delta}{2}$, which proves the claim.
\end{proof}

\begin{cor}
    \label{discrete-semimartingale-fact}
    Let $(Z_n)_{n \in \N}$ be a sequence of random variables adapted to some filtration $\{\mathcal{G}_n\}_{n \in \N}$ such that 
    $$
    Z_0 = 0, \qquad 
    \E[(Z_n - Z_{n-1})^2 \mid \mathcal{G}_{n-1}] \leq B, \qquad \E[Z_n - Z_{n-1} \mid \mathcal{G}_{n-1}] \geq b
    $$ for some constants $B,b > 0$. Then for every $\delta > 0$ there is $C > 0$ depending only on $\delta$, $B$, and $b$ such that 
    $$
    \Prb\Big(Z_n \geq \frac{b}{2}n-C \text{ for all } n \geq 0\Big) > 1 - \delta \,.
    $$ 
\end{cor}
\begin{proof}
   It is standard to show that  
   $$
   M_n \coloneqq \sum_{i = 1}^n Z_i - Z_{i-1} - \E[(Z_i - Z_{i-1}) \mid \mathcal{G}_{i-1}] 
   $$ 
   is a martingale with respect to 
   $\{\mathcal{G}_n\}_{n \geq 0}$. 
   Then, 
   \begin{multline*}
   \E[(M_n - M_{n-1})^2 \mid \mathcal{G}_{n-1}] = 
   \E[(Z_n - Z_{n-1} - \E[(Z_n - Z_{n-1}) \mid \mathcal{G}_{n-1}])^2 \mid \mathcal{G}_{n-1}]
   \\
   =
   \E[(Z_n - Z_{n-1})^2 \mid \mathcal{G}_{n-1}] - \E[(Z_n - Z_{n-1}) \mid \mathcal{G}_{n-1}]^2 \leq B,
    \end{multline*}
    and therefore the assumptions of \Cref{discrete-martingale-fact} are satisfied. Hence, there is some $C > 0$ such that 
    $$
    \Prb\Big(|M_n| \leq \frac{b}{2}n + C \text{ for all } n \geq 0\Big) > 1 - \delta,
    $$ 
    where $C$ depends only on $\delta$, $B$, and $b$. Since $$Z_n = M_n + \sum_{i = 1}^n \E[(Z_i - Z_{i-1}) \mid \mathcal{G}_{i-1}] \geq M_n + bn,$$
    then
   $$\Prb\Big(Z_n \geq \frac{b}{2}n-C \text{ for all } n \geq 0\Big) \geq \Prb\Big(|M_n| \leq \frac{b}{2}n + C \text{ for all } n \geq 0\Big) > 1 - \delta \,,$$
   as desired. 
\end{proof}

\section{Continuous Dependence on Initial Condition} \label{feller-stuff}

In this section we extend results known for $C_0$-Feller processes to our weaker ``$C_b$" definition of Feller (\Cref{as2}) combined with \Cref{as3}. Notably, we show that the strong Markov property holds and the entire path of the process (viewed as a random element of the Skorkhod space $\D_{[0,\infty)}(\mcM)$, defined later) depends continuously on the initial condition. We also recall some facts about the Skorokhod topology and show that certain functions between Skorokhod spaces are continuous. On a first read most of this section (especially the proofs) can be skipped, but the statements of \Cref{strong-Markov} and \Cref{cts-dep} should be understood well before reading \Cref{main-arguments}. We begin with a technical lemma that is not used outside of this section:

\begin{lem} \label{unif-in-t}
    If $\{X_t^x\}_{x \in \mcM, t \geq 0}$ is Feller (see \Cref{as2}) satisfying \Cref{as3} then for any $f \in C_b(\mcM)$ we have that $\Pp_tf \to f$ uniformly on compact sets as $t \downarrow 0$.
\end{lem}

\begin{proof}
Let $K_m = \{W \leq m\}$ and let $A \subset C_b(\mcM)$ be the set of functions $f$ such that $\Pp_tf \to f$ uniformly on each $K_m$ as $t \to 0$.

Let $f \in C_b(\mcM)$ and $\lambda > 0$ be arbitrary. We claim that $A$ contains the function $R_\lambda f$, where $R_\lambda = \int_0^\infty e^{-\lambda s}\Pp_sds$ denotes the resolvent operator. Indeed, by Fubini's theorem $\Pp_tR_{\lambda} = \int_0^{\infty} e^{-\lambda s}\Pp_{s+t} ds$  and since $\Pp_{t}$ is a contraction, we obtain
\begin{align*}
    \|\Pp_tR_{\lambda}f - R_{\lambda}f\| &= \Big\|\int_0^{\infty} e^{-\lambda s}\Pp_{s+t}f ds - \int_0^\infty e^{-\lambda s}\Pp_sfds\Big\| \\
    &\leq (e^{\lambda t} - 1)\Big\|\int_0^\infty e^{-\lambda s}\Pp_sfds\Big\| + e^{\lambda t}\Big\|\int_0^t e^{-\lambda s}\Pp_sfds\Big\| \\
    &\leq 2\|f\| \frac{e^{\lambda t} - 1}{\lambda} \,.
\end{align*}
and the claim follows after passing $t \to 0$.

Since $\lambda R_\lambda = \int_0^\infty e^{-s}\Pp_{s/\lambda}ds$, by  \Cref{as2} and Dominated Convergence Theorem,  $\lambda R_\lambda f \to f$ pointwise as $\lambda \to \infty$. For any integer $m \geq 1$, Riesz-Markov Representation Theorem implies that for any linear functional $F$ on $C_b(K_m)$ there is a finite Borel measure $\mu$ such that $F(f) = \mu f$ for any $f \in C_b(K_m)$. 
Since $\|\lambda R_\lambda f\| \leq \|f\|$, 
the Dominated Convergence Theorem implies $\mu(\lambda R_\lambda f) \to \mu f$, and therefore 
$\lambda R_\lambda f \rightharpoonup f$ in the weak topology on $C_b(K_m)$. By Mazur's lemma (see \cite[Corollary 3.8]{Brezis}), for every $m \in \N$ there are convex combinations $f_n$ of $\lambda R_\lambda f$ such that $f_n \to f$ uniformly on $K_m$. By a diagonal argument, we can assume the convex combinations satisfy $f_n \to f$ uniformly on each $K_m$. Since $A$ is closed under finite convex combinations, $f_n \in A$, and since 
$\|\lambda R_\lambda f\| \leq \|f\|$ we also have $\sup_n\|f_n\| \leq \|f\| < \infty$.

Next we show that our fixed $f\in C_b(\mcM)$ belongs to $A$. Fix $K_m$ and let $\|\cdot\|_m$ denote the sup norm over $K_m$, that is, $\|f\|_m = \sup_{x\in K_m}|f(x)|$.  For any $t>0$, $M > 0$, and $x \in \mcM$ we have 
\begin{multline*}
\Pp_tf_n(x) - \Pp_tf(x) = \E [f_n(X^x_t) - f(X^x_t)] 
\\
= 
\E [(f_n(X^x_t) - f(X^x_t))\Id_{X^x_t \in K_{Mm}}]
+ \E [(f_n(X^x_t) - f(X^x_t))\Id_{X^x_t \not \in K_{Mm}}] \,,
\end{multline*}
and therefore
$$
\|\Pp_tf_n - \Pp_tf\|_m \leq \|f_n - f\|_{Mm} + 
\|\Prb (X_t^x \notin K_{Mm})\|_m(\|f_n\| + \|f\|),
$$ 
where $\Prb(X_t^x \notin K_{Mm})$ is viewed as a function of $x$. By Markov inequality and \Cref{ineq1}, 
$$
\Prb(X_t^x \notin K_{Mm}) = \Prb(W(X_t^x) > Mm) \leq \frac{\Pp_tW(x)}{Mm} \leq \frac{W(x)+Kt}{Mm},
$$ 
and consequently 
$$
\|\Pp_tf_n - \Pp_tf\|_m \leq \|f_n - f\|_{Mm} + \frac{m+Kt}{Mm}(\|f_n\| + \|f\|).
$$
Also, 
$$
\|\Pp_tf - f\|_m \leq \|\Pp_tf_n - \Pp_tf\|_m + \|\Pp_tf_n - f_n\|_m + \|f_n - f\|_m \,.
$$
The middle term converges to $0$ as $t \to 0$ because $f_n \in A$. The last term and $\|f_n - f\|_{Mm}$ converge to $0$ as $n \to \infty$ since $f_n \to f$ uniformly on each $K_m$. By letting $t \to 0$ and then $n \to \infty$, we have for any $M > 0$, $$\lim_{t\to 0} \|\Pp_tf - f\|_m \leq \frac{\sup_n \|f_n\| + \|f\|}{M}.$$
Since $M$ and $m$ were arbitrary, therefore $\Pp_tf \to f$ uniformly on each $K_m$ as $t \to 0$. Consequently, $f \in A$. Since $f \in C_b(\mcM)$ was arbitrary, by our definition of $A$ the claim is proven.
\end{proof}

One important consequence of \Cref{as2} is the strong Markov property. 
In fact, the proof is identical to the one for $C_0$ Feller semigroups. 
In order to state the strong Markov property, we introduce the Skorokhod space $\D_{[0,\infty)}(\mcM)$, which consists of all cadlag functions $f:[0,\infty) \to \mcM$  endowed  with the Skorokhod ($J_1$-)topology. Specifically, $f_n \to f$ if and only if there exists increasing bijections $\lambda_n:[0,\infty) \to [0,\infty)$ such that $f_n \circ \lambda_n \to f$ and $\lambda_n \to \textnormal{Id}$ uniformly on compact subsets of $[0,\infty)$, where $\textnormal{Id}$ denotes the identity map on $[0,\infty)$. It is known that $\D_{[0,\infty)}(\mcM)$ is a Polish space and that the Borel $\sigma$- algebra is generated by the projections $\pi_t:f \mapsto f(t)$ (see \cite[Lemma A5.3]{Kallenberg}). Thus, there is no difference between processes $\{X_t\}_{t \in [0,\infty)}$ with cadlag sample paths and random elements $X$ of $\D_{[0,\infty)}(\mcM)$. In what follows, $X^x$ denotes the random variable given by $\Omega \ni \omega \mapsto X_\cdot^x(\omega) \in \D_{[0,\infty)}(\mcM)$.

\begin{lem} \label{strong-Markov}
If $\{X_t^x\}_{x \in \mcM, t \geq 0}$ is Feller (see \Cref{as2}), then it satisfies the strong Markov property, meaning that if $\phi: \D_{[0,\infty)}(\mcM)  \to [0,\infty)$ is a measurable function and $\tau$ is a stopping time, then $\tilde{\phi} : \mcM \to [0,\infty)$ given by
    $$\tilde{\phi}(y) \coloneqq \E[\phi(X^y)]$$ is measurable 
and for any $x \in \mcM$ we have
    $$
    \E[\Id_{\tau < \infty}\phi(X_{\tau+\cdot}^x) \mid \mathcal{F}_{\tau}] = \Id_{\tau < \infty}\tilde{\phi}(X_{\tau}^x).
    $$
\end{lem}

\begin{proof}
The proof is the same as \cite[Theorems 6.16 and 6.17]{LeGall}. We remark that our notation differs from \cite{LeGall}. For example, we use $\mcM$ instead of $E$, $X_t^x$ instead of $Y_t$, $\Pp_t$ instead of $Q_t$, and $C_b(\mcM)$ instead of $C_0(E)$. Note that $C_b(\mcM)$ differs from $C_0(E)$, but the proof in \cite{LeGall} is unchanged if $C_0(E)$ is replaced by $C_b(\mcM)$ as it only uses continuity of functions in $C_0(E)$, but not the vanishing at infinity.
The book \cite{LeGall} also uses a stronger definition of Feller, but in the proofs of Theorems 6.16 and 6.17, it only uses that $Q_t: C_0(E) \to C_0(E)$, which in our notation translates to $\Pp_t : C_b(\mcM) \to C_b(\mcM)$ (see \Cref{as2}).
\end{proof}

\begin{rem}
For example, if $\tau$ is a finite stopping time, $s > 0$, and $f \in C_b(\mcM)$, then applying the strong Markov property detailed in \Cref{strong-Markov} with $\phi(x(\cdot)) := x(s)$ we have 
\begin{equation}\label{rem:gmp}
    \Pp_s f(X_\tau^x) = \E[f(X_{\tau + s}^x)|\mathcal{F}_\tau] \,.
\end{equation}
\end{rem}

The Feller property (\Cref{as2}) tells us that if $x_n \to x$, then $X_t^{x_n} \to X_t^x$ (in distribution). In other words, the law of the process at time $t$ depends continuously on the initial position. Using \Cref{as3}, we can actually show that the entire path of the process (viewed as a random element of $\D_{[0,\infty)}(\mcM)$) depends continuously on the initial condition.

First, we recall an Arzela-Ascoli-type result for weak convergence on $\D_{[0,\infty)}(\mcM)$:

\begin{lem} \label{AA}
    If $Y^n,Y$ are $\D_{[0,\infty)}(\mcM)$-valued random variables, then $Y^n \to Y$ in distribution if both statements hold:
    \begin{enumerate}[label=(\roman*)]
        \item \label{fd} For any $0 \leq t_1 < \ldots < t_m$, $(Y^n(t_1),\ldots,Y^n(t_m)) \to (Y(t_1),\ldots,Y(t_m))$  in distribution. 
        \item \label{st} For any $t > 0$, any sequence $(\tau_n)_{n \geq 1}$ of stopping times with $\tau_n \leq t$, and any sequence $(h_n)_{n \geq 1}$ of positive constants such that $h_n \to 0$, we have $\E[d(Y_{\tau_n}^n,Y_{\tau_n+h_n}^n)] \to 0$. 
\end{enumerate}

\begin{rem}
    By convergence in distribution we mean that the laws of the random variables converge weakly, and in \ref{fd} we view each $(Y^n(t_1),\ldots,Y^n(t_m))$ as a $\mcM^m$-valued random variable. In \ref{st}, we view each $Y^n$ as a collection of $\mcM$-valued random variables $Y^n_t \coloneqq Y^n(t)$ and $\tau_n$ is a stopping time with respect to the filtration generated by $\{Y^n_t\}_{t \geq 0}$.

\end{rem}

\end{lem}
\begin{proof}
    Can be found in \cite[Theorem 23.9i and Theorem 23.11]{Kallenberg}.
\end{proof}

The proof of the following lemma is inspired by \cite[Theorem 17.25]{Kallenberg}.

\begin{lem} \label{cts-dep}
   Suppose $\{X_t^x\}_{x \in \mcM, t \geq 0}$ is Feller (see \Cref{as2}) and satisfies \Cref{as3}. If $x_n \to x$ then $X^{x_n} \to X^x$ in distribution. (Recall we use $X^x$ to denote the random variable given by $\Omega \ni \omega \mapsto X_\cdot^x(\omega) \in \D_{[0,\infty)}(\mcM)$.)
\end{lem}
\begin{proof}
    It suffices to show conditions \ref{fd} and \ref{st} of \Cref{AA}. 

    \textit{Proof of \ref{fd}.} The proof is by induction on $m$. If $m = 1$, then for any $f \in C_b(\mcM)$, \Cref{as2} implies 
     $\Pp_{t_1}f \in C_b(\mcM)$, and therefore 
     $$
     \E[f(X^{x_n}(t_1))] = \Pp_{t_1}f(x_n) \to \Pp_{t_1}f(x) = \E[f(X^x(t_1))].
     $$ 
     Suppose the claim holds for $m-1$ and fix $0 \leq t_1 < \ldots < t_m$ and let $f \in C_b(\mcM^m)$ be Lipschitz. Recall that by Portmanteau theorem it suffices to consider only Lipschitz functions $f$ to show weak convergence of measures, or equivalently the convergence in distribution of the corresponding random variables. 
     Let $\phi: \D_{[0,\infty)}(\mcM) \to \R$ be given by $$\phi(x(\cdot)) = f(x(0),x(t_2 - t_1),\dots,x(t_m - t_1))$$ 
     and $\tilde{\phi}: \mcM \to \R$ be as defined in \Cref{strong-Markov}:
     \begin{equation}\label{phi-tilde0}
         \tilde{\phi}(y) \coloneqq \E[\phi(X^y)] = \E[f(y,X^y(t_2 - t_1),\dots,X^y(t_m - t_1))] \,.
     \end{equation}
     Applying the (strong) Markov property (\Cref{strong-Markov}) with $\tau = t_1$, for any $x \in \mcM$ we have
     \begin{equation}\label{f-m}
         \E[f(X^x(t_1),X^x(t_2),\dots,X^x(t_m))] = \E[\tilde{\phi}(X^x(t_1))] = \Pp_{t_1}\tilde{\phi}(x) \,.
     \end{equation}

     Next, we show that the left hand side of \eqref{f-m} is continuous in $x$, which by \Cref{as2} follows once we prove that $\tilde{\phi} \in C_b(\mcM)$. By our inductive hypothesis, 
     $$
     \mcM \ni y \mapsto (X^y(t_2 - t_1),\dots,X^y(t_m - t_1)) \in P(\mcM^{m-1})
     $$ 
     is continuous (again we conflate random variables with their laws).
     
     Since $f \in C_b(\mcM^m)$ is Lipschitz, 
     $$
     \mcM \ni y \mapsto f_y \in C_b(\mcM^{m-1}) \quad \text{where} \quad f_y(x_1,\dots,x_{m-1}) \coloneqq f(y,x_1,\dots,x_{m-1})
     $$ 
     is continuous. Since 
     $$
     P(\mcM^{m-1}) \times C_b(\mcM^{m-1}) \ni (\mu,f) \mapsto \mu f \in \R
     $$ 
     is continuous, by \eqref{phi-tilde0} $\tilde{\phi}$ is continuous, as desired.

   \textit{Proof of \ref{st}} Let $t > 0$, $(\tau_n)_{n \geq 1}$, and $(h_n)_{n \geq 1}$ be as in \Cref{AA} \ref{st}. Define $\phi_n: \D_{[0,\infty)}(\mcM) \to [0,1]$ as $\phi_n(Y) = d(Y(0),Y(h_n))$ and $\tilde{\phi}_n(y) = \E[\phi_n(X^y)]$. By the strong Markov property (\Cref{strong-Markov}), 
    \begin{equation}\label{nu-n}
    \E[d(X_{\tau_n}^{x_n},X_{\tau_n+h_n}^{x_n})] = \E[\tilde{\phi}_n(X_{\tau_n}^{x_n})] = \nu_n\tilde{\phi}_n,
    \end{equation}
    where $\nu_n$ denotes the law of $X_{\tau_n}^{x_n}$. Since $\tau_n \leq t$,  \Cref{ineq1} implies
    $$
    \nu_n W = \E[W(X_{\tau_n}^{x_n})] \leq W(x_n) + Kt.
    $$ 
    Since $x_n \to x$, $\sup_n W(x_n) < \infty$, and therefore $\sup_n \nu_n W < \infty$. This implies
    \begin{equation}\label{nu-n-of-Km}
        \lim_{m \to \infty} \sup_n \nu_n(K_m) = 0 \,.
    \end{equation}
    
    Note that by the Markov property, 
    \begin{equation}\label{phi-tilde}
        \tilde{\phi}_n(x) = \E[d(x,X_{h_n}^x)] = \E[d_x(X_{h_n}^x)] = \Pp_{h_n}d_x(x) \,,
    \end{equation}
    
    where $d_x:\mcM \to [0,1]$ given by $d_x(y) = d(x,y)$ is continuous. 
    
    Let $K_m = \{W \leq m\}$ and $\|\cdot \|_m$ be the supremum norm over $K_m$ as in the proof of \Cref{unif-in-t}. Then for any $m$ and any $x,y \in K_m$, 
    we obtain 
    \begin{align*}
        |\tilde{\phi}_n(x) - \tilde{\phi}_n(y)| &=
        |\Pp_{h_n}d_x(x) - \Pp_{h_n}d_y(y)|
        \\
        &\leq \|\Pp_{h_n}d_x - \Pp_{h_n}d_y\| + |\Pp_{h_n}d_x(x) - \Pp_{h_n}d_x(y)| \\
        &\leq d(x,y) + |\Pp_{h_n}d_x(x) - \Pp_{h_n}d_x(y)| 
    \end{align*}
    and by the Feller property (see \Cref{as2}) we have that $\Pp_{h_n}d_x$, and thus $\tilde{\phi}_n$, is continuous. 
In addition, by using $d_x(x) = 0$ it holds that 
\begin{align*}
    |\Pp_{h_n}d_x(x) - \Pp_{h_n}d_x(y)| &= 
    |\Pp_{h_n}d_x(x)  - d_x(x) - d_x(y) + d_x(y) - \Pp_{h_n}d_x(y)| 
    \\
    &\leq d(x,y) + 2\|\Pp_{h_n}d_x - d_x\|_m \,,
\end{align*}
and consequently
        \begin{align}\label{eq:sres}
        |\tilde{\phi}_n(x) - \tilde{\phi}_n(y)| 
        \leq 2d(x,y) + 2\|\Pp_{h_n}d_x - d_x\|_m \,.
    \end{align}

 \Cref{unif-in-t} implies $\|\Pp_{h_n}d_x - d_x\|_m \to 0$ as $n \to \infty$, and thus by \eqref{eq:sres} and continuity of each $\tilde{\phi}_n$ we have that $\{\tilde{\phi}_n\}_{n \in \mathbb{N}}$ is equicontinuous on $K_m$. 
By \Cref{as2} and $h_n \to 0$, for each $x \in \mcM$, $\Pp_{h_n}d_x \to d_x$ pointwise, which by \eqref{phi-tilde} means that $\tilde{\phi}_n \to 0$ pointwise since $d_x(x) = 0$. Then Arzela-Ascoli theorem yields that $\tilde{\phi}_n \to 0$ uniformly on each $K_m$. Since $\sup_n\|\tilde{\phi}_n\| \leq 1$, then by \eqref{nu-n} and \eqref{nu-n-of-Km}, $\E[d(X_{\tau_n}^{x_n},X_{\tau_n+h_n}^{x_n})] = \nu_n\tilde{\phi}_n \to 0$, which proves \ref{st}.
\end{proof}

The proof of the next lemma is standard for the more familiar topology of uniform convergence on compact sets, but since we work with the Skorokhod topology we provide  more details.

\begin{lem} \label{sk-cts}
    The following functions are continuous:
    \begin{enumerate}
        \item[(i)] $x(\cdot) \mapsto g(x(\cdot))$ from $\D_{[0,\infty)}(\mcM)$ to $\D_{[0,\infty)}(\R)$, where $g: \mcM \to \mathbb{R}$ is continuous.
        \item[(ii)] $x(\cdot) \mapsto \int_0^\cdot x(s)ds$ from $\D_{[0,\infty)}(\R)$ to $\D_{[0,\infty)}(\R)$.
        \item[(iii)] $x(\cdot) \mapsto \inf_{0 \leq s \leq \cdot} x(s)$  from $\D_{[0,\infty)}(\R)$ to $\D_{[0,\infty)}(\R)$.
        \item[(iv)] $x(\cdot) \mapsto x(T)$ from $\{x \in \D_{[0,\infty)}(\R) \mid x \text{ continuous}\} \eqqcolon C_{[0,\infty)}(\R) \to \R$, where $T$ is a constant.
    \end{enumerate}
\end{lem}

\begin{proof}
The claim (iv) holds true, since  the subspace topology of $C_{[0,\infty)}(\R) \subset \D_{[0,\infty)}(\R)$ coincides with the topology of uniform convergence on compact sets (see for example \cite[Theorem 23.9iii]{Kallenberg}). 

Next, we  prove (i) -- (iii). 
In all three cases the corresponding map denoted $F:\D_{[0,\infty)}(A) \to \D_{[0,\infty)}(\R)$ (where $A = \mcM$ or $\R$) is continuous if we endow $\D_{[0,\infty)}(A)$ and $\D_{[0,\infty)}(\R)$ with the topology of uniform convergence on compact sets.
Thus, by the definition of the Skorokhod topology it suffices to show that if $f_n \in \D_{[0,\infty)}(A)$ and there exists an increasing bijection $\lambda_n:[0,\infty) \to [0,\infty)$ such that $f_n \circ \lambda_n \to f$ and $\lambda_n \to \textnormal{Id}$ uniformly on compact subsets of $[0,\infty)$,  then 
\begin{equation}\label{eq:ucst}
F(f_n) \circ \lambda_n - F(f_n \circ \lambda_n) \to 0    
\end{equation}
 uniformly on compact subsets of $[0,\infty)$ (we already noted that $F(f_n \circ \lambda_n) \to F(f)$).

For (i) and (iii),  \eqref{eq:ucst} is immediate since respectively $(g \circ f_n) \circ \lambda_n = g \circ (f_n \circ \lambda_n)$, and $\inf_{0 \leq s \leq \lambda_n(t)} f_n(s) = \inf_{0 \leq s \leq t} f_n(\lambda_n(s))$ for each $t$.

For (ii) we need to show that for all $T > 0$ 
\begin{equation}\label{eq:icst}
\int_0^{\lambda_n(t)} f_n(s)ds - \int_0^t f_n(\lambda_n(s))ds \to 0
\end{equation}
uniformly for $t \in [0,T]$. Since $(f_n)_{n \geq 1}$ converges in $\D_{[0,\infty)}(\R)$, $(f_n)_{n \geq 1}$ is relatively compact in $\D_{[0,\infty)}(\R)$ and thus by \cite[Theorem A5.4]{Kallenberg} we have
\begin{equation}\label{eq:ksct}
\lim_{h \downarrow 0} \sup_n \tilde{w}_{T+1}(f_n,h) = 0 \,,    
\end{equation}
where for any $g:[0,T+1) \to \mcM$ and $h > 0$
$$
\tilde{w}_{T+1}(g,h) = \inf_{(I_k)} \max_k \sup_{r, s \in I_k} 
|g(r) - g(s)|
$$
and the infimum is taken over all partitions of the interval 
$[0,T+1)$ into sub-intervals $I_k = [u_k,v_k)$ with $v_k-u_k \geq h$ when 
$v_k < T+1$.

To prove \eqref{eq:icst} we 
fix any $T > 0$ and denote 
$$
C:= 1 + \sup_{t \in [0,T+1]} |f(t)|
$$
For any fixed $\epsilon > 0$,  by \eqref{eq:ksct} there is a finite partition of $[0, T+1)$ into $M = M(\epsilon)$ intervals $I_k = [u_k,v_k)$ such that $\sup_{s, t \in I_k} |f_n(s) - f_n(t)| < \epsilon$ for any $k \leq M$, $n \geq 1$. 
Fix $N \geq 1$ such that for any $n \geq N$ we have 
$$\sup_{t \in [0,T]}|\lambda_n(t) - t| < \frac{\epsilon}{CM} \wedge 1 \quad \text{ and } \quad  C \geq \sup_{t \in [0,T]} |f_n(\lambda_n(t))| \,.$$

Then for any $n \geq N$ and $t \in [0,T]$
\begin{align*}
   J_1(t) :=  \Big| \int_0^{\lambda_n(t)} f_n(s)ds - \int_0^t f_n(s)ds \Big| \leq C 
   \sup_{t \in [0, T]}|\lambda_n(t) - t| \leq \frac{\epsilon}{M} 
\end{align*}
and 
\begin{align*}
    Q_k &:= 
    \int_{u_k}^{v_k} |f_n(s) - f_n(\lambda_n(s))|ds \leq  \int_{u_k + \frac{\epsilon}{CM}}^{v_k - \frac{\epsilon}{CM}} |f_n(s) - f_n(\lambda_n(s))| ds \\
    &+ \int_{u_k}^{u_k + \frac{\epsilon}{CM}} |f_n(s) - f_n(\lambda_n(s))| ds
    + \int_{v_k - \frac{\epsilon}{CM}}^{v_k} |f_n(s) - f_n(\lambda_n(s))| ds \\
    &\leq 
    \epsilon (v_k - u_k) + 2C\frac{2\epsilon}{CM}
    = \epsilon (v_k - u_k) + \frac{4\epsilon}{M}
    \,,
\end{align*}
where in the last inequality we used that $s, \lambda_n(s) \in [u_k, v_k)$ if $s \in [u_k + \frac{\epsilon}{CM}, v_k - \frac{\epsilon}{CM})$. The desired result \eqref{eq:icst} follows once we observe that
$$
\Big|\int_0^{\lambda_n(t)} f_n(s)ds - \int_0^t f_n(\lambda_n(s))ds\Big| 
\leq J_1 + \sum_{k = 1}^M Q_k \leq \epsilon(T+1) + 4\epsilon + \frac{\epsilon}{M} \,. 
$$
\end{proof}

 \section{Proof of Main Theorem} \label{main-arguments}
 In this entire section, we fix a Markov quadruple $(\mcM, \mcM_0, \inv, \{X_t^x\}_{x \in \mcM, t \geq 0})$ (see \Cref{quadruple}) satisfying  \Cref{as1} -- \ref{as5}.
We aim to provide a proof of \Cref{main2} which is split into two parts. In the first step  we show that for any compact set $\K$, if  $V(x)$ is big enough for some $x \in \inv$, then, with high probability, 
$X_t^x$ is close to $\mcM_0$ whenever $t$ is such that $X_t^x \in \K$. 
 Furthermore, we establish that there is a sequence of times $t_n \to \infty$ such that $V(X_{t_n}^x) \to \infty$ as $n \to \infty$ and $X_{t_n}^x \in \K$
for all $n$. Specifically, we prove the following theorem.

\begin{thm} \label{part1}
There is an $N > 0$ such that for all $\delta > 0, M > N$ there is  $D > 0$ such that for all $y \in \inv \cap \{V > D\} \cap \{W \leq M\}$, we have
$$
\Prb\Big(\sup_{t \geq 0}d(X_t^y,\mcM_0)\Id_{W(X_t^y) < M} \leq \frac{1}{M} \text{ and } \limsup_{t \to \infty, W(X_t^y) \leq M} V(X_t^y) = \infty\Big) \geq 1 - \delta\,,
$$
where $\limsup_{t \to \infty, W(X_t^y) \leq M} V(X_t^y) = \infty$ means that there is some sequence of times $t_n \to \infty$ such that $W(X_{t_n}^y) \leq M$ and $V(X_{t_n}^y) \to \infty$.
\end{thm}

In the second step we show that \Cref{part1} implies \Cref{main2}, which amounts to showing that if $V(x)$ is big enough, then almost surely all limit points of the sequence of measures $\mu_t^x$ (see \Cref{occ-meas}) are supported on $\mcM_0$.

To prove \Cref{part1} we sample $X_t^x$ at special stopping times that are 
introduced and analyzed 
in \Cref{good-stopping-time}, \Cref{return-time}, and \Cref{bounding-the-variance}. Then, we obtain a discrete-time Markov chain which, after composition with $V$, induces 
a discrete-time semimartingale satisfying the assumptions of \Cref{discrete-semimartingale-fact}. We make this argument precise in \Cref{proof-of-part1} to prove \Cref{part1} and then finally prove \Cref{main2} in \Cref{proof-of-main2}.

\subsection{Good Stopping Time} \label{good-stopping-time}
For this section,  recall \Cref{as4}. We start by showing that as long as $y$ is close enough to $\mcM_0$, there is some bounded stopping time $\tau_{y,T}$ (see \Cref{st1} below) for which $V(X_{\tau_{y,T}}^y) - V(y)$ is (on average) large. The idea is to 
estimate $V(X_t^y) - V(y)$ by bounding
$\int_0^t H(X^y_s)ds$. Since we have information about the behavior of $\int_0^t H(X^y_s)ds$ when $y \in \mcM_0$, we can use the continuity with respect to initial conditions shown in \Cref{feller-stuff} to infer similar behavior when $y$ is close to $\mcM_0$. 

\begin{lem} \label{H-big}
    For every $x \in \mcM_0$,  $\delta > 0$, and $S > 0$ there is $T \geq S$ and $\epsilon > 0$ such that for all $y \in \mcM$ with $d(y,x) < \epsilon$, 
    $$
    P_{y,T} \coloneqq \Prb\Big(\int_0^{T} H(X_s^y) ds > \frac{\alpha}{2}T \text{ and } \inf_{0 \leq t' \leq T} \int_0^{t'} H(X_s^y) ds > -\frac{\alpha}{2}T\Big) > 1 - \delta \,.
    $$
\end{lem}

\begin{proof}
Fix $x \in \mcM_0$. Then \Cref{muH-big}\ref{V} is satisfied, and therefore 
\Cref{muH-big}\ref{H} holds true, that is, $\liminf_{t \to \infty} \mu_t^x(H) \geq \alpha$ a.s. for $\mu_t^x$ as in \Cref{occ-meas}. 
 Hence, almost surely there is a finite (random) $\mathcal{T}$ such that $\mu_t^x H = \frac{1}{t}\int_0^t H(X_s^x) ds > \frac{\alpha}{2}$ for all $t \geq \mathcal{T}$. By making $\mathcal{T}$ larger if necessary, we may suppose $\inf_{0 \leq t' \leq t} \int_0^{t'} H(X_s^y) ds > -\frac{\alpha}{2}t$ for all $t \geq \mathcal{T}$. It follows that $\liminf_{T \to \infty} P_{x,T} = 1$, so there is a (deterministic) $T > S$ such that $P_{x, T} \geq 1 - \frac{\delta}{2}$.

 Define the map $F_T : \D_{[0,\infty)}(\mcM) \to \R$ which is the composition of the  following maps
$$
\D_{[0,\infty)}(\mcM) \ni x(\cdot) \mapsto H(x(\cdot)) \mapsto \int_0^\cdot H(x(s))ds \mapsto \inf_{0 \leq t \leq T} \int_0^t H(x(s))ds \,.
$$ 
Then, $F$ is continuous by \Cref{sk-cts}, and similarly 
$G_T:\D_{[0,\infty)}(\mcM) \to \R$ given by 
$$
G_T(x(\cdot)) := \int_0^{T} H(x(s)) ds
$$
is continuous. Thus, $Q := \big\{Y \in \D_{[0,\infty)}(\mcM): G_T(Y) > \frac{\alpha}{2}T \text{ and } F_T(Y) > -\frac{\alpha}{2}T \big\}$ is open and  
$$
P_{y,T} = \Prb(X^y_\cdot \in Q) \,.
$$
Then \Cref{cts-dep} and the Portmanteau theorem give us
$$
\liminf_{y \to x} P_{y,T} \geq P_{x,T} \geq 1 - \frac{\delta}{2} \,.
$$
Thus, by choosing $\epsilon > 0$ small enough, $d(y,x) < \epsilon $ 
 implies $P_{y,T} > 1 - \delta$.
\end{proof}

\begin{lem} \label{V-close-to-H}
For every $\delta > 0$ and compact $\K \subset \mcM$ there is $S > 0$ such that for all $x \in \K \cap \inv$ and $T \geq S$, 
$$
\Prb\Big(\sup_{0 \leq t \leq T} \Big|V(X_t^x) - V(x) - \int_0^t H(X_s^x)ds\Big| > \frac{\alpha}{2}T\Big) \leq \delta \,.
$$
\end{lem}

\begin{proof}
Let $M_t^{V}(x)$ be as in \eqref{Mtf}. 
By \Cref{V-W-linearly-bdd} $V$ has linearly bounded quadratic variation (see \Cref{lin-bdd}), so there is a continuous function $C: \overline{\inv} \to \mathbb{R}$ such that for $T \geq 1$ and $x \in \overline{\inv}$ it holds that
\begin{equation}\label{lbva}
\int_0^T \Pp_s \Gamma V(x)ds \leq C(x)T   \,.
\end{equation}
 By Tonelli's theorem and \Cref{V-W-linearly-bdd} 
$$
\E[(M_T^{V}(x))^2] = \int_0^T \Pp_s\Gamma V(x)ds \leq C(x)T \,.
$$
Then by Doob's inequality and \eqref{lbva}, 
$$
\Prb\Big(\sup_{0 \leq t \leq T} |M_t^{V}(x)| > \frac{\alpha }{2}T\Big) \leq \frac{4}{T^2\alpha^2} \E[(M_T^{V}(x))^2] \leq \frac{4}{T\alpha^2}C(x).
$$ 
Since $C$ is continuous and thus bounded on the compact set $\K \cap \overline{\inv}$, 
$$
\lim_{T \to \infty} \sup_{x \in \K \cap \inv} \Prb\Big(\sup_{0 \leq t \leq T} |M_t^{V}(x)| > \frac{\alpha}{2}T\Big) = 0,
$$ 
which proves the claim.
\end{proof}

\begin{deff}
    \label{st1} For any $x \in \inv$ and  $T > 0$ define the bounded stopping time 
    $$
    \tau_{x,T} \coloneqq \inf\{t \geq 0: |V(X_t^x) - V(x)| > \alpha T\} \wedge \inf\Big\{t \geq 0: \Big|\int_0^t H(X_s^x)ds\Big| > \frac{\alpha}{2}T\Big\} \wedge T.
    $$
\end{deff}

\begin{lem} \label{V-small}
    For every $x \in \mcM_0$ and $S > 0$ there is $\epsilon > 0$ and $T \geq S$ such that for all $y \in \inv$ with $d(y,x) < \epsilon$, $\E[V(X_{\tau_{y,T}}^y) - V(y)] \geq \frac{\alpha}{4}T$.
\end{lem}

\begin{proof}
For any $y \in \inv$ and $T > 0$ we define the events
\begin{multline*}
    Q \coloneqq \Big\{ 
    \int_0^T H(X_s^y) ds > \frac{\alpha}{2}T\Big\} \cap \Big\{\inf_{0 \leq t' \leq T} \int_0^{t'} H(X_s^y) ds > -\frac{\alpha}{2}T\Big\} \\
       \cap  \Big\{\sup_{0 \leq t \leq T} \Big|V(X_t^y) - V(y) - \int_0^t H(X_s^y)ds\Big| \leq \frac{\alpha}{2}T
    \Big\} =: Q_1 \cap Q_2 \cap Q_3 \,,
\end{multline*}
and 
$$
 \Big\{\int_0^{\tau_{y,T}} H(X_s^y)ds \geq \frac{\alpha}{2}T\Big\} =: R \,.
$$
Fix some compact neighborhood $\K$ of $x$ (recall that $\mcM$ is locally compact). 
   By \Cref{V-close-to-H}, there is some $S > 0$ such that for all $y \in \K \cap \inv$ and $T \geq S$, 
   \begin{equation}\label{ifvh}
   \Prb(Q_3^c) = \Prb\Big(\sup_{0 \leq t \leq T} \Big|V(X_t^y) - V(y) - \int_0^t H(X_s^y)ds\Big| > \frac{\alpha}{2}T\Big) \leq \frac{1}{8}.    
   \end{equation}
   By \Cref{H-big}, there is a $T \geq S$ and $\epsilon > 0$ such that for all $y \in \mcM$ with $d(y,x) < \epsilon$, it holds that 
   \begin{equation}
   \label{q1q2} \Prb(Q_1 \cap Q_2) = P_{y,T} > \frac{7}{8}\,.
   \end{equation}
   We may decrease $\epsilon > 0$ if necessary such that $d(y,x) < \epsilon$ implies $y \in \K$. We claim that $Q \subset R$. Indeed, for fixed $\omega \in Q$, if $\tau_{y,T}(\omega) = T$, then $\omega \in Q_1$ implies $\omega \in R$. If $\tau_{y,T}(\omega) < T$, then either $\Big|\int_0^{\tau_{y,T}} H(X_s^y)ds\Big| \geq \frac{\alpha}{2}T$, in which case $\omega \in Q_2$ implies $\omega \in R$, or $|V(X_{\tau_{y, T}}^y) - V(y)| \geq \alpha T$, in which case  $\omega \in Q_3$ yields that
 \begin{align*}
    \alpha T &\leq |V(X_{\tau_{y, T}}^y) - V(y)| \\
    &\leq  
\Big|V(X_{\tau_{y, T}}^y) - V(y) - \int_0^{\tau_{y, T}} H(X_s^y)ds\Big| + 
\Big|
 \int_0^{\tau_{y, T}} H(X_s^y)ds
\Big|
\\
&\leq \frac{\alpha}{2}T + \Big|
 \int_0^{\tau_{y, T}} H(X_s^y)ds
\Big| \,,
 \end{align*}
and again $\omega \in Q_2$ implies $\omega \in R$, as desired. 

Hence, by \eqref{ifvh} and \eqref{q1q2} we have for all $y \in \inv$ with $d(y,x) < \epsilon$ that
    \begin{equation}\label{p-big}
 \Prb(R) \geq \Prb(Q) = \Prb((Q_1 \cap Q_2) \setminus Q_3^c)
       \geq \Prb(Q_1 \cap Q_2) - \Prb(Q_3^c) \geq 
      \frac{7}{8} - \frac{1}{8} = \frac{3}{4}\,.
   \end{equation}
By the continuity of $t \mapsto \int_0^t H(X_s^y)ds$ and the definition of $\tau_{y,T}$ (\Cref{st1}) we have $\Big|\int_0^{\tau_{y,T}} H(X_s^y)ds \Big| \leq \frac{\alpha}{2}T$, and thus by \eqref{p-big} we conclude
\begin{align*}
        \E\Big[\int_0^{\tau_{y,T}} H(X_s^y)ds\Big] 
        \geq \Prb(R)\frac{\alpha}{2}T + (1-\Prb(R))\frac{-\alpha}{2}T \geq \frac{\alpha}{4}T \,.
    \end{align*}
 Since $\tau_{y,T}$ is a bounded stopping time, by \Cref{V-is-integral-of-H}
 $$
 \E[V(X_{\tau_{y,T}}^y) - V(y)] = \E\Big[\int_0^{\tau_{y,T}} H(X_s^y)ds\Big],
 $$ 
 which proves the claim. 
 \end{proof}

\subsection{Return Time} \label{return-time}
As mentioned at the beginning of this section, we wish to sample the paths of $X_t^y$ at a sequence of stopping times $\tau_n$ so that $V(X_{\tau_n}^y)$ satisfies \Cref{discrete-semimartingale-fact}. This means that we not only need a uniform lower bound on $\E[V(X^y_{\tau_1}) - V(y)]$, but also a uniform upper bound on $\E[(V(X^y_{\tau_1}) - V(y))^2]$. Given any $x \in \mcM_0$ and $S > 0$, \Cref{V-small} gives us the required uniform lower bound (namely, $\frac{\alpha}{4}S$) using $\tau_1 = \tau_{y,T_x}$ for some $T_x \geq S$. However, this is only a local result holding for $y$ close enough to $x$. In particular, if $\mcM_0$ is not compact there is no reason to believe that $\sup_{x \in \mcM_0} T_x < \infty$. Since it is unclear how to obtain an upper bound on $\E[(V(X^y_{\tau_{y,T}}) - V(y))^2]$ which is independent of $T$, we need to modify $\tau_{y,T}$ to ensure that $X_{\tau_n}^y$ stays in some fixed compact set. In particular, we let  the process run until  $X^y_t$ reenters a fixed compact set at time $\sigma_{x,T,m}$ (see \Cref{st2} below). The remainder of this section is dedicated to proving an upper bound on $\E[|V(X_{\sigma_{x,T,m}}^x) - V(X_{\tau_{x,T}}^x)|]$ which ensures that $\E[V(X_{\sigma_{x,T,m}}^x) - V(x)]$ is big.

\begin{deff}
    \label{st2} For $x \in \inv$, $T > 0$, and $m > 0$ we define the stopping time $$\sigma_{x,T,m} \coloneqq \inf\{t \geq \tau_{x,T} \mid W(X_t^x) < m\} \,,$$ where $\tau_{x,T}$ is as in \Cref{st1}.
\end{deff}

\begin{lem}
    \label{stopped-W-is-small}
    For $x \in \inv$, $T > 0$, and $m > 0$, $X_{\sigma_{x,T,m}}^x \in \inv \cap \{W \leq m\}$ a.s.
\end{lem}

\begin{proof}
    By \Cref{as1} and \Cref{inv-set}, $X_{\sigma_{x,T,m}}^x \in \inv$ a.s. By the right-continuity of the sample paths of $X_t^x$ and the definition of $\sigma_{x,T,m}$ (\Cref{st2}), $X_{\sigma_{x,T,m}}^x \in \{W \leq m\}$ a.s.
\end{proof}

For the rest of the section, recall that $K > 0$ is as in \Cref{as3}.
\begin{lem}
    \label{return-time-fast}    
    There is  $N > 0$ such that for all $m \geq N$ and $x \in \mcM$, 
    the stopping time $\eta_{m}(x) \coloneqq \inf\{t \geq 0 \mid W(X_t^x) < m\}$ satisfies
    $$
    \tilde{\eta}_m(x) \coloneqq \E[\eta_{m}(x)] \leq W(x).
    $$     
\end{lem}

\begin{proof}
Since $W'$ is proper, there is a compact set $\K \subset \mcM$ such that $W'(x) \geq K+1$ for $x \notin \K$. Let $N$ be the maximum value of $W + 1$ on $\K$. Fix any $m \geq N$ and $x \in \mcM$ and we simplify the notation by setting $\eta \coloneqq \eta_{m}(x)$. Then by the definition of $\eta$ we have 
$$
\inf_{0 \leq t < \eta} W(X_t^x) \geq m \geq N \,.
$$ 
Since $X_t^x \notin \K$ implies $W'(X_t^x) \geq K+1$ for all $0 \leq t < \eta$, then  $W \geq 0$ and  \Cref{ineq1} applied to $\eta \wedge t$ yield for any $t \geq 0$
 $$
  (K+1)\E[\eta \wedge t] \leq \E\Big[W(X_{\eta \wedge t}^x) + \int_0^{\eta \wedge t} W'(X_s^x)ds\Big] 
    \leq W(x) + K\E[\eta \wedge t].
    $$ Thus, $\E[\eta \wedge t] \leq W(x)$ and the claim follows from the Monotone convergence theorem after passing $t \to \infty$.
\end{proof}

\begin{cor}
    \label{sigma-small}
    There is  $N > 0$ such that for $x \in \inv$, $T > 0$, and $m \geq N$, $$\E[\sigma_{x,T,m} - \tau_{x,T}] \leq W(x) + KT.$$
        In particular, $$\E[\sigma_{x,T,m}] \leq W(x) + (K+1)T.$$
\end{cor}

\begin{proof}

Let $N > 0$ be as in \Cref{return-time-fast}, fix any $m \geq N$ and let $\eta_{m}(x)$ be as in \Cref{return-time-fast}.
By \Cref{return-time-fast}, $\eta_m(x)$ is almost surely finite, 
and then from the strong Markov property (\Cref{strong-Markov}) it follows that
$$
\E[\sigma_{x,T,m} - \tau_{x,T} \mid \mathcal{F}_{\tau_{x,T}}] = \tilde{\eta}_m(X_{\tau_{x,T}}^x) \leq W(X_{\tau_{x,T}}^x).$$
    Since $\tau_{x,T} \leq T$, by \Cref{ineq1} $$\E[W(X_{\tau_{x,T}}^x)] \leq W(x) + KT,$$ which proves the first claim. The second one follows since $\tau_{x,T} \leq T$.
\end{proof}

\begin{lem}
\label{change-in-V-small}
    For every $\beta > 0$ there is  $N > 0$ such that for all $T > 0$, $m \geq N$, and $x \in \inv \cap \{W \leq m\}$,$$\E[V(X_{\sigma_{x,T,m}}^x) - V(X_{\tau_{x,T}}^x)] \geq -\beta (K + 1)(m + KT).$$
\end{lem}

\begin{proof}
Since $H$ vanishes over $W'$ (\Cref{W'-vanish}) and $W$ is proper (\Cref{as3}), there is $N > 0$ such that $W(y) \geq N - 1$ implies 
$|H(y)| \leq \beta W'(y)$. Increase $N$ if necessary such that the assumptions of \Cref{return-time-fast} are satisfied and fix $m \geq N$. 

For any $y \in \inv \cap \{W \geq m\}$, 
 let $\eta \coloneqq \eta_m(y)$ be as in \Cref{return-time-fast} and recall that $\E[\eta] \leq W(y)$. Since $W(X^y_t) \geq m \geq N$ for $t \in (0, \eta)$, we have $|H(X^y_t)| \leq \beta W'(X^y_t)$ for $t \in (0, \eta)$, and so it follows from \Cref{V-is-integral-of-H} that
 $$
 \E[V(X_{\eta}^y) - V(y)] = \E\Big[\int_0^{\eta} H(X_s^y)ds\Big] \geq \beta \E\Big[\int_0^{\eta} -W'(X_s^y)ds\Big].
 $$
By \Cref{ineq1} and $\E[\eta] \leq W(y)$,
$$
\E\Big[\int_0^{\eta} W'(X_s^y)ds\Big] \leq W(y) + K\E[\eta] \leq (K+1)W(y) \,,
$$ 
and therefore
\begin{equation} \label{ozov}
\E[V(X_{\eta}^y) - V(y)] \geq -\beta (K + 1)W(y).    
\end{equation}
Fix $T > 0$ and $x \in \inv \cap \{W \leq m\}$ and note that $x$ belongs to a different set than $y$ above.

By the definition of $\sigma_{x,T,m}$, the 
strong Markov property (\Cref{strong-Markov}), and \eqref{ozov} we obtain
\begin{align*}
\E[V(X_{\sigma_{x,T,m}}^x) - V(X_{\tau_{x,T}}^x)] &= 
\E[(V(X_{\sigma_{x,T,m}}^x) - V(X_{\tau_{x,T}}^x)) \Id_{W(X_{\tau_{x,T}}^x) \geq m}]
\\
&\geq -\beta(K + 1)\E[W(X_{\tau_{x,T}}^x)].    
\end{align*}
Noting that $\tau_{x,T} \leq T$ (see \Cref{st1}), applying \Cref{ineq1}, and recalling that $W(x) \leq m$ proves the claim.
\end{proof}

\begin{cor}
    \label{can-choose-M-big}
    There is $N > 0$ such that for all $m \geq N$ there is $S > 0$ such that for all $T \geq S$ 
    $$
    \inf_{x \in \inv \cap \{W \leq m\}} \E[V(X_{\sigma_{x,T,m}}^x) - V(X_{\tau_{x,T}}^x)] \geq -\frac{\alpha T}{8} \,,
    $$
    where $\alpha > 0$ is as in \Cref{as4}.
\end{cor}

\begin{proof}
 Let $\beta = \frac{\alpha}{16K(K+1)}$, let $N > 0$ be as in \Cref{change-in-V-small}, and fix $m \geq N$. For any $T > 0$, by \Cref{change-in-V-small} 
 $$
 \inf_{x \in \inv \cap \{W \leq m\}} \E[V(X_{\sigma_{x,T,m}}^x) - V(X_{\tau_{x,T}}^x)] \geq -\beta (K + 1)(m + KT).$$ Let $S = \frac{m}{K}$ so that $T \geq S$ implies $m + KT \leq 2KT$, which gives
 $$
   \beta (K + 1)(m + KT) \leq \frac{\alpha T}{8},
 $$ proving the claim. 
\end{proof}

\begin{deff}
    \label{N}
    For the rest of the paper, we fix $N > 0$ large enough to satisfy \Cref{return-time-fast}, \Cref{sigma-small}, and \Cref{can-choose-M-big}. In particular, recalling the proof of \Cref{return-time-fast} we may assume that $W(x) \geq N$ implies $W'(x) \geq K+1$ and thus $\mathcal{L}W(x) \leq -1$ by \Cref{as3}.
\end{deff}

For every $\epsilon > 0$ and any closed set $\mathcal{S} \subset \mcM$ we define 
\begin{equation}\label{eq:nedf}
    N_{\epsilon}(\mathcal{S}) := \{y \in \mcM \mid d(y, \mathcal{S}) = \inf_{x \in \mathcal{S}} d(y, x) < \epsilon\} \,.
\end{equation}

\begin{lem}
    \label{V-supermartingale}
    Let $N$ be as in \Cref{N}. 
    For every $m \geq N$ there is $\epsilon > 0$, $n \geq 1$, and $T_1,\dots,T_n > 0$ such that for all $y \in \inv \cap \{W \leq m\} \cap N_{\epsilon}(\mcM_0)$ there is some $1 \leq i \leq n$ such that $\E[V(X_{\sigma_{y,T_i,m}}^y) -V(y)] \geq \frac{\alpha T_i}{8}$.
\end{lem}

\begin{proof}
Fix any $m \geq N$. By \Cref{can-choose-M-big}, there is $S > 0$ such that for all $T \geq S$ 
$$
\inf_{x \in \inv \cap \{W \leq m\}} \E[V(X_{\sigma_{x,T,m}}^x) - V(X_{\tau_{x,T}}^x)] \geq -\frac{\alpha T}{8}.
$$

By \Cref{V-small}, for every $x \in \mcM_0$ there is $\epsilon_x > 0$ and $T_x \geq S$ such that for all $y \in \inv$ with $d(y,x) < \epsilon_x$, it holds that $\E[V(X_{\tau_{y,T_x}}^y) - V(y)] \geq \frac{\alpha T_x}{4}$. Thus, if $y \in \inv \cap \{W \leq m\}$ and $d(y,x) < \epsilon_x$, 
$$
\E[V(X_{\sigma_{y,T_x,m}}^y) -V(y)] \geq \frac{\alpha T_x}{8}.
$$
Hence, by the compactness of $\mcM_0 \cap \{W \leq m\}$, 
there is  $\epsilon' > 0$ and $T_1,\dots,T_n > 0$ such that for all $y \in \inv \cap \{W \leq m\} \cap N_{\epsilon'}(\mcM_0 \cap \{W \leq m\})$ there is some $1 \leq i \leq n$ such that $\E[V(X_{\sigma_{y,T_i,m}}^y) -V(y)] \geq \frac{\alpha T_i}{8}$. 
To finish the claim, we show that there is $\epsilon > 0$ such 
that 
$$
\{W \leq m\} \cap N_{\epsilon}(\mcM_0) \subset \{W \leq m\} \cap N_{\epsilon'}
(\mcM_0 \cap \{W \leq m\}) \,.
$$ 
Since $\{W \leq m\}$ is compact, then  $\{W \leq m\} \setminus N_{\epsilon'}(\mcM_0 \cap \{W \leq m\})$ is compact and is covered by the collection of open sets $(\{W \leq m\} \setminus \overline{N_{\epsilon}(\mcM_0)})_{\epsilon > 0}$, because  
$$
\bigcup_{\epsilon > 0} \{W \leq m\} \backslash \overline{N_{\epsilon}(\mcM_0)} = \{W \leq m\} \backslash \mcM_0 = \{W \leq m\} \backslash (\mcM_0 \cap \{W \leq m\})
$$ 
which clearly contains $\{W \leq m\} \backslash N_{\epsilon'}(\mcM_0 \cap \{W \leq m\})$. Thus there is $\epsilon > 0$ such that 
$$
\{W \leq m\} \backslash N_{\epsilon'}(\mcM_0 \cap \{W \leq m\}) \subset \{W \leq m\} \backslash \overline{N_{\epsilon}(\mcM_0)},
$$ 
which means that 
$$
\{W \leq m\} \cap N_{\epsilon'}(\mcM_0 \cap \{W \leq m\}) \supset \{W \leq m\} \cap \overline{N_{\epsilon}(\mcM_0)} \supset \{W \leq m\} \cap N_{\epsilon}(\mcM_0)
$$
and the proof is finished. 
\end{proof}
\subsection{Bounding the Variance} \label{bounding-the-variance}
In this section we upper bound $\E[(V(X_{\sigma_{x,T,m}}^x) - V(x))^2]$, where $\sigma_{x,T,m}$ was introduced in \Cref{st2}.  We start by bounding $\E[\sigma_{x,T,m}^2]$.
 
\begin{lem}
    \label{quadratic-moments}
    For every $T > 0$ and $m > N$ (see \Cref{N}) we have  
    $$
    \sup_{x \in \inv \cap \{W \leq m\}}\E[\sigma_{x,T,m}^2] < \infty \,,
    $$
    where $\sigma_{x,T,m}$ was defined in \Cref{st2}.
\end{lem}

\begin{proof}
Fix $x \in \inv \cap \{W \leq m\}$, denote
   $\sigma = \sigma_{x,T,m}$, and note that $\sigma$ is a.s. finite by \Cref{sigma-small}. By \Cref{as3} \ref{3.1} (see \Cref{D2}) and \Cref{as5} \ref{5.1} (see \Cref{D+}), 
   $$
   M_t \coloneqq (M_t^W(x))^2 - \int_0^t \Gamma W(X_s^x)ds + KM_t^U(x)
   $$ 
   is a local martingale since it is a sum of a martingale and a local martingale. By \Cref{as5} \ref{5.3} and \ref{5.2}, $\Gamma W \leq KU' \leq K(K - \mathcal{L}U)$, and therefore by the definition of $M_t^U(x)$ (see \eqref{Mtf}) we have for each $t \geq 0$ that
   \begin{equation} \label{mtlowerbound}
   M_t \geq (M_t^W(x))^2 - K^2t + K(U(X_t^x) - U(x)) \geq 
   - K^2t - KU(x).
   \end{equation}
In particular, 
$$
\inf_{t \geq 0}M_{t \wedge \sigma} \geq -K^2 \sigma -K U(x) 
$$
and by \Cref{sigma-small},
\begin{equation} \label{sigmabound}
\E[K^2 \sigma] \leq K^2(W(x) + (K+1)T) \,. \end{equation} 
Thus the martingale $M_{t \wedge \sigma}$ is uniformly bounded from below by an integrable random variable, so by Fatou's lemma,  optional stopping and \eqref{mtlowerbound}, we obtain
\begin{equation} \label{rnoi}
0 \geq \E[M_{\sigma}] \geq \E[(M_{\sigma}^W(x))^2 - K^2\sigma  - KU(x)].
\end{equation}
By applying the inequality $(a+b)^2 \geq \frac{1}{2}a^2 - b^2$ twice, 
\begin{align*}
(M_{\sigma}^W(x))^2 &\geq \frac{1}{4}\Big(\int_0^{\sigma} \mathcal{L}W(X_s^x)\Id_{W(X_s^x) > N}ds\Big)^2 - \frac{1}{2}\Big(\int_0^{\sigma} \mathcal{L}W(X_s^x)\Id_{W(X_s^x) \leq N}ds\Big)^2 \\
&\qquad - (W(X_{\sigma}^x) - W(x))^2.    
\end{align*}
 Observe $W(X_s^x) \geq m > N$ for all $s \in (\tau_{x,T}, \sigma)$, and let $A \coloneqq \sup_{W(x) \leq N} |\mathcal{L}W(x)|$, which is finite since $\mathcal{L}W$ is continuous (by \Cref{D+}) and $W$ is proper (by \Cref{as3}). Since $\tau_{x,T} \leq T$, 
 $$
 \Big(\int_0^{\sigma} \mathcal{L}W(X_s^x)1_{W(X_s^x) \leq N}ds\Big)^2 
 = \Big(\int_0^{\tau_{x, T}} \mathcal{L}W(X_s^x)1_{W(X_s^x) \leq N}ds\Big)^2
 \leq A^2T^2.
 $$ 
 Similarly, since $W(x) > N$ implies $\mathcal{L}W(x) \leq -1$ (recall \Cref{N}), we obtain
 $$
 \frac{1}{4}\Big(\int_0^{\sigma} \mathcal{L}W(X_s^x)1_{W(X_s^x) > N}ds\Big)^2 \geq \frac{1}{4}(\sigma - \tau_{x,T})^2 \geq \frac{1}{8}\sigma^2 - \frac{1}{4}T^2.
 $$ 
 Finally, since we assumed $W(x) \leq m$, then it follows from \Cref{stopped-W-is-small} that 
 $$
 (W(X_{\sigma}^x) - W(x))^2 \leq 4m^2 \,.
 $$  
 Overall,  
 $$
 (M_{\sigma}^W(x))^2 \geq \frac{1}{8}\sigma^2 - A'
 $$ 
 where $A' \coloneqq \frac{1}{4}T^2 + \frac{1}{2}A^2T^2 + 4m^2$ is a positive constant. A substitution into 
 \eqref{rnoi} gives us 
  $$
  0 \geq \E[M_{\sigma}] \geq \E\Big[\frac{1}{8}\sigma^2 - A' -K^2\sigma -KU(x)\Big].
  $$ 
  Rearranging and using \eqref{sigmabound} gives 
  $$
  \E[\sigma^2] \leq 8A' + 8K^2(W(x) + (K+1)T) + 8KU(x),
  $$ 
  which proves the claim since the right hand side is uniformly bounded for $x \in \inv \cap \{W \leq m\}$.
\end{proof}

\begin{lem}
    \label{change-in-V-squared-small}
    For every $T > 0$, $m > N$ (see \Cref{N}) 
    $$
    \sup_{x \in \inv \cap \{W \leq m\}}\E[(V(X_{\sigma_{x,T,m}}^x) - V(x))^2] < \infty \,.
    $$
\end{lem}

\begin{proof}
Fix $T > 0$, $m > N$, and $x \in \inv \cap \{W \leq m\}$. First note that for $M^V_t$ as in \eqref{Mtf} and $H$ as in \Cref{as4} \ref{4.3} (so $H$ agrees with $\mathcal{L}V$ on $\inv$) we have 
$$
\E[(V(X_{\sigma_{x,T,m}}^x) - V(x))^2] \leq 2\E[(M_{\sigma_{x,T,m}}^V(x))^2] +2\E\Big[\Big(\int_0^{\sigma_{x,T,m}} H(X_s^x)ds\Big)^2\Big] 
\,.
$$ 
 Since $H$ vanishes over $W'$ (see \Cref{W'-vanish}, \Cref{vanish}), there is $A > 0$ such that $|H| \leq A + W'$.  Thus, 
$$
\E\Big[\Big(\int_0^{\sigma_{x,T,m}} H(X_s^x)ds\Big)^2\Big] \leq 2A^2\E[\sigma_{x,T,m}^2] + 2\E\Big[\Big(\int_0^{\sigma_{x,T,m}} W'(X_s^x)ds\Big)^2\Big].
$$ 
Since $0 \leq W' \leq K - \mathcal{L}W$ (see \Cref{as3}), 
$$
\E\Big[\Big(\int_0^{\sigma_{x,T,m}} W'(X_s^x)ds\Big)^2\Big] \leq 2K^2\E[\sigma_{x,T,m}^2] + 2\E\Big[\Big(\int_0^{\sigma_{x,T,m}} \mathcal{L}W(X_s^x)ds\Big)^2\Big].
$$ 
Furthermore, by \eqref{Mtf}
$$
\E\Big[\Big(\int_0^{\sigma_{x,T,m}} \mathcal{L}W(X_s^x)ds\Big)^2\Big] \leq 2\E[(M_{\sigma_{x,T,m}}^W(x))^2] + 2\E[(W(X_{\sigma_{x,T,m}}^x) - W(x))^2] \,.
$$ 
By \Cref{stopped-W-is-small} and our assumption that $W(x) \leq m$,
$$
\E[(W(X_{\sigma_{x,T,m}}^x) - W(x))^2] \leq 4m^2 \,.
$$ 
Combining all of these inequalities, 
\begin{multline*}
\E[(V(X_{\sigma_{x,T,m}}^x) - V(x))^2] \\
\leq 2\E[(M_{\sigma_{x,T,m}}^V(x))^2] + (4A^2 + 8K^2)\E[\sigma_{x,T,m}^2] + 16\E[(M_{\sigma_{x,T,m}}^W(x))^2] + 64m^2.
\end{multline*}
We finish the proof by noting that by 
\Cref{quadratic-moments} it holds that 
$$
\sup_{x \in \inv \cap \{W \leq m\}} \E[\sigma_{x,T,m}^2] < \infty
$$
and by \Cref{V-W-linearly-bdd}
$$\E[(M_{\sigma_{x,T,m}}^V(x))^2] + \E[(M_{\sigma_{x,T,m}}^W(x))^2] \leq 2K(U(x) + K\E[\sigma_{x,T,m}]) \,,
$$ 
where the right hand side  is uniformly bounded for $x \in \inv \cap \{W \leq m\}$ due to \Cref{quadratic-moments} and 
the continuity of $U$ on the compact set $\{W \leq m\}$.
\end{proof}

\subsection{Proof of \Cref{part1}} \label{proof-of-part1}
In this section we prove \Cref{part1} and we start with an auxiliary lemma.

\begin{lem}
    \label{stopped-V-not-that-much-smaller}
    For $x \in \inv$, $T > 0$, and $m > 0$, 
    $$
    \inf_{0 \leq t < \sigma_{x,T,m}} (V(X_t^x) - V(x))\Id_{W(X_t^x) < m} \geq -\alpha T \text{ a.s.}
    $$
\end{lem}

\begin{proof} By \Cref{st1}
    $$
    \inf_{0 \leq t < \tau_{x,T}} (V(X_t^x) - V(x)) \geq -\alpha T \text{ a.s.}
    $$  The claim follows since $W(X_t^x) \geq m$ for $\tau_{x,T} \leq t < \sigma_{x,T,m}$ (see \Cref{st2}).
\end{proof}

Next, we prove \Cref{part1}:

\begin{proof} 
Let $N$ be as in \Cref{N} and fix $\delta > 0, M > N$.
By \Cref{V-supermartingale} there is $\epsilon \in \big(0, \frac{1}{M}\big)$, 
$n \in \N$,  and $T_1,\dots,T_n > 0$ such that for all 
\begin{equation}\label{N-epsilon}
    z \in \mathcal{N}_{\epsilon, M} := \inv \cap \{W \leq M\} \cap N_{\epsilon}(\mcM_0)
\end{equation} (see \eqref{eq:nedf}) there is $1 \leq i(z) \leq n$ such that 
\begin{equation}\label{eq:lbpd}
\E[V(X_{\sigma_{z,T_{i(z)},M}}^z) -V(z)] \geq \frac{\alpha T_{i(z)}}{8} \,.
\end{equation}
To simplify the notation, we set $\sigma_z := \sigma_{z,T_{i(z)},M}$ for
$z \in \mathcal{N}_{\epsilon, M}$ and $\sigma_z = 1$ for $z \notin \mathcal{N}_{\epsilon, M}$.

For $y \in \inv$, we define an increasing sequence of stopping times $\tau_n$ as $\tau_0 = 0$, $\tau_{n+1} = \tau_n + \sigma_{X_{\tau_n}^y}$. (To interpret $\tau_{n+1} = \tau_n + \sigma_{X_{\tau_n}^y}$ rigorously, we note that there is a measurable function $\psi: \D_{[0,\infty)}(\mcM) \to [0,\infty)$ such that $\sigma_z = \psi(X^z_\cdot)$ for all $z \in \mcM$, and we set $\tau_{n+1} = \tau_n + \psi(X^y_{\tau_n + \cdot})$.) It follows from the strong Markov property for $X_t^y$ (\Cref{strong-Markov}) that the discrete-time process $Y_n^y \coloneqq X_{\tau_n}^y$ is a Markov chain on $\inv$ with respect to the filtration $\mathcal{G}_n \coloneqq \mathcal{F}_{\tau_n}$.

Next, we claim that $\tau_n \to \infty$ a.s. as $n \to \infty$. Indeed, let $T^\ast$ be the random variable such that $\tau_n \uparrow T^\ast$. Suppose $T^\ast < \infty$ with positive probability, and we restrict to such event. To simplify the notation, we set $\tau_z \coloneqq \tau_{z,T_{i(z)}}$ for
$z \in \mathcal{N}_{\epsilon, M}$ and $\tau_z = 1$ for $z \notin \mathcal{N}_{\epsilon, M}$. Then $\sigma_{Y^y_n} = \tau_{n + 1} - \tau_n \to 0$, and since $\sigma_z \geq \tau_z$ (see \Cref{st2}), we obtain $\tau_{Y^y_n} \to 0$. Note that a.s. $t \mapsto \int_0^t H(X_s^y)ds$ is continuous (recall that $t \mapsto X_t^y$ is cadlag a.s.) and $t \mapsto M^V_t(y)$ is cadlag (see \Cref{D2}), so $t \mapsto V(X_t^y)$ is also cadlag. In particular, we have that the limit as $t \uparrow T^\ast$ of the aforementioned functions exist, so $\Big|\int_{\tau_n}^{\tau_n + \tau_{Y^y_n}} H(X_s^y)ds\Big|$ and $|V(X_{\tau_n + \tau_{Y^y_n}}^y) - V(X_{\tau_n}^y)|$ both $\to 0$ as $n \to \infty$. By \Cref{st1}, this implies that $\tau_{Y^y_n}$ is eventually equal to $T_{i(Y^y_n)}$, contradicting $\tau_{Y^y_n} \to 0$, so $T^\ast = \infty$ a.s.

We also define the discrete-time process 
$$
Z_n^y \coloneqq V(Y_{\tau' \wedge n}^y) + b(n - \tau')_+- V(y) \,,
$$ 
where $b \coloneqq \frac{\alpha}{8} \min_{1 \leq i \leq n} T_i$, 
$a_+ := a \vee 0$
for any $a \in \mathbb{R}$, and 
$$
\tau' := \tau'_y \coloneqq \inf\{n \geq 0 \mid Y_n^y \notin \mathcal{N}_{\epsilon, M} \}\,,
$$ which is a $\mathcal{G}_n$-stopping time.
Next, we show that $Z_n^y$ satisfies the assumptions of \Cref{discrete-semimartingale-fact}.  By \eqref{eq:lbpd}, 
\begin{equation}\label{b-ineq}
    \inf_{z \in \mathcal{N}_{\epsilon, M}} \E[V(X_{\sigma_z}^z) - V(z)] \geq b > 0
\end{equation}
and by \Cref{change-in-V-squared-small}, there is some $B > 0$ such that 
\begin{equation}\label{eq:smev}
\sup_{z \in \mathcal{N}_{\epsilon, M}} \E[(V(X_{\sigma_z}^z) - V(z))^2] \leq B.    
\end{equation}
 Furthermore, by the definition of $Z_n$, 
 \begin{equation} \label{z-decomp}
     Z_n^y - Z_{n-1}^y = (V(Y_n^y) - V(Y_{n-1}^y))\Id_{\tau' > n-1} + b
 \Id_{\tau' \leq n-1} \,.
 \end{equation}
 Let $f(y) = \E[(V(Y_1^y) - V(y))^2]$. Using \eqref{z-decomp}, recalling that $\tau'$ is a stopping time with respect to $\mathcal{G}_n$, and using the Markov property of $Y_n^y$, we have 
 $$
 \E[(Z_n^y - Z_{n-1}^y)^2 \mid \mathcal{G}_{n-1}] = \Id_{\tau' \leq n-1}b^2 + \Id_{\tau' > n-1}f(Y_{n-1}^y) \,.
 $$ 
 On the event $\tau' > n-1$, $Y_{n-1}^y \in \mathcal{N}_{\epsilon, M}$ and \eqref{eq:smev} implies
    $$
    \Id_{\tau' > n-1}f(Y_{n-1}^y) \leq \Id_{\tau' > n-1}\sup_{z \in  \mathcal{N}_{\epsilon, M}} \E[(V(X_{\sigma_z}^z) - V(z))^2] \leq \Id_{\tau' > n-1}B \,.
    $$
    Consequently,  
    $$
    \E[(Z_n^y - Z_{n-1}^y)^2 \mid \mathcal{G}_{n-1}] \leq b^2 + B.
    $$
    Similarly, by \eqref{b-ineq} \begin{align*}
    \E[Z_n^y - Z_{n-1}^y \mid \mathcal{G}_{n-1}] &= \Id_{\tau' \leq n-1}b + \Id_{\tau' > n-1}\E[V(Y_n^y) - V(Y_{n-1}^y) \mid \mathcal{G}_{n-1}] \\
    &\geq \Id_{\tau' \leq n-1}b + \Id_{\tau' > n-1}\inf_{z \in \mathcal{N}_{\epsilon, M}} \E[V(X_{\sigma_z}^z) - V(z)] \\
    &\geq b \,.
\end{align*} 
Then by \Cref{discrete-semimartingale-fact}, there is some $C > 0$ such that 
for all $y \in \inv$
\begin{equation}\label{Pomega*=1}
   \Prb \Big(Z_n^y \geq \frac{b}{2}n-C \text{ for all } n \geq 0\Big) > 1 - \delta \,,
\end{equation}
where we stress that $C$ is independent of $y$. 
Also, by \Cref{as4} \ref{4.1}, there is $A > 0$ such that \begin{equation} \label{A}
    V(x) \geq A \qquad \textrm{implies } \quad  x \in N_{\epsilon}(\mcM_0) \,.
\end{equation}
We set $D \coloneqq A+C+\alpha \max_{1 \leq i \leq n}T_i$ and fix $y \in \inv \cap \{V > D\} \cap \{W \leq M\}$. Denote $\Omega^*$ the event that $Z_n^y \geq \frac{b}{2}n-C$ for all $n \geq 0$. Then \eqref{Pomega*=1} implies that $\Prb(\Omega^*) > 1 - \delta$, so we conclude the proof by verifying the following claims:
\begin{gather}\label{eq:smta}
    \limsup_{t \to \infty, W(X_t^y) \leq M} V(X_t^y) = \infty \textrm{ a.s. on } \Omega^* \\ \label{eq:fmta}
    \sup_{t \geq 0}d(X_t^y,\mcM_0)\Id_{W(X_t^y) < M} \leq \frac{1}{M} \textrm{ a.s. on } \Omega^* 
\end{gather}
First, we note that by \Cref{stopped-W-is-small} and the definition of $\mathcal{N}_{\epsilon, M}$ in \eqref{N-epsilon} we have that
$$
\tau' = \inf\{n \geq 0 \mid Y_n^y \notin N_{\epsilon}(\mcM_0)\} \,,
$$
so by \eqref{A} we obtain $V(Y_{\tau'}^y) < A$ a.s. on the event $\{\tau' < \infty\}$. In particular, on the event $\{\tau' < \infty\}$ we have a.s. that

$$
Z_{\tau'}^y = V(Y_{\tau'}^y) - V(y) 
\leq A - D = - C - \alpha \max_{1 \leq i \leq n} T_i < -C \,.
$$
Since $Z_n^y \geq \frac{b}{2}n-C \geq -C$ for all $n \geq 0$ on $\Omega^*$,
\begin{equation}\label{tau'=infty}
\tau' = \infty \quad {and} \quad Z_n^y = V(Y_{n}^y) - V(y) \textrm{ for all } n \geq 0 \text{ a.s. on the event } \Omega^*\,.
\end{equation}
Consequently, on $\Omega^*$
\begin{equation}\label{hsfes}
V(Y_{n}^y) = Z_n^y + V(y) > \frac{b}{2}n-C + D = 
\frac{b}{2}n + A + \alpha \max_{1 \leq i \leq n} T_i 
\end{equation}
so that with $\U \coloneqq \inv \cap \{W \leq M\} \cap \{V > A + \alpha \max_{1 \leq i \leq n} T_i\} \subset \mathcal{N}_{\epsilon,M}$
we have \begin{equation}\label{always-in-U}
    Y_n^y \in \U \textrm{ for all } n \geq 0 \text{ a.s. on the event } \Omega^*\,.
\end{equation}
By passing $n \to \infty$ in \eqref{hsfes}, noting that $W(Y_n^y) \leq M$ by \eqref{always-in-U}, and recalling that $\tau_n \to \infty$ a.s., we obtain \eqref{eq:smta}. 

By \Cref{stopped-V-not-that-much-smaller}, if $z \in \U$, then 
almost surely
$$
\inf_{0 \leq t < \sigma_z} (V(X_t^z) - A)\Id_{W(X_t^z) < M} \geq \inf_{0 \leq t < \sigma_z} (V(X_t^z) - V(z) + \alpha \max_{1 \leq i \leq n} T_i)\Id_{W(X_t^z) < M} \geq 0 
$$ so that by \eqref{A} we have
\begin{equation*}
    \sup_{0 \leq t < \sigma_{z}} d(X_t^z,\mcM_0)\Id_{W(X_t^z) < M} \leq \epsilon < \frac{1}{M} \text{ a.s.}
\end{equation*}

In other words, defining $\tilde{\phi}: \mcM \to [0,1]$ by
$$\tilde{\phi}(z) \coloneqq \Prb\Big(\sup_{0 \leq t < \sigma_{z}} d(X_t^z,\mcM_0)\Id_{W(X_t^z) < M} < \frac{1}{M}\Big) \,,$$
we have $\tilde{\phi}(z) = 1$ for all $z \in \U$. Recall that $Y_n^y = X_{\tau_n}^y$, so applying the strong Markov property (\Cref{strong-Markov}) we have
\begin{align*}
\Prb\Big(\sup_{\tau_{n} \leq t \leq \tau_{n + 1}} d(X_t^y,\mcM_0)\Id_{W(X_t^y) < M} < \frac{1}{M} \text{ and } Y_n^y \in \U\Big) &= \E[\tilde{\phi}(Y_n^y)\Id_{Y_n^y \in \U}] \\
&= \Prb(Y_n^y \in \U) \,,
\end{align*}
or in other words
\begin{equation}\label{in-U}
    Y_n^y \in \U \qquad \textrm{implies }\quad \sup_{\tau_{n} \leq t \leq \tau_{n + 1}} d(X_t^y,\mcM_0)\Id_{W(X_t^y) < M} < \frac{1}{M} \text{ a.s.}
\end{equation}

Thus, \eqref{eq:fmta} follows from \eqref{always-in-U} and \eqref{in-U} by recalling that $\tau_n \to \infty$ a.s. and writing
$$\sup_{t \geq 0} d(X_t^y,\mcM_0)\Id_{W(X_t^y) < M} 
    = \sup_{n \geq 0} \sup_{\tau_{n} \leq t \leq \tau_{n + 1}} d(X_t^y,\mcM_0)\Id_{W(X_t^y) < M} \,.$$
\end{proof}

\subsection{Proof of \Cref{main2}} \label{proof-of-main2}
In this section we provide the proof of \Cref{main2}:

\begin{proof}
Since by enlarging $M$ we make the statement stronger, we can without loss of generality assume $M > N$, where $N$ is as in \Cref{part1}.
 By \Cref{part1}, for all $n \geq 1$ there is a $D_n > 0$ such that for all $y \in \inv \cap \{V \geq D_n\} \cap \{W \leq M + n\}$,
    \begin{multline}\label{sbpe}
    \Prb\Big(\sup_{t \geq 0}d(X_t^y,\mcM_0)\Id_{W(X_t^y) < M + n} \leq \frac{1}{M + n} \\
    \text{ and } \limsup_{t \to \infty, W(X_t^y) \leq M + n} V(X_t^y) = \infty \Big) \geq 1 - \frac{\delta}{n},    
    \end{multline}
    and we may without loss of generality assume that $D_n \uparrow \infty$. Fix $y \in \inv \cap \{V \geq D_1\} \cap \{W \leq M\}$ and define 
    $$
    \tau_n := \inf\{t \geq 0 \mid V(X_t^y) > D_{n+1}, W(X_t^y) < M+2\},
    $$ 
    and we claim that a.s. $\tau_n \uparrow \infty$. Indeed, 
    by \Cref{D2} $t \mapsto M^V_t(y)$ is cadlag and so by the continuity of $t \mapsto \int_0^t H(X_s^y)ds$ it follows that $t \mapsto V(X_t^y)$ is cadlag. However, on the event $\tau_n \to T^* < \infty$, have $V(X_{\tau_n}^y) \geq D_{n + 1} \to \infty$, and therefore such event has zero probability. 
    
      Next,  let 
    $$
    A_n = \{\tau_n < \infty\} \cap \Big\{\limsup_{t \to \infty}d(X_t^y,\mcM_0)\Id_{W(X_t^y) < M + n + 1} \leq \frac{1}{M + n + 1}\Big\},
    $$ 
    a decreasing sequence of events. Define $\tilde{\phi}: \inv \to [0,\infty)$  by 
    $$
    \tilde{\phi}(z) = \Prb\Big(\limsup_{t \to \infty}d(X_t^z,\mcM_0)\Id_{W(X_t^z) < M + n + 1} \leq \frac{1}{M + n + 1}\Big)
    $$ 
    and by \eqref{sbpe}, $\tilde{\phi}(z) \geq 1 - \frac{\delta}{n+1}$ for $z \in \inv \cap \{V \geq D_{n+1}\} \cap \{W \leq M + n + 1\}$. In particular, by the definition of $\tau_n$, 
    $$
    \Id_{\tau_n < \infty}\tilde{\phi}(X_{\tau_n}^y) \geq \Id_{\tau_n < \infty}\Big(1 - \frac{\delta}{n+1}\Big) \,,
    $$ 
    where we used that $X_{\tau_n}^y \in \inv \cap \{V \geq D_{n+1}\} \cap \{W \leq M + 2\}$ a.s. by the right continuity of $t \mapsto X_t^y$ and the invariance of $\inv$ (see \Cref{as1}). By the strong Markov property (\Cref{strong-Markov}), 
    $$
    \Prb(A_n) = \E[\Id_{\tau_n < \infty}\tilde{\phi}(X_{\tau_n}^y)] \geq \Big(1 - \frac{\delta}{n+1}\Big)\Prb(\tau_n < \infty).
    $$
Since $\tau_n < \infty$ on the event that $\limsup_{t \to \infty, W(X_t^y) \leq M + 1} V(X_t^y) = \infty$, then for each $n \geq 1$ we have by \eqref{sbpe} that 
    $$
    \Prb(\tau_n < \infty) \geq 1 -\delta \,.
    $$ 
    Thus,
    $$
    \Prb\Big(\bigcap_{n=1}^{\infty} A_n\Big) = \lim_{n \to \infty} \Prb(A_n) \geq 1 - \delta.
    $$
    Since 
    $$
    \bigcap_{n=1}^{\infty} A_n \subset \big\{\lim_{M \to \infty} \limsup_{t \to \infty}d(X_t^y,\mcM_0)\Id_{W(X_t^y) < M} = 0\big\},
    $$ 
    by \Cref{distance-small-makes-V-big} 
    $$
    \Prb\Big(\liminf_{t \to \infty} \frac{V(X_t^y)}{t} \geq \alpha\Big) \geq 1 - \delta \,,
    $$ 
proving the claim.
\end{proof}

\section{Global Extinction}\label{global}
In this section we prove technical facts related to accessibility (see \Cref{accessible}) and give a proof of \Cref{main3}. In the entire section we fix a Markov quadruple $(\mcM, \mcM_0, \inv, \{X_t^x\}_{x \in \mcM, t \geq 0})$ satisfying \Cref{as1}--\ref{as5}.

\begin{lem}\label{equal-acc}
    For $x \in \mcM$ and $\U \subset \mcM$ open the following are equivalent:
    \begin{enumerate}
        \item[(1)]     $\int_0^\infty e^{-t}\Prb(X_t^x \in \U)dt > 0$.
        \item[(2)] There exists $t \geq 0$ such that $\Prb(X_t^x \in \U) > 0$.
        \item[(3)] $\Prb(\exists t \geq 0 \text{ such that } X_t^x \in \U) > 0$.
    \end{enumerate}
\end{lem}

\begin{proof}
    The statements (1) implies (2) and (2) implies (3) are straightforward. To prove that (3) implies (1), denote $f(t) = \Id_{\U}(X_t^x)$ so that (3) is equivalent to $\Prb(\exists t \geq 0 \text{ such that } f(t) > 0) > 0$. By the right continuity of $t \mapsto X_t^x$ and openness of $\U$, this is equivalent to the existence of random $t_2 > t_1 \geq 0$ such that $f(t) > 0$ for $t \in [t_1,t_2)$. In particular, the condition (3) implies $\Prb(\int_0^\infty e^{-t}f(t)dt > 0) > 0$, which by $f(t) \geq 0$, Tonelli's theorem, and $\E[f(t)] = \Prb(X_t^x \in \U)$ is equivalent to (1).
\end{proof}

\begin{lem}\label{accessibility-conditions}
    Suppose every point $x \in \inv$ satisfies \eqref{acc}.
  Then there is some $N > 0$ such that for all $x \in \inv, D > 0$ it holds that $\{V > D\} \cap \{W < N\}$ is accessible from $x$.
\end{lem}

\begin{proof}
Let $N$ be as in \Cref{return-time-fast} and fix $x \in \inv$. Let $M$ satisfy \eqref{acc} and let $D > 0$ be arbitrary. Set $T = 3M$ and choose any $y \in \{W \leq M\} \cap \inv$. Using the notation of \Cref{return-time-fast} we have 
$$
\Prb(\eta_N(y) > T) \leq \frac{\E[\eta_N(y)]}{T} \leq \frac{W(y)}{T} \leq \frac{1}{3}\,.
$$ 
With $M^V_t(y)$ as defined in \eqref{Mtf} and $H$ as in \Cref{as4} so that $H$ agrees with $\mathcal{L}V$ on $\inv$, we have by Doob's inequality, \Cref{V-W-linearly-bdd}, and \eqref{eq:puboh} that \begin{align*}
        \E[\sup_{t \leq T} |V(X_t^y) - V(y)|] &\leq \E\Big[\sup_{t \leq T} |M^V_t(y)|\Big] + \E\Big[\int_0^T |H(X_s^y)|ds\Big] \\
        &\leq 1 + 4\E\Big[|M^V_T(y)|^2\Big] + \E\Big[\int_0^T |H(X_s^y)|ds\Big] \\
        &\leq 1 + 4K(U(y) + KT) + bW(y) + (A+bK)T \\
        &\leq C\,,
    \end{align*}
where $A,b$ are as in the proof of \Cref{V-is-integral-of-H} and $C > 0$ is a constant depending only on $M$. Thus, for any $y \in \{W \leq M\} \cap \inv$ we have by Chebyshev inequality that
\begin{multline*}
    \Prb(\exists t \leq T \text{ such that } |V(X_t^y) - V(y)| \leq 3C \text{ and } W(X_t^y) \leq N) \\
    \geq \Prb(\sup_{t \leq T} |V(X_t^y) - V(y)| \leq 3C \text{ and } \eta_N(y) \leq T) \geq \frac{1}{3} \,.
\end{multline*}
Consequently, the strong Markov property (\Cref{strong-Markov}), applied to the stopping time $\tau = \inf\{t: V(X_t^x) > D + 3C \text{ and } W(X_t^x) < M\}$ implies 
\begin{multline*}
    \Prb(\exists t \geq 0 \text{ such that } V(X_t^x) > D \text{ and } W(X_t^x) < N + 1) \\
    \geq \frac{1}{3}\Prb(\exists t \geq 0 \text{ such that } V(X_t^x) > D + 3C \text{ and } W(X_t^x) < M) \,,
\end{multline*}
which is strictly positive by \eqref{acc} (using condition 3 of \Cref{accessible}).
Since $x \in \inv, D > 0$ were arbitrary, the claim is proven.
\end{proof}

Now we prove \Cref{main3}:

\begin{proof}
Let $N$ be as in \Cref{accessibility-conditions} and fix $x \in \inv$. By \Cref{main2}, there is $D > 0$ such that for all $y \in \Psi_{D, N} := \inv \cap \{V > D\} \cap \{W < N\}$ we have    
\begin{equation}\label{eq:Psi-io}
    \Prb\Big(\liminf_{t \to \infty} \frac{V(X_t^y)}{t} \geq \alpha\Big)  \geq \frac{1}{2}\,.
\end{equation}
Let $\tau_0' = 0$, $\tau_n' = \inf\{t \geq \tau_{n-1}' + 1 \mid X^x_t \in \Psi_{D, N}\}$, and  $\Omega^* \coloneqq \Big\{\liminf_{t \to \infty} \frac{V(X_t^x)}{t} \geq \alpha\Big\}$. Then the strong Markov property (\Cref{strong-Markov}) and \eqref{eq:Psi-io} imply that $\Prb(\Omega^* \mid \mathcal{F}_{\tau_n'}) \geq 1/2$ on the event $\{\tau_n' < \infty\}$, where $n \geq 1$. By Levy's 0-1 law, $\lim_{n \to \infty} \Prb(\Omega^* \mid \mathcal{F}_{\tau_n'}) = \Id_{\Omega^*}$ a.s. Thus, (up to a set of measure zero) if $\omega \in \Omega$ is such that $\tau_n'(\omega) < \infty$ for all $n \geq 1$, then $\Prb(\Omega^* \mid \mathcal{F}_{\tau_n'})(\omega) \geq 1/2$ for all $n \geq 1$, which implies that $\Id_{\Omega^*}(\omega) = 1$. To summarize, almost surely we have that 
    \begin{equation}\label{eq:osupb}
    \{X_t^x \in \Psi_{D, N} \text{ i.o.}\} \subset \Big\{\liminf_{t \to \infty} \frac{V(X_t^x)}{t} \geq \alpha\Big\} = \Omega^* \,,    
    \end{equation} 
    where by ``i.o." we mean there are times $t_n \uparrow \infty$ such that $X^x_{t_n} \in \Psi_{D, N}$.
    Thus, it suffices to show that almost surely \begin{equation}\label{equiv-io}
        (\Omega^*)^c \subset \{X_t^x \in \Psi_{D, N} \text{ i.o.}\} \,,
    \end{equation} 
    which combined with \eqref{eq:osupb} implies $\Prb((\Omega^*)^c) = \Prb((\Omega^*)^c \cap \Omega^*) = 0$. 

In the rest of the proof we show \eqref{equiv-io}. Let $K_n$ be as in \Cref{sigma-compact} applied to $A = \inv$. Let $\tau_n := \sum_{i=1}^n \sigma_i$, where $\sigma_i$ are iid exponentially distributed (with parameter 1) random variables independent of $X_t^x$, and $Y_n \coloneqq X_{\tau_n}^x$. We claim that almost surely 
\begin{equation}\label{Yn-io}
    (\Omega^*)^c \subset \bigcup_{m=1}^\infty\{Y_n \in K_m \text{ i.o.}\} \,.
\end{equation} 

By \Cref{muH-big} and \Cref{occ-is-tight}, almost surely on $(\Omega^*)^c$ we may assume there are $t_n \to \infty, \mu \in P(\mcM)$ such that $\mu_{t_n}^x \to \mu$ and $\mu(\mcM_0) < 1$ (where $\mu_t^x$ is as in \Cref{occ-meas}). Then, by Prohorov's Theorem and  $\mcM = \inv \cup \mcM_0$, it follows that there is some $m$ such that $\inf_n \mu_{t_n}^x(K_m) > 0$. With the notation $B = \{t \in [0,\infty) \mid X_t^x \in K_m\}$, the previous statement implies that
\begin{equation*}
\liminf_{n \to \infty} \frac{1}{t_n} \int_0^{t_n} \Id_B(t) \, dt = \liminf_{n \to \infty} \mu_{t_n}^x(K_m) > 0 \,.
\end{equation*}
In particular, $m(B) = \infty$, where $m(\cdot)$ denotes the Lebesgue measure on $[0,\infty)$. By the definition of Poisson process we have for any $t > 0$ that $|\{n : \tau_n \in B \cap [0,t]\}|$ is a Poisson random variable with parameter $m(B \cap [0,t])$, where $|S|$ denotes the cardinality of the set $S$. Thus, for any $N > 0$ we have $\lim_{t \to \infty} \Prb[|\{n : \tau_n \in B \cap [0,t]\}| \leq N] = 0$. It follows that $\Prb(|\{n : \tau_n \in B\}| = \infty) = 1$, and since $\tau_n \in B$ is equivalent to $Y_n \in K_m$, this shows \eqref{Yn-io} ($Y_n \in K_m$ i.o.)

Finally, we prove that \eqref{Yn-io} implies \eqref{equiv-io}. 
Fix $m$ and note that by Fatou's Lemma, \Cref{as2}, and Portmanteau's Theorem, the function $F : \inv \to \R$ defined as 
$$
F(y) \coloneqq \int_0^\infty e^{-t}\Prb(V(X_t^y) > D \text{ and } W(X_t^y) < N) \, dt = \int_0^\infty e^{-t}\Prb(X_t^y \in \Psi_{D, N}) \, dt
$$ 
is lower semicontinuous, where the equality follows from the invariance of $\inv$ (see \Cref{as1}). By \Cref{accessibility-conditions} (using condition (1) of \Cref{accessible}), we have $F(y) > 0$
for all $y \in K_m$ and thus $\inf_{y \in K_m} F(y) > 0$. By strong Markov property (\Cref{strong-Markov}) and exponential distribution of $\sigma_{n+1}$ (see definition of $\tau_n$), it follows that for each $n$
$$
\Prb(Y_{n+1} \in \Psi_{D, N} | Y_n) = F(Y_n)
$$
and so, similarly to the above, we have by Levy's 0-1 law that almost surely 
$$
\{Y_n \in K_m \text{ i.o.}\} \subset \{Y_n \in \Psi_{D, N} \text{ i.o.}\} \,,
$$ 
which by \eqref{Yn-io} shows \eqref{equiv-io}, completing the proof.
\end{proof}

\section{Generator Details} \label{generator-bs}
In this section, we first expand upon the exposition given in \Cref{generator-stuff} regarding the generator $\mathcal{L}$, the Carre Du Champ $\Gamma$, and the domain $\Dm(\mcM)$. Then we show that if an unbounded function can be approximated in a certain sense by functions in $\Dm(\mcM)$ (resp. $\Dm_2(\mcM)$), then it lies in the extended domain $\Dme_+(\mcM)$ (resp. $\Dme_2(\mcM)$). Finally, under general assumptions on $\mcM$ and the action of  the generator $\mathcal{L}$ on the domain $\Dm(\mcM)$, we show that  any smooth enough function that does not grow too quickly belongs to the extended domain. The proofs are mostly routine and can be skipped on the first reading.

\subsection{Generator for Bounded Functions} \label{bounded-gen}
In this subsection we fix a Feller process $\{X_t^x\}_{x \in \mcM, t \geq 0}$ (see \Cref{as2}) on a locally compact Polish space $\mcM$ with the Markov semigroup $\Pp_s$ (see \Cref{semigroup}). Recall \Cref{gen} and note that since $f \mapsto \mathcal{P}_tf$ is linear, it is standard to check that $\Dm(\mcM)$ is a subspace of $C_b(\mcM)$ and $\mathcal{L}:\Dm(\mcM) \to C_b(\mcM)$ is a linear map.

\begin{lem}
    \label{derivative-of-Pt-is-L}
    For $f \in \Dm(\mcM)$ and $t \geq 0$, $\Pp_tf \in \Dm(\mcM)$ and $\mathcal{L}\Pp_tf = \Pp_t\mathcal{L}f$.
\end{lem}

\begin{proof}
By \Cref{gen} \ref{first}, the definition of $\Pp_t$ (see \Cref{semigroup}),  \Cref{gen} \ref{third} and the Dominated convergence theorem, we have for any $t \geq 0$ that 
    $$
    \mathcal{L}\Pp_t f = 
    \lim_{s \downarrow 0} \frac{\Pp_s\Pp_tf - \Pp_tf}{s} = \lim_{s \downarrow 0} \Pp_t\Big (\frac{\Pp_sf - f}{s}\Big ) = \Pp_t\Big(\lim_{s \downarrow 0} \frac{\Pp_sf - f}{s}\Big) = \Pp_t\mathcal{L}f,
    $$ 
    where all limits are pointwise. Thus, \Cref{gen} \ref{first} holds with $f$ replaced by $\Pp_t f$. 
    Next, since $\mathcal{L}f \in C_b(\mcM)$ by \Cref{gen} \ref{second}, then $\mathcal{L}\Pp_t f = \Pp_t\mathcal{L}f  \in C_b(\mcM)$ by \Cref{as2} and so \Cref{gen} \ref{second} holds with $f$ replaced by $\Pp_t f$. Finally, since $\Pp_t$ is a contraction ($\|\Pp_t f\| \leq \|f\|$), then \Cref{gen} \ref{third} implies 
         $$\sup_{s>0} \Big\|\frac{\Pp_s\Pp_t f - \Pp_t f}{s} \Big\| = \sup_{s>0} \Big\|\Pp_t\Big (\frac{\Pp_s f - f}{s}\Big ) \Big\| \leq \sup_{s>0} \Big\|\frac{\Pp_s f - f}{s} \Big\| < \infty$$
     and \Cref{gen} \ref{third} with $f$ replaced by $\Pp_t f$ follows. 
\end{proof}

Before proceeding we need a preliminary lemma, called quasi-left continuity, which allows us to deduce that $t \mapsto P_tf(x)$ is continuous in $t$ (we only assumed it was right-continuous in \Cref{as2}):

\begin{lem}
    \label{quasi-left}
    If $x \in \mcM$ and $\tau$ is a predictable stopping time, meaning there are stopping times $\tau_n$ such that $\tau_n \uparrow \tau$ and $\tau_n < \tau$ almost surely on $\{\tau > 0\}$ for all $n$, then almost surely on $\{\tau < \infty\}$ we have $X_{\tau}^x = X_{\tau-}^x \coloneqq \lim_{t \uparrow \tau} X_t^x$. In particular, if $f \in C_b(\mcM)$ then $t \mapsto P_tf(x)$ is continuous in $t$.
\end{lem}

\begin{proof}
    We follow \cite[Proposition 17.29]{Kallenberg}. By \cite[Lemma 10.1v]{Kallenberg}, $\tau \wedge n$ is predictable for all $n > 0$, and since 
   $\{\tau < \infty\} = \cup_{n \geq 1}  \{\tau \leq n\}$ we may without loss of generality assume $\tau$ is finite. Since $\mcM$ is Polish, there is a countable collection of bounded continuous functions which separate the points, so it is enough to show that for all $f \in C_b(\mcM)$, $\E[(f(X_{\tau-}^x) - f(X_\tau^x))^2] = 0$.
   Since $f$ is continuous and $X_t^x$ is cadlag, then
   $$
   (f(X_{\tau-}^x) - f(X_\tau^x))^2 = \lim_{h \downarrow 0}\lim_{n \to \infty} (f(X_{\tau_n}^x) - f(X_{\tau_n + h}^x))^2 \text{ a.s.}
   $$
 Thus,  \eqref{rem:gmp} and the Dominated Convergence Theorem combined with \Cref{as2} imply
   \begin{align*}
       \E[(f(X_{\tau-}^x) &- f(X_\tau^x))^2] = \lim_{h \downarrow 0}\lim_{n \to \infty} \E \big[\E[(f(X_{\tau_n}^x) - f(X_{\tau_n + h}^x))^2 \mid \mathcal{F}_{\tau_n}] \big]\\
       &= 
       \lim_{h \downarrow 0}\lim_{n \to \infty} \E[f^2(X_{\tau_n}^x) + P_hf^2(X_{\tau_n}^x) - 2f(X_{\tau_n}^x)P_hf(X_{\tau_n}^x)] \\
       &= \lim_{h \downarrow 0} \E[f^2(X_{\tau-}^x) + P_hf^2(X_{\tau-}^x) - 2f(X_{\tau-}^x)P_hf(X_{\tau-}^x)] \\
       &= 0 \,,
   \end{align*}
   as desired. The remaining claim follows by taking $\tau = t$ and applying Dominated Convergence Theorem to deduce left-continuity of $t \mapsto P_tf(x)$, and then noting that we assumed right-continuity in \Cref{as2}.
\end{proof}

\begin{cor}
\label{martingale}
For all $f \in \Dm(\mcM)$, $x \in \mcM$ the process $M_t^f(x)$ given by \eqref{Mtf} is a martingale.
\end{cor}

\begin{proof}
For any $t \geq 0$, $M_t^f(x)$ is $\mathcal{F}_t$ adapted, and 
in addition $M_t^f(x)$ is integrable ($\E[|M_t^f(x)|] < \infty$) since $f, \mathcal{L}f \in C_b(M)$. By Fubini's theorem for conditional expectation and the Markov property, for $t,s \geq 0$
\begin{align*}
\E[M_{t+s}^f(x) - M_s^f(x) \mid \mathcal{F}_s] &= 
\E \Big[f(X^x_{t+s}) - f(X_s^x) - \int_s^{t+s} \mathcal{L}f(X_a^x)da\Big| \mathcal{F}_s \Big].
\\
&=
\Pp_tf(X_s^x) - f(X_s^x) - \int_0^t \Pp_a\mathcal{L}f(X_s^x)da.
\end{align*}
 Thus, to show that $M_t^f(x)$ is a martingale it suffices to show that $\Pp_tf = f + \int_0^t \Pp_s\mathcal{L}fds$ for all $t > 0$. By \Cref{derivative-of-Pt-is-L}, for fixed $x \in M$ the function $g(t) = \Pp_tf(x)$ has right derivative equal to $ t \mapsto \Pp_t\mathcal{L}f(x)$, and both $g$ and its right derivative are continuous by \Cref{quasi-left}. Thus, by \cite[Theorem 1.3]{Bruckner1978} $g$ is differentiable and $g'(t) = \Pp_t\mathcal{L}f(x)$. By fundamental theorem of calculus, $g(t) = g(0) + \int_0^t g'(s)ds$, which proves the claim.
\end{proof}

Next, we focus on the quadratic variation of $M_t^f(x)$ defined as follows:

\begin{deff}\label{def:sqrv}
Given a square integrable martingale $M_t$, we say that $\langle M \rangle_t$ is the (predictable) quadratic variation of $M_t$ if  $\langle M \rangle_t$ is the unique integrable, predictable, increasing process such that $M_t^2 - \langle M \rangle_t$ is a martingale. The existence and uniqueness of  $\langle M \rangle_t$ is guaranteed by Doob-Meyer decomposition theorem (see \cite[Theorem 10.5]{Kallenberg}).
\end{deff}

\begin{rem}
    Note that for $f \in \Dm(\mcM)$ (see \Cref{gen}) the martingale $M_t^f(x)$ defined in \eqref{Mtf} is bounded and thus square integrable, so \Cref{def:sqrv} applies.
\end{rem}
We also use the following definition.

\begin{deff}
\label{Gamma}
    Let $\Dm_2(\mcM)$ be the set of all $f \in \Dm(\mcM)$ such that $f^2 \in \Dm(\mcM)$. For $f \in \Dm_2(\mcM)$, let $\Gamma f \coloneqq \mathcal{L}(f^2) - 2f\mathcal{L}f$.
\end{deff}

\begin{rem}
    Note that $\Gamma f \geq 0$ since $\Gamma f = \lim_{s \downarrow 0} \frac{\Pp_sf^2 - (\Pp_sf)^2}{s}$. Indeed,
    \begin{align*}
 \mathcal{L}(f^2) - 2f\mathcal{L}f &= 
 \lim_{s \downarrow 0} \frac{\Pp_sf^2 - f^2}{s} - 2f\frac{\Pp_sf - f}{s} = \lim_{s \downarrow 0} \frac{\Pp_sf^2 - 2f\Pp_sf + f^2}{s} \\
 &=  \lim_{s \downarrow 0} \frac{\Pp_sf^2 - (\Pp_sf)^2}{s}
 + \lim_{s \downarrow 0} \frac{(\Pp_sf - f)^2}{s} 
 \\
 &= \lim_{s \downarrow 0} \frac{\Pp_sf^2 - (\Pp_sf)^2}{s} + 0\times (\mathcal{L}f)^2\,.
    \end{align*}
\end{rem}

\begin{lem}
\label{Gamma-is-quad-var}
    For $f \in \Dm_2(\mcM)$, $\langle M^f(x) \rangle_t = \int_0^t \Gamma f(X_s^x)ds$.
\end{lem}

\begin{proof}
    This is proved in \cite[Lemma 9.1]{persistence}.
\end{proof}

\subsection{Extending the Domain of the Generator} \label{extended-gen}
In this subsection we fix a Feller process $\{X_t^x\}_{x \in \mcM, t \geq 0}$ (see \Cref{as2}) on a locally compact Polish space $\mcM$ with the Markov semigroup $\Pp_s$ (see \Cref{semigroup}).

In lieu of \Cref{martingale} and \Cref{Gamma-is-quad-var}, we have shown that $\Dm_2(\mcM) \subset \Dme_2(\mcM)$ as defined in \Cref{D2} and that the two notions of $\mathcal{L}$ and $\Gamma$ agree. Similarly, $\Dm(\mcM) \cap \{f \in C_b(\mcM) \mid f \geq 0\} \subset \Dme_+(\mcM)$ as defined in \Cref{D+}. Next we show that $\Dme_2(\mcM), \Dme_+(\mcM)$ are in some sense the closures of $\Dm_2(\mcM),\Dm(\mcM) \cap \{f \in C_b(\mcM) \mid f \geq 0\}$ respectively.

In order to verify $f \in \Dme_+(A)$, we construct a localizing sequence of stopping times for $M_t^f(x)$ based on \Cref{sigma-compact}. We expand upon a similar proof given in \cite[Lemma 9.3]{persistence}.

\begin{lem}
    \label{localization}
    Suppose $A \subset \mcM$ is an open invariant set (see \Cref{inv-set}) with $K_n$ as in \Cref{sigma-compact}. Then for
    any $x \in A$, the sequence
$$\tau_n \coloneqq \inf\{t \geq 0 \mid X_t^x \notin K_n\} \wedge n$$
    is an increasing sequence of stopping times such that $\lim_{n \to \infty} \tau_n = \infty$ a.s.
\end{lem}

\begin{proof}
  Since $K_n^c$ is open, $X_t^x$ is cadlag, and the filtration is right-continuous, then 
$(\tau_n)_{n \geq 1}$ is a non-decreasing sequence of stopping times  (see \cite[Lemma 9.6iii]{Kallenberg}) and the limit $\tau = \lim_{n \to \infty} \tau_n$ is almost surely well defined.

We claim that $\tau_n < \tau$ almost surely. Indeed, for any $m\geq  n \geq 1$ denote
$$
\Omega_{n, m} := \{\omega \in \Omega \mid \tau_n = \tau < \infty, X_\tau \in  K_m\} \,.
$$ 
Then on $\Omega_{n, m}$, by right the continuity there is a random $h > 0$ such that $X_t^x \in K_{m+1}^\circ$ for all $\tau \leq t \leq \tau + h$, and since $\tau_n = \tau$ we have that $X_t^x \in K_n \subset K_{m+1}^\circ$ for all $t < \tau$. Thus, $X_t^x \in K_{m+1}$ for all $t \leq \tau + h$ and so $\infty > \tau \geq \tau_{m+1} \geq \tau + h > \tau$. Since $\Prb(\infty > \tau > \tau) = 0$,  $\Omega_{n, m}$ has zero probability. 

Hence, by the invariance of $A = \cup_{m \geq n} K_m$ for each $n \geq 1$ we have
   $$
   \{\tau_n =\tau < \infty\} \subset \bigcup_{m\geq n} \Omega_{n, m}
   $$ 
and the claim follows. 
In particular, we showed that $\tau$ is a predictable stopping time.

   We finish the proof by showing that $\Prb(\tau < \infty) = 0$. For a contradiction assume that there is $N > 1$ such that $\Prb(\tau < N) > 0$. On $\{\tau < N\}$, for any $n > N$ it holds that $\tau_n < n$ and thus $X_{\tau_n}^x \in \overline{K_{n}^c}$. Since $K_{n - 1} \subset K_n^\circ$, then 
   $$
   \overline{K_{n}^c} \subset \overline{(K_{n}^\circ)^c} = (K_{n}^\circ)^c \subset 
   K_{n - 1}^c \,, 
   $$
   and therefore $X_{\tau_n}^x \in K_{n-1}^c$ for any $n > N$. By \Cref{quasi-left}, $\lim_{n \to \infty} X_{\tau_n}^x = X_{\tau}^x$ almost surely on the event $\{\tau < N\}$, and thus $X_\tau^x \in \cap_n K_n^c = A^c$ which by \Cref{inv-set} has probability $0$, a contradiction to $\Prb(\tau < N) > 0$.
  \end{proof}

\begin{rem} \label{V-cadlag}
    Note that if $f: A \to \mathbb{R}$ is a continuous function and $t \mapsto X_t^x$ is a cadlag process, then $t \mapsto f(X_t^x)$ is almost surely cadlag, so all of the martingales defined in this section are cadlag (without any modification).
\end{rem}

\begin{lem}
       \label{local-martingale}
 Let $A \subset \mcM$ be an open invariant set (see \Cref{inv-set}) and assume $f:A \to [0,\infty)$ is such that:
    \begin{enumerate}
        \item[(1)] There exists $f_n \in \Dm(\mcM)$ such that $f_n \to f$ uniformly on compact subsets of $A$ and $f \geq f_n \geq 0$.
        \item[(2)] $\mathcal{L}f_n$ converges uniformly on compact subsets of $A$ to a continuous function from $A$ to $\R$, which we denote by $\mathcal{L}f$.
    \end{enumerate}

Then for all $x \in A$, $M_t \coloneqq M_t^f(x)$ defined in \eqref{Mtf} is a cadlag local martingale with the localizing sequence $(\tau_n)_{n\geq 1}$ as defined in \Cref{localization}. In particular, $f \in \Dme_+(A)$ (see \Cref{D+}).
\end{lem}

\begin{proof}
    By \Cref{martingale}, for any $n \geq 1$
    \begin{equation} \label{eq:amtg}
    M_t^n \coloneqq f_n(X_t^x) - f_n(x) - \int_0^t \mathcal{L}f_n(X_s^x)ds    
    \end{equation}
    is a  martingale. Since for any $m \geq 1$ $(\mathcal{L}f_n)_{n \in \N}$ converges uniformly on the compact set $K_m$ and $X^x_{t} \in K_m$ for any $t < \tau_m \leq m$, then a.s.
    $\int_0^{t \wedge \tau_m} \mathcal{L}f_n(X_s^x)ds \to \int_0^{t \wedge \tau_m} \mathcal{L}f(X_s^x)ds$. Consequently, for any $m \geq 1$ we have that $M_{t \wedge \tau_m}^n \to M_{t \wedge \tau_m}$ almost surely as $n \to \infty$. 
 Thus, to show that $M$ is a local martingale it suffices to show that $M_{\tau_m}^n \to M_{\tau_m}$ in $L^1$, or equivalently that $\{M_{\tau_m}^n\}_{n \in \mathbb{N}}$ is uniformly integrable. By the uniform convergence $\mathcal{L}f_n \to \mathcal{L}f$ on $K_m$ again, we obtain that $\int_0^{\tau_m} \mathcal{L}f_n(X_s^x)ds$ is uniformly bounded. Since $f_n \leq f$, by \eqref{eq:amtg} we just need to show that $f(X_{\tau_m}^x)$ is integrable, which follows from Fatou's lemma and $M_t^n$ being a martingale: 
 \begin{equation}\label{unif-int}
 \begin{aligned}
      \E[f(X_{\tau_m}^x)] &\leq \liminf_{n \to \infty} \E[f_n(X_{\tau_m}^x)] \\
      &= \liminf_{n \to \infty} \E\Big[f_n(x) + \int_0^{\tau_m} \mathcal{L}f_n(X_s^x)ds\Big] < \infty.
\end{aligned}
 \end{equation}

\end{proof}

We can prove a similar claim for $\Dm_2(\mcM)$ which can be viewed as a generalization of \Cref{Gamma-is-quad-var}:

\begin{lem}
    \label{square-integrable-martingale}
     Let $A \subset \mcM$ be an open invariant set (see \Cref{inv-set}) and $f:A \to \R$ be such that:
    \begin{enumerate}[label=(\roman*)]
        \item \label{fst} There exists $f_n \in \Dm_2(\mcM)$ such that $f_n \to f$ uniformly on compact subsets of $A$ and $|f| \geq |f_n|$.
        \item \label{snd} $\mathcal{L}f_n$ converges uniformly on compact subsets of $A$ to a continuous function from $A$ to $\R$, which we denote by $\mathcal{L}f$.
        \item \label{thr} $\Gamma f_n$ converges uniformly on compact subsets of $A$ to a continuous function from $A$ to $[0,\infty)$, which we denote by $\Gamma f$.
        \item \label{fourth} $\Gamma f \leq KU'$ for some $U, U': \mcM \to [0,\infty)$ and $K > 0$ satisfying \Cref{as5} \ref{5.1}, \ref{5.2}. 
    \end{enumerate}

Then for all $x \in A$, $M_t \coloneqq M_t^f(x)$ from \eqref{Mtf} is a square-integrable cadlag martingale with $\langle M \rangle_t = \int_0^t \Gamma f(X_s^x)ds$ (see \Cref{def:sqrv}). In particular, $f \in \Dme_2(A)$ (see \Cref{D2}).
\end{lem}

\begin{proof}
Let $(\tau_n)_{n\geq 1}$ be as defined in \Cref{localization}. 
Recall from \Cref{martingale} and \Cref{Gamma-is-quad-var} that $M_t^n$ defined in \eqref{eq:amtg} is a square-integrable martingale with $\langle M^n \rangle_t = \int_0^t \Gamma f_n(X_s^x)ds$, and from the proof of \Cref{local-martingale} that $M_{t \wedge \tau_m}^n \to M_{t \wedge \tau_m}$ a.s. as $n \to \infty$. Consequently  the Monotone Convergence Theorem, \Cref{localization}, Doob's Inequality, and Fatou's Lemma yield
\begin{align*}
    \E[\sup_{0 \leq s \leq t} M_s^2] &= \lim_{m \to \infty} \E[\sup_{0 \leq s \leq t \wedge \tau_m} M_s^2] \leq \liminf_{m \to \infty} 4\E[M_{t \wedge \tau_m}^2]  \\
         &\leq \liminf_{m \to \infty}\liminf_{n \to \infty} 4\E[(M_{t \wedge \tau_m}^n)^2] \,,
\end{align*}
and then \Cref{Gamma-is-quad-var}, Domintated convergence theorem,  Monotone convergence theorem, \Cref{localization}, and Tonelli's theorem imply
     \begin{equation} \label{eq:bsmar}
     \begin{aligned}
         \E[\sup_{0 \leq s \leq t} M_s^2] &\leq  
          \liminf_{m \to \infty}\liminf_{n \to \infty} 4\E\Big[\int_0^{t \wedge \tau_m} \Gamma f_n(X_s^x)ds\Big]  \\
         &= \liminf_{m \to \infty} 4\E\Big[\int_0^{t \wedge \tau_m} \Gamma f(X_s^x)ds\Big] = 4\E\Big[\int_0^t \Gamma f(X_s^x)ds\Big] 
         \\
         &= 4\int_0^t \Pp_s \Gamma f(x)ds 
         < \infty \,,
     \end{aligned}
     \end{equation}
     where in the last inequality we used our assumption \ref{fourth} and \Cref{ineq1}. 
 Thus, $M$ is a square-integrable martingale, so to finish the proof we show that 
  $M_t^2 - \int_0^t \Gamma f(X_s^x)ds$ is a martingale (see \Cref{def:sqrv}). Since $\Gamma f \geq 0$ and \eqref{eq:bsmar} imply
  $$
  \E\Big[\sup_{0 \leq s \leq t} \int_0^s \Gamma f(X_u^x)du\Big] = \E\Big[\int_0^t \Gamma f(X_s^x)ds\Big] < \infty \,,
  $$ 
  it is enough to establish that $M_t^2 - \int_0^t \Gamma f(X_s^x)ds$ is a local martingale. Similarly to the proof of \Cref{local-martingale}, we obtain 
  that that almost surely for any $m \geq 1$
  $$
  (M_{t \wedge \tau_m}^n)^2 - \int_0^{t \wedge \tau_m} \Gamma f_n(X_s^x)ds \to (M_{t \wedge \tau_m})^2 - \int_0^{t \wedge \tau_m} \Gamma f(X_s^x)ds \qquad 
  \textrm{as} \quad n \to \infty \,,
  $$ 
  and consequently it suffices to show that $\{(M_{\tau_m}^n)^2 - \int_0^{\tau_m} \Gamma f_n(X_s^x)ds\}_{n \in \mathbb{N}}$ is uniformly integrable. Since 
  $X_s^x \in K_m$ for any $s < \tau_m$ and $\Gamma f_n \to \Gamma f$ uniformly on 
  $K_m$, then $\sup_n \int_0^{\tau_m} \Gamma f_n(X_s^x)ds < \infty$.  We finish the proof by showing that $\{(M_{\tau_m}^n)^2\}_{n \in \mathbb{N}}$ is uniformly integrable. By QM-AM inequality we have 
  \begin{align*}
      (M_{\tau_m}^n)^2 &= \Big(f_n(X_{\tau_m}^x) - f_n(x) - \int_0^{\tau_m} \mathcal{L}f_n(X_s^x)ds\Big)^2 \\
      &\leq 3\Big[f_n(X_{\tau_m}^x)^2 + f_n(x)^2 + \Big(\int_0^{\tau_m} \mathcal{L}f_n(X_s^x)ds\Big)^2\Big],
  \end{align*} 
   and as above, since $f_n \to f$ and $\mathcal{L}f_n \to \mathcal{L}f$ uniformly on $K_m$,  the last two terms on the right hand side are bounded uniformly in $n$. Finally, since $f_n(X_{\tau_m}^x)^2 \leq f(X_{\tau_m}^x)^2$ it suffices to show that $\E[f(X_{\tau_m}^x)^2] < \infty$, and this follows similarly to the proof of \Cref{local-martingale}:
\begin{align*}
\E[f(X_{\tau_m}^x)^2] &\leq \liminf_{n \to \infty} \E[f_n(X_{\tau_m}^x)^2] \\
&\leq \lim_{n \to \infty}3\E\Big[(M^n_{\tau_m})^2 + f_n(x)^2 + \Big(\int_0^{\tau_m} \mathcal{L} f_n(X_s^x)ds\Big)^2\Big] \\
       &= \lim_{n \to \infty}3\E\Big[\int_0^{\tau_m} \Gamma f_n(X_s^x)ds\Big] + 3f_n(x)^2 + 3\E\Big[\Big(\int_0^{\tau_m} \mathcal{L}f_n(X_s^x)ds\Big)^2\Big] \\
       &= 3\E\Big[\int_0^{\tau_m} \Gamma f(X_s^x)ds\Big] + 3f(x)^2 + 3\E\Big[\Big(\int_0^{\tau_m} \mathcal{L}f(X_s^x)ds\Big)^2\Big] \\
       &< \infty \,,
     \end{align*}
as desired. 
\end{proof}

\subsection{Special Generator} \label{special-generator}
In this subsection we assume that the generator $\mathcal{L}$ can be written as $L_1 +  L_2$, where $L_1$ is a second order differential operator  representing the ``smooth" part of the process and $L_2$ is the generator for the ``jump" part of the process. Below in applications $L_1$ arises when Stochastic Differential Equations (SDEs) driven by Brownian motion are involved, whereas $L_2$ appears in the context of random switching and other jump processes. 
 Our goal is to show that the assumptions of \Cref{local-martingale} and \Cref{square-integrable-martingale} are satisfied if $L_1f, L_2f$ are well defined and continuous. We being with some preliminary definitions and a technical lemma.

\begin{deff}
\label{all-meas}
    For a Polish space $\mcM$, let $\PP(\mcM)$ denote the set of all finite positive Borel measures on $\mcM$. We endow $\PP(\mcM)$ with topology of weak convergence, that is, $\mu_n \to \mu$ if for all $f \in C_b(\mcM)$, $\mu_nf \to \mu f$.
\end{deff}

\begin{rem}
\label{meas-rem}
Note that the Portmanteau theorem is valid for $\PP(\mcM)$ as well. Indeed, if $\mu_n \in \PP(\mcM)$ and $\mu_n \to \mu$ weakly, then for $f = 1$ we have $\mu_n(\mcM) \to \mu(\mcM)$. Thus, if $\mu(\mcM) \neq 0$, then $\mu_n \to \mu$ if and only if $\mu_n / \mu_n(\mcM) \to \mu/\mu(\mcM)$. Since $\mu_n / \mu_n(\mcM)$ is a probability measure, Portmanteau theorem applies to $\mu_n \to \mu$. If $\mu(\mcM) = 0$, then the Portmanteau theorem theorem is straightforward to prove since all relevant limits vanish. 
\end{rem}

\begin{lem}\label{generalized-DCT} 
Let $\mcM$ be a Polish space, let $\mu: \mcM \to \PP(\mcM)$ and $f: \mcM \times \mcM \to \R$ be continuous maps, and fix $\epsilon > 0$. Then the function $x \mapsto \int f(x,y) d\mu_x(y)$ is continuous if $x \mapsto \int |f(x,y)|^{1 + \epsilon} d\mu_x(y)$ is bounded above by a continuous function.
\end{lem}

\begin{proof}
First we show that the function from $\mcM$ to $\PP(\mcM \times \mcM)$ given by $x \mapsto \delta_x \otimes \mu_x$ is continuous, where $\delta_x$ is dirac delta measure and $\otimes$ denotes the product of measures. Indeed, for any bounded Lipschitz function $f: \mcM \times \mcM \to \R$ with the Lipschitz constant $C$ and any $x_n \to x$ we have 
\begin{align*}
|\delta_{x_n} \otimes \mu_{x_n}(f) - \delta_x \otimes \mu_x(f)| &=
|\mu_{x_n} f(x_n,\cdot) - \mu_x f(x,\cdot)| \\
&\leq C \mu_{x_n}(\mcM) d(x_n,x) + |\mu_{x_n} f(x,\cdot) - \mu_x f(x,\cdot)| \to 0 \,,    
\end{align*} 
where in the limit $n \to \infty$ we used \Cref{meas-rem}, continuity of $x\mapsto \mu_x$, and $x_n \to x$.  

Consequently, the function from $\mcM$ to $\PP(\R)$ given by $x \mapsto f^*(\delta_x \otimes \mu_x)$ is continuous, where $f^*$ denotes the pushforward by $f$. Since $a \mapsto a$ is a continuous function from $\R \to \R$ which vanishes over the proper function $a \mapsto |a|^{1+\epsilon}$, the result follows by the same argument as in the proof of \Cref{lem:bhcon}.
\end{proof}

\begin{deff}\label{smooth-functions}
Let $\mathcal{X}$ be a locally compact Polish space, $n \geq 0$, and $A \subset \R^n \times \mathcal{X}$. We define $\mathcal{C}^2(A)$ to be the set of all continuous functions $f:A \to \R$ such that there is an open set $\U \subset \R^n \times \mathcal{X}$ with $A \subset \U$ such that $f$ extends to a continuous function $\tilde{f}$ on $\U$ which is twice continuously differentiable with respect to $b \in \R^n$. For $(b,x) \in \U$, we denote $\partial_i f(b,x)$ the partial derivative of $\tilde{f}$ with respect to the $i$th coordinate of $b$. We define $\Cc^2(A)$ to be the space of  compactly supported functions in $\mathcal{C}^2(A)$. If $n = 0$, that is $\R^{n} = \{0\}$,  we abbreviate $\mathcal{C}^2(A)$ as $\mathcal{C}(A)$, the set of all continuous functions $f: A \to \R$.
\end{deff}

For the rest of the section, we fix a Feller quadruple $(\mcM, \mcM_0, \inv, \{X_t^x\}_{x \in \mcM, t \geq 0})$ (see \Cref{feller-quadruple}) such that $\mcM = B \times \mathcal{X}$ where $B$ is a closed subset of $\R^n$ for some $n \geq 0$ and $\mathcal{X}$ is a locally compact Polish space. We also fix a continuous function $\mu: B \times \mathcal{X} \to \PP(B \times \mathcal{X})$, and we use $\mu_{(b,x)}$ to denote $\mu(b,x)$.

   Let $\Sigma_{ij},F_i: B \times \mathcal{X} \to \R$ be continuous. We define operators $L_1,L_2,L,\Gamma$ for suitable functions $f$ by 
   \begin{align*}
       L_1f(b,x) &= \sum_{i,j = 1}^n \frac{1}{2}\Sigma_{ij}(b,x)\partial_i \partial_j f(b,x) + \sum_{i=1}^n F_i(b,x)\partial_if(b,x) \\
       L_2f(b,x) &= \int_{A \times \mathcal{X}} f(c,y) - f(b,x) d\mu_{(b,x)}(c,y) \\
       Lf(b,x) &= L_1f(b,x) + L_2f(b,x) \\
       \Gamma f(b,x) &= \sum_{i,j = 1}^n \Sigma_{ij}(b,x)\partial_i f(b,x) \partial_j f(b,x) \\&+ \int_{A \times \mathcal{X}} (f(c,y) - f(b,x))^2 d\mu_{(b,x)}(c,y) \,.
   \end{align*} 
 Note that if $f \in \Cc^2(\mcM)$, then $Lf, \Gamma f$ are well defined and continuous.
   
\begin{ass}
    \label{as6}
Assume  $\mathcal{C}_c^2(\mcM) \subset \Dm(\mcM)$ (see  \Cref{gen}) and for  all $f \in \mathcal{C}_c^2(\mcM)$ the equality $\mathcal{L}f = Lf$ holds.
\end{ass}

\begin{lem}\label{special-gen}
    Suppose the Feller quadruple $(\mcM, \mcM_0, \inv, \{X_t^x\}_{x \in \mcM, t \geq 0})$ satisfies \Cref{as6} and let $A = \mcM$ or $\inv$. Then for any $f:A \to \R$:
    \begin{enumerate}
        \item If $f \in \mathcal{C}^2(A)$, $f \geq 0$, and $L_2 f$ is finite and continuous, then $f \in \Dme_+(A)$ with $\mathcal{L}f = Lf$ (see \Cref{D+}).
        \item If $f \in \mathcal{C}^2(A)$, $\Gamma f$ is finite and continuous, and \Cref{square-integrable-martingale} \ref{fourth} holds then $f \in \Dme_{2}(A)$ with $\mathcal{L}f = Lf$ and $\Gamma f$ as above (see \Cref{D2}).
    \end{enumerate}
\end{lem}

\begin{proof}
We only prove the second claim, since the first one follows similarly, using \Cref{local-martingale} instead of \Cref{square-integrable-martingale}.

Let $\U, \tilde{f}$ be as in \Cref{smooth-functions}. Let $K_n \subset \U$, $n \geq 1$ be compact such that $K_n \subset K_{n+1}^\circ$, where the interior is with respect to the subspace topology on $\U$ inherited from $\R^n \times \mathcal{X}$, and $\cup_n K_n = \U$.

Next, we construct a continuous cut-off function which is smooth in $b \in \R^n$. Our argument is standard, but since we were not able to locate the exact statement in the literature, we provide details. 

For each $n \geq 1$, by Urysohn's lemma there is 
a continuous function $g_n: \U \to [0,1]$ such that $K_n \subset \{g_n = 1\}$ and $\{g_n \neq 0\} \subset K_{n+1}$. By convolving $g_n(b,x)$ with a positive mollifier $\phi_n(b)$, we may obtain a continuous function $h_n: \U \to [0,1]$ such that $h_n$ is twice continuously differentiable in $b$, $K_{n-1} \subset \{h_n = 1\}$, and $\{h_n \neq 0\} \subset K_{n+2}$. Define $f_n = h_n \tilde{f}$ on $K_{n+2}$  and $f_n = 0$ on $(\R^n \times \mathcal{X}) \setminus K_{n+2}$, and therefore $f_n$ is a bounded continuous function on $\R^n \times \mathcal{X}$. Restricting the domain of $f_n$ to $\mcM$, we have $\{f_n \neq 0\} \subset K_{n+2} \cap \mcM$.
 Since $\tilde{f}(b,x), h_n(b,x)$ are twice continuously differentiable in $b$, then $f_n \in \mathcal{C}_c^2(\mcM)$. Thus, by \Cref{as6}, $f_n \in \Dm(\mcM)$ and $\mathcal{L}f_n = Lf_n$. Applying the same argument to $f_n^2 = h_n^2f^2$ and noting that $h_n^2$ shares the same properties as $h_n$ shows that $f_n^2 \in \mathcal{C}_c(\mcM) \subset \Dm(\mcM)$, and thus $f_n \in \Dm_2(\mcM)$ and $\mathcal{L}f_n^2 = Lf_n^2$.

Since $\mcM \setminus \inv$ is closed in $\mcM$, it is closed in $\R^n \times \mathcal{X}$. Thus, if $A = \inv$ we may without loss of generality assume that $\U$ is disjoint from $\mcM \setminus \inv$ (otherwise intersect $\U$ with the open set $(\R^n \times \mathcal{X})\setminus (\mcM \setminus \inv)$), or equivalently that $\U \cap \mcM = \U \cap \inv$. This ensures that $|f| \geq |f_n| \geq 0$ on $\mcM$, since if $(b,x) \in \mcM$ then either $(b,x) \in \U \cap \mcM = \U \cap \inv$ in which case $f_n(b,x) = h_n(b,x)f(b,x)$ or $(b,x) \notin \U \cap \mcM$ in which case $f_n = 0$. Similarly, $f_n \geq 0$ on $\mcM$ if $f \geq 0$ on $\mcM$. Also note that $(f_n^+,f_n^-) \to (f^+,f^-)$ uniformly on compact subsets of $A$ since $K_{n-1} \subset \{h_n = 1\}$, so $f_n = f_n^+ - f_n^- \to f^+ - f^- = f$ uniformly on compact subsets of $A$ as well. Thus, we have verified \Cref{square-integrable-martingale} \ref{fst}.

Next, we  show that \Cref{square-integrable-martingale} \ref{snd} holds true, that is, $\mathcal{L}f_n$ converges uniformly on compact subsets of $A$ to $Lf$. Observe that $Lf = L_1f + L_2f$ is continuous since $f \in \mathcal{C}^2(A)$ implies that $L_1f$ is continuous, and $\Gamma f$ being continuous with \Cref{generalized-DCT} yields that $L_2 f$ is continuous. 
Next, since $L_1f_n(b,x) = L_1f(b,x)$ for $(b,x) \in K_{n-1}^{\circ} \cap A$, then $L_1f_n$ converges uniformly on compact subsets of $A$ to $L_1f$.  To treat $L_2 f_n$, by the continuity of $L_2 f$ and $\Gamma f$ we have that $(b,x) \mapsto \mu_{(b,x)}f^2$ is continuous. Consequently,  
by \Cref{generalized-DCT} the maps $(b,x) \mapsto \mu_{(b,x)}f^+$ and $(b,x) \mapsto \mu_{(b,x)}f_n^-$ are continuous. Since $(b,x) \mapsto \mu_{(b,x)}f_n^+$ is an increasing sequence of nonnegative continuous functions and by Monotone Convergence Theorem it converges pointwise to the continuous function $\mu_{(b,x)}f^+$, by Dini's Theorem the convergence is uniform on compact subsets of $A$. Similarly, $\mu_{(b,x)}f_n^- \to \mu_{(b,x)}f^-$ uniformly on compact subsets of $A$. This shows that $L_2f_n \to L_2f$ uniformly on compact subsets of $A$ and thus $Lf_n = \mathcal{L}f_n \to Lf$ uniformly on compact subsets of $A$.

Applying the same argument to $f_n^2 = h_n^2f^2$ shows that $\mathcal{L}f_n^2 \to Lf^2$ uniformly on compact subsets of $A$. Standard computations show that $\Gamma f = Lf^2 - 2fLf$ and thus $\Gamma f_n = Lf_n^2 - 2f_nLf_n \to Lf^2 - 2fLf = \Gamma f$ uniformly on compact subsets of $A$, verifying \Cref{square-integrable-martingale} \ref{thr}. Thus, by  \Cref{square-integrable-martingale} we conclude that $f \in \Dme_{2}(A)$, as desired.
\end{proof}

\section{Some Important Classes of Examples} \label{verify-assumptions}
In this section we introduce model examples of Markov processes to which our theory can be applied: switching diffusions, SDEs driven by Brownian motion, and discrete-time Markov chains. In each case, under general assumptions, we provide results that verify some or all of \Cref{as1} -- \ref{as5}.

\subsection{Switching Diffusions} \label{switching-diff-section}
In this section we consider a class of Markov processes given by solutions to switching diffusions, which are a generalization of SDEs driven by Brownian motion. In particular, we investigate 
the model from \cite[equations (2.2) and (2.3)]{hybridSwitching}:
\begin{equation}\label{switch-diff}
    \begin{aligned}
        dX^{(x,i)}(t) &= F(X^{(x,i)}(t), \alpha^{(x,i)}(t))dt + \sigma(X^{(x,i)}(t), \alpha^{(x,i)}(t))dw(t) \\
        \Prb(\alpha^{(x,i)}(t) = j) &= q_{ij}(x)t + o(t) \quad \text{ for } i \neq j
    \end{aligned}
\end{equation}
where $n,m,d \in \N$, $\mathcal{X} = \{1,\dots,m\}$, $(x,i) \in \R^n \times \mathcal{X}$ is the initial condition, $F: \R^n \times \mathcal{X} \to \R^n$ and $\sigma: \R^n \times \mathcal{X} \to \R^{n \times d}$ are locally Lipschitz, $w$ is an $\R^d$ Brownian motion, and $Q: \R^n \times \mathcal{X} \times \mathcal{X} \to \R$ is continuous, uniformly bounded, and such that $q_{ij}(x) \coloneqq Q(x,i,j)$ satisfies $q_{ij}(x) \geq 0$ if $i \neq j$ and $\sum_{j = 1}^m q_{ij} = 0$.

In other words, $\alpha$ is a pure jump process on the finite state $\mathcal{X}$ with $x$-dependent generator $q(x)$. We define the following operators as in  \cite[(2,4)]{hybridSwitching} for suitable functions $f$:
  \begin{equation}\label{switch-gen}\begin{aligned}
       L_1f(x,\alpha) &= \sum_{i,j = 1}^n \frac{1}{2}\Sigma_{ij}(x,\alpha)\partial_i \partial_j f(x,\alpha) + \sum_{i=1}^n F_i(x,\alpha)\partial_if(x,\alpha) \\
       L_2f(x,\alpha) &= \sum_{\beta \neq \alpha} q_{\alpha \beta}(x)(f(x,\beta) - f(x,\alpha)) \\
       Lf(x,\alpha) &= L_1f(x,\alpha) + L_2f(x,\alpha) \\
       \Gamma f(x,\alpha) &= \sum_{i,j = 1}^n \Sigma_{ij}(x,\alpha)\partial_i f(x,\alpha) \partial_j f(x,\alpha) \\&+ \sum_{\beta \neq \alpha} q_{\alpha \beta}(x)(f(x,\beta) - f(x,\alpha))^2 \,.
       \end{aligned}
   \end{equation} 
   where $\Sigma(x,\alpha) = \sigma(x,\alpha)\sigma(x,\alpha)^T$ (here $^T$ denotes the transpose).

\begin{lem}\label{switch-is-feller}
   If there is a proper $W: \R^n \times \mathcal{X} \to [0,\infty)$ which is twice continuously differentiable with respect to $x \in \R^n$ and $LW \leq \gamma_0 W$ for some constant $\gamma_0 > 0$, then there is a (unique) solution $(X_t^y,\alpha_t^y)$ to \eqref{switch-diff} for any initial condition $y = (x,i)$ and $(X_t^y,\alpha_t^y)$ satisfies \Cref{as2} (the Feller property).
\end{lem}

\begin{proof}
The existence and uniqueness of 
 solution $(X_t^y,\alpha_t^y)$ to (2.2) and (2.3) for any initial condition $y = (x,i)$ follows from \cite[Proposition 2.20]{hybridSwitching}. To show \Cref{as2}, define the truncated process $X_t^{N,y}$ as in \cite[Proposition 2.20]{hybridSwitching}, and note that \cite[Proposition 2.20]{hybridSwitching} implies that 
 $$
 \Prb(\sup_{0 \leq t < \beta_N^x}|X_t^{N,y} - X_t^y| > 0) = 0\,,
 $$
where $\beta^x_N$ is the first exit time from a ball of radius $N$.
If we denote $W_N \coloneqq \inf_{|x| \geq N, i \in \mathcal{X}} W(x,i)$, then by \cite[proof of Theorem 2.7]{hybridSwitching}, for any $T > 0$ we have 
$$
W(y)e^{\gamma_0 T} = 
W(x,i)e^{\gamma_0 T} \geq W_N\Prb(\beta_N^y \leq T) \,,
$$
and since $W_N \to \infty$ as $N \to \infty$ ($W$ is proper) we obtain $\Prb(\beta_N^y \leq T) \to 0$ as $N \to \infty$ uniformly on compact subsets of $\R^n \times \mathcal{X}$ (for fixed $T$). Consequently, if $g \in C_b(\R^n \times \mathcal{X})$, then the functions $u_N(x,\alpha) = \E[g(X_t^{N,(x,\alpha)}, \alpha_t^{(x,\alpha)})]$ converge to $u(x,\alpha) = \E[g(X_t^{(x,\alpha)}, \alpha_t^{(x,\alpha)})]$ uniformly on compact sets for any $t > 0$. By \cite[Theorem 2.18]{hybridSwitching}, $u_N$ is continuous, and so $u$ is also continuous and thus \Cref{as2} holds.
\end{proof}

\begin{lem}\label{switch-domain}
     Supposed \Cref{as2} holds for \eqref{switch-diff}. Then if $A \subset \R^n \times \mathcal{X}$ is invariant (see \Cref{inv-set}) for $(X_t^y,\alpha_t^y)$ and $f \in  \mathcal{C}^2(A)$ (see \Cref{smooth-functions}), then
    \begin{enumerate}
        \item[(1)] If  $f \geq 0$, then $f \in \Dme_+(A)$ and $\mathcal{L}f = Lf$ (see \Cref{D+}).
        \item[(2)] If \Cref{square-integrable-martingale} \ref{fourth} holds then $f \in \Dme_{2}(A)$ (see \Cref{D2}),  $\mathcal{L}f = Lf$, and $\Gamma f$ is as above.
    \end{enumerate}
\end{lem}

\begin{proof} 
To prove (1) and (2), by \Cref{special-gen}, it suffices to verify \Cref{as6} since $L_1 f$ and the continuous part of $\Gamma$ are clearly continuous for $f \in \mathcal{C}^2(A)$ and
$L_2 f$ and the jump part of $\Gamma f$ are continuous for continuous $f$ since $q_{ij}$ are continuous. 
 To verify \Cref{as6}, we use the generalized It\^ o formula from \cite[(2.8)]{hybridSwitching} that imply that $M^f_t$ in \eqref{Mtf} is a martingale with $\mathcal{L}f = Lf$. Then \Cref{as6} follows since
  the converse of \Cref{martingale} is true: if $f, g$ are bounded continuous functions such that \eqref{Mtf} with $\mathcal{L}f$ replaced with $g$ is a martingale for all $x$, then $f$ is in the domain of the generator and $\mathcal{L}f = g$ (see \Cref{gen}).
\end{proof}

\subsection{SDEs} \label{SDE}
As a special case of \Cref{verify-assumptions}, we suppose $\mathcal{X}$ is a singleton set so that \eqref{switch-diff} is just a SDE driven by Brownian motion and $L_2 = 0$ in \eqref{switch-gen}. Then we have the following sufficient conditions for verifying the bulk of our assumptions.

\begin{lem}\label{cts-paths}
    Suppose $(\mcM, \mcM_0, \inv, \{X_t^x\}_{x \in \mcM, t \geq 0})$ is a Feller quadruple, where $X_t^x$ is a solution to a SDE driven by Brownian motion (\eqref{switch-diff} with $\alpha$ constant). Let $V \in \C^2(\inv)$. Then  \Cref{as3}, \Cref{as4} \ref{4.2}, and \Cref{as5} are satisfied if there is a proper map $\UU: \mcM \to [1,\infty)$ in $\mathcal{C}^2(\mcM)$ and constants $K,c > 0$ such that $L\UU \leq K - c\UU$ and $|\mathcal{L}V| + \Gamma V \lesssim 2K - \frac{L\UU}{\UU} + \frac{\Gamma \UU}{\UU^2}$.
\end{lem}
\begin{rem} \label{lesssim}
In what follows,
$f \lesssim g$ means that there is a constant $C > 0$ depending only on the parameters of the problem such that $f \leq Cg$.
\end{rem}
\begin{rem}
    \Cref{cts-paths} also applies in the case where $\mathcal{X}$ is not a singleton set but $U$ only depends on $x \in \R^n$ so that $L_2 U = 0$.
\end{rem}

\begin{proof}
Without loss of generality we can assume $c < 1 < K$. 
We show that  
\Cref{as3}, \Cref{as4} \ref{4.2}, and \Cref{as5} are satisfied for $U = \UU^{1/2}$,
$W = \UU^{1/4}$, $U' = \frac{c}{8K}\UU^{1/2}\phi$ and $W' = \frac{c}{16K}\UU^{1/4}\phi$, where
 \begin{equation}\label{phi}
\phi \coloneqq 2K - \frac{L\UU}{\UU} + \frac{\Gamma \UU}{\UU^2} = K + \Big(K - \frac{L\UU}{\UU}\Big) + \frac{\Gamma \UU}{\UU^2}  \geq K\,.    
\end{equation}

Since $\UU \geq 1$ and $\UU \in \mathcal{C}^2(\mcM)$, then $U, W \in \mathcal{C}^2(\mcM)$. By 
\Cref{switch-domain}, we obtain $U \in \Dme_+(\mcM)$ (\Cref{as5} \ref{5.1})
and $\mathcal{L}U = LU$.  
By calculus and \eqref{switch-gen}, 
\begin{equation*}
\mathcal{L}U = 
    L\UU^{1/2} = \frac{1}{2\UU^{1/2}}L\UU - \frac{1}{8\UU^{3/2}}\Gamma{\UU} = \UU^{1/2}\Big(\frac{1}{2}\frac{L\UU}{\UU} - \frac{1}{8}\frac{\Gamma \UU}{\UU^2}\Big) \,.
\end{equation*} 
Since $\UU$ is proper and satisfies $L\UU \leq K -c\UU$, then outside of a compact set we have $\frac{L\UU}{\UU} \leq -\frac{3c}{4}$, and so outside of the compact set we have $\frac{1}{2}\frac{L\UU}{\UU} \leq \frac{c}{8K}(\frac{L\UU}{\UU} - 2K)$, where we used that for all $0 < \gamma < \frac{1}{8}$ and $x \leq -\frac{3c}{4} \leq -\frac{2c}{3}$ it holds that $\frac{1}{2}x \leq \gamma x - \frac{c}{4}$. Then for large enough $K' > \frac{K}{4}$, by $\Gamma \UU \geq 0$ and $c \leq 1 \leq K$ we have 
\begin{equation*}
\begin{aligned}
    \mathcal{L} U &\leq K' + \UU^{1/2}\Big(\frac{c}{8K}\Big(\frac{L\UU}{\UU} - 2K - \frac{\Gamma \UU}{\UU^2}\Big) + \Big(\frac{c}{8K} - \frac{1}{8}\Big)\frac{\Gamma \UU}{\UU^2}\Big) \\
    &\leq K' - \frac{c}{8K}\UU^{1/2}\phi = K' - U' \,,
    \end{aligned}
\end{equation*}
and thus \Cref{as5} \ref{5.2} holds.

 Next, we focus on $W$:
\begin{equation*}
LW = 
    L(\UU)^{\frac{1}{4}} = \frac{1}{4 \UU^{\frac{3}{4}}}L\UU - \frac{3}{32\UU^{\frac{7}{4}}}\Gamma{\UU} = (\UU)^{\frac{1}{4}}\Big(\frac{1}{4}\frac{L\UU}{\UU} - \frac{3}{32}\frac{\Gamma \UU}{\UU^2}\Big)
\end{equation*} 
and by reusing the estimate $\frac{1}{2}\frac{L\UU}{\UU} \leq \frac{c}{8K}(\frac{L\UU}{\UU} - 2K)$ (which holds outside of a compact set) we obtain for sufficiently large $K'' > 0$ that 
\begin{equation}\label{eq:LWest}
\begin{aligned}
LW &\leq K'' + 
(\UU)^{\frac{1}{4}}\Big(
\frac{c}{16K}\Big(\frac{L\UU}{\UU} - 2K - \frac{\Gamma \UU}{\UU^2}\Big) + \Big(\frac{c}{16K} - \frac{3}{32}\Big) \frac{\Gamma \UU}{\UU^2}
\Big)
\\
&\leq K'' - \frac{c}{16K}(\UU)^{\frac{1}{4}} \phi = K'' - W' 
\,,
\end{aligned}
\end{equation}
where in the last inequality we used $\Gamma \UU \geq 0$ and $c \leq 1 \leq K$. 
Furthermore, since $\phi \geq \frac{\Gamma \UU}{\UU^2}$, by calculus \begin{equation}\label{GU}
  \Gamma W =  \Gamma (\UU^{\frac{1}{4}}) = \frac{\Gamma \UU}{16\UU^\frac{3}{2}} = \frac{1}{16} \UU^{1/2} \frac{\Gamma \UU}{\UU^2} \leq \frac{1}{16}\UU^{1/2} \phi = 
    \frac{1}{16}U'.
\end{equation}
Thus, $W$ satisfies \Cref{square-integrable-martingale}\ref{fourth}, and therefore \Cref{switch-domain} implies $W \in \Dme_2(\mcM)$ and $\mathcal{L}W = LW$. By \eqref{eq:LWest} and \eqref{GU}, 
$W$ satisfies \Cref{as3} and \Cref{as5}\ref{5.3}, where we note that $W$ and $W'$ are proper on account of $\UU$ being proper and \eqref{phi}.

Since $\UU \geq 1$, then $\Gamma V \lesssim \phi \lesssim  U'$. Thus, $V$ satisfies \Cref{square-integrable-martingale}\ref{fourth}, and therefore \Cref{switch-domain} implies $V \in \Dme_2(\inv)$ and \Cref{as5}\ref{5.4} holds true.
Finally, \Cref{as4}\ref{4.2} is a consequence of $|LV| \lesssim \phi \lesssim U' \lesssim (W')^2$.
\end{proof}

\subsection{Discrete Time} \label{discrete-time}

In this section, we investigate discrete time processes. 
 In \Cref{special-generator} we defined $\mcM = \R^n \times \mathcal{X}$, but since our process is discrete, we assume $n = 0$, that is $\R^n = \{0\}$. In the rest of the section we suppress any dependence on the ``continuous" variable $b \in \R^n$ and simply write $\mcM$ instead of $\mathcal{X}$. 

Suppose $\{X_n^x\}_{x \in \mcM, n \in \N}$ is a discrete-time Markov chain, which means that there is a filtered probability space $(\Omega, \mathcal{F}, \{\mathcal{F}_n\}_{n \in \N}, \Prb)$ and a family of $\mcM-$valued random variables $\{X_n^x\}_{x \in \mcM, n \in \N}$ such that:

\begin{itemize}
    \item $X_0^x = x$ a.s.
    \item $X^x_\cdot$ is adapted to $\{\mathcal{F}_n\}_{n \in \N}$, meaning $X_n^x$ is $\mathcal{F}_n$ measurable for each $n \in \N$.
    \item For all bounded measurable functions $f:\mcM \to \mathbb{R}$, the map
    $$
   \mcM \ni x \mapsto \Pp f(x) \coloneqq \E[f(X_1^x)] 
    $$ 
    is measurable and for any $n \geq 0$ we assume (homogeneity)
    that
    $$
    \Pp f(X_n^x) = \E[f(X_{n + 1}^x)|\mathcal{F}_n] \,.
    $$
\end{itemize}  
We also define the operators \begin{equation}\label{discrete-gen}
\begin{aligned}
     Lf(x) &\coloneqq \Pp f(x) - f(x) = \E[f(X_1^x) - f(x)] \\
    \Gamma f(x) &\coloneqq Lf^2(x) - 2f(x)Lf(x) = \E[(f(X_1^x) - f(x))^2]
\end{aligned}
\end{equation}
Again we let $\mcM_0 \subset \mcM$ be closed and $\inv \subset \mcM_0^c$ be open and dense. Let $\{N_t\}_{t \geq 0}$ be a Poisson process with rate $1$ independent of $X_n^x$ and define $Y_t^x \coloneqq X_{N_t}^x$. Then it is standard to show that $Y_t^x$ is a continuous-time Markov process on $\mcM$ with cadlag sample paths and semigroup $\Pp_t = e^{-t}e^{t\Pp}$.
Indeed, heuristically we have for any set $A$ that 
\begin{align*}
\Pp_t\Id_{A} &= 
\Prb(Y_t^x \in A) = \sum_{n=0}^\infty \Prb(X_n^x \in A \cap N_t = n) = \sum_{n=0}^\infty e^{-t}\frac{t^n}{n!}\Prb(X_n^x \in A) \\
&=  \sum_{n=0}^\infty e^{-t}\frac{t^n}{n!}\Pp^n \Id_A(x)
= e^{-t}e^{t\Pp} \Id_{A}
\,.
\end{align*}

The following lemma helps us translate between conditions on the discrete-time Markov chain $X^x_n$ and the continuous-time Markov process $Y_t^x$.

\begin{lem}\label{discretization}
     Let $\{X_n^x\}_{x \in \mcM, n \in \N}$ and $\{Y_t^x\}_{x \in \mcM, t \geq 0}$ be as above and $L, \Gamma$ be as in \eqref{discrete-gen}. Then
     \begin{enumerate}[label=(\roman*)]
        \item \label{discrete-inv-meas} 
        A probability measure $\mu$ is invariant for $Y_t$  if and only if $\mu$ is invariant for $X_t$. Specifically,
         $\mu \Pp f = \mu f$ for each $f \in C_b(\mcM)$ if and only if $\mu\Pp_t f = \mu f$ for each $t \geq 0$ and $f \in C_b(\mcM)$. 
        \item \label{disc-acc} $\U$ is accessible from $x$ in the sense of \Cref{accessible} applied to $Y_t$ if and only if one of the following equivalent conditions holds:
    \begin{enumerate}
        \item $\sum_{n=0}^\infty e^{-n}\Prb(X_n^x \in \U) > 0$.
        \item There exists $n \in \N$ such that $\Prb(X_n^x \in \U) > 0$.
        \item $\Prb(\exists n \in \N \text{ such that } X_n^x \in \U) > 0$.
    \end{enumerate}
        \item \label{discrete-as1} If for $A = \mcM_0$ and $A = \inv$ it holds that for all $x \in A,$ $\Pp \Id_A(x) = 1$, then \Cref{as1} holds for $Y_t^x$.
        \item \label{discrete-as2} If $\Pp f \in C_b(\mcM)$ for all $f \in C_b(\mcM)$, then \Cref{as2} holds for $Y_t^x$.
        \end{enumerate}
In particular, if \ref{discrete-as1} and \ref{discrete-as2} hold then $(\mcM, \mcM_0, \inv, \{Y_t^x\}_{x \in \mcM, t \geq 0})$ is a Feller quadruple (see \Cref{feller-quadruple}), which is assumed for the remaining items.
        \begin{enumerate}[label=(\roman*),resume]
        \item \label{discrete-gen-lem} If $f \in C_b(\mcM)$, then $f \in \Dm(\mcM)$ and $\mathcal{L}f = L f$ (see \Cref{gen}).
        \item \label{upsilon} Suppose $\Upsilon: \mcM \to [1,\infty)$ is proper (see \Cref{proper}) and $V: \inv \to \R$ is continuous. Suppose also that $\E[|V(X_1^x) - V(x)|^{2 + \epsilon}] \lesssim \sqrt{\Upsilon(x)}$ (see \Cref{lesssim}) and $\Pp \Upsilon \leq \rho^2 \Upsilon + C^2$ for some constants $\epsilon > 0, \rho \in (0,1), C > 0$. Then \Cref{as3}, \Cref{as4} \ref{4.2}, and \Cref{as5} hold for $Y_t^x$ with  $U = \sqrt{\Upsilon}, U' = (1 - \rho) U, W = \Upsilon^{1/4}, \text{ and } W' = (1 - \sqrt{\rho})W$.
        \item \label{disc-last} If \Cref{as3}--\ref{as5} hold for $(\mcM, \mcM_0, \inv, \{Y_t^x\}_{x \in \mcM, t \geq 0})$,  then  \Cref{main2} (and its corollaries \Cref{main}, \Cref{main3}, and \Cref{change-of-variables}) hold for $X_n^x$ (with $t$ replaced with $n$).
    \end{enumerate}
\end{lem}

\begin{proof}
    The statements (i)--(iv) are standard, so we omit the proofs.
    
    For \ref{discrete-gen-lem}, fix $x \in \mcM$ and compute
    \begin{align*}
         \lim_{s \downarrow 0} \frac{\Pp_sf(x) - f(x)}{s} &= \lim_{s \downarrow 0} \frac{e^{-s}e^{s\Pp}f(x) - f(x)}{s} \\
&= \lim_{s \downarrow 0} \frac{se^{-s}(\sum_{n = 0}^{\infty} s^n\frac{\Pp^{n+1}f(x)}{(n+1)!}) - (1 - e^{-s})f(x)}{s} \\
&=  \Pp f(x) - f(x) \,,
    \end{align*}
    where the last equality is justified since $\Pp$ is a contraction and 
    $$
    e^{-s}\Big|\sum_{n = 1}^{\infty} s^n\frac{\Pp^{n+1}f(x)}{(n+1)!}\Big| \leq e^{-s}\Big(\frac{e^s - 1 - s}{s}\Big)\|f\| \to 0
    \qquad \textrm{as} \quad s \downarrow 0 \,.
    $$ 
    Similarly to the above, we obtain
\begin{align*}
\sup_{s > 0} \Big\|\frac{\Pp_sf - f}{s}\Big\| &= 
\sup_{s > 0} \Big\| e^{-s}\sum_{n = 0}^{\infty} s^n\frac{\Pp^{n+1}f}{(n+1)!} - \frac{(1 - e^{-s})}{s}f\Big\|    \\
&\leq
\sup_{s > 0}  \Big|\frac{e^{-s}(e^s - 1)}{s}\Big| \|f\|+ 
\sup_{s > 0} \Big|\frac{(1 - e^{-s})}{s}\Big| \|f\| \\
&= 2 \sup_{s > 0} \Big|\frac{(1 - e^{-s})}{s}\Big| \|f\| < \infty \,.
\end{align*}
    Thus, \ref{discrete-gen-lem} is verified, which shows that \Cref{as6} is satisfied with  $L_1 = 0$ and $\mu_x(A) = \Pp\Id_A(x)$.
    
   To prove \ref{upsilon} we first note that if $\mu_{x}$ is the law of $X_1^x$, then by the Feller property, $x \mapsto \mu_x$ is continuous. Hence, by $\E[|V(X_1^x) - V(x)|^{2 + \epsilon}] \lesssim U$ and \Cref{generalized-DCT}, the function $\Gamma V$ defined in \eqref{discrete-gen} is finite and continuous. In addition, by our assumption in \ref{upsilon}
   $$
   \E[|U(X^x_1) - U(x)|^2] \leq  2\Pp U^2(x) + 2U^2(x) \leq 2\rho^2\Upsilon(x) + 2C^2 + 2U^2(x) 
   $$
   and so the continuity of $x \mapsto LU(x)$ follows from \eqref{discrete-gen}, the continuity of $U$ and $x \mapsto \mu_x$, and \Cref{generalized-DCT}.
Analogously,  we establish that $\Gamma W$ is continuous. Thus, \Cref{special-gen} implies 
$U \in \Dme_+(\mcM)$ and $V, W \in \Dme_2(\mcM)$, and consequently \Cref{as3}\ref{3.1} and \Cref{as5}\ref{5.1} are satisfied. Next, by Jensen's inequality and $\sqrt{a+b} \leq \sqrt{a} + \sqrt{b}$ we have
$$
\Pp U = \Pp \sqrt{\Upsilon} \leq \sqrt{\Pp \Upsilon} \leq \rho U + C\,,
$$
and therefore
\begin{equation}\label{sqrt}
          LU = \Pp U - U \leq (\rho - 1)U + C = C - U'\,
\end{equation}
and \Cref{as5}\ref{5.2} follows. Similarly we obtain $L W \leq (\sqrt{\rho} - 1)W + \sqrt{C}$, which proves \Cref{as3}\ref{3.2}. Also, since $W = \sqrt{U} \geq 0$ and $\Pp W \geq 0$, using \eqref{sqrt} we have
$$
\Gamma W = LU - 2WL W = LU - 2W\Pp W + 2U \leq (\rho + 1)U + C \lesssim U 
$$
and \Cref{as5}\ref{5.3} holds true. Furthermore, by Jensen's inequality
$$
|LV| \leq (\E[|V(X_1^x) - V(x)|^{2 + \epsilon}])^{\frac{1}{2 + \epsilon}} \lesssim 
(W'(x))^{\frac{2}{2 + \epsilon}} 
$$
and since $\epsilon > 0$ and $W'$ is proper, $LV$ vanishes over $W'$, establishing \Cref{as4}\ref{4.2}. Finally, by Jensen's inequality and $U \geq 1$ we have 
$$
\Gamma V \leq (\E[|V(X_1^x) - V(x)|^{2 + \epsilon}])^{\frac{2}{2 + \epsilon}} \lesssim U^{\frac{2}{2 + \epsilon}} \lesssim U \,, 
$$
which implies \Cref{as5}\ref{5.4} and the proof of (vi) is finished.

  For \ref{disc-last}, recall that \Cref{as1}-\ref{as2} are met since  $(\mcM, \mcM_0, \inv, \{Y_t^x\}_{x \in \mcM, t \geq 0})$ is a Feller quadruple. Then
  by the strong law of large numbers we have that almost surely $\lim_{t \to \infty} \frac{t}{N_t} = 1$ and thus (almost surely) $$\liminf_{n \to \infty} \frac{V(X_n^x)}{n} = \liminf_{t \to \infty} \frac{V(Y_t^x)}{N_t} = \liminf_{t \to \infty} \frac{V(Y_t^x)}{t}\frac{t}{N_t} =  \liminf_{t \to \infty} \frac{V(Y_t^x)}{t}\,.$$
\end{proof}

\section{Examples} \label{examples}
In this section we illustrate how our theory can be applied to some characteristic examples. We recommend that the reader is familiar with the notation and definitions from \Cref{notation} and the results in \Cref{main-results}. The first example is a model of an epidemic using a switching diffusion, the second one is a SDE obtained by adding noise to the Lorenz system (a prototypical model for turbulence), the third one is a very general class of discrete-time ecological models, and the fourth is a very general class of continuous-time ecological models based on Stochastic Kolmogorov equations. We use the setup and results from \Cref{switching-diff-section} for the first example, from \Cref{SDE} for the second and fourth examples, and from \Cref{discrete-time} for the third example.

\subsection{Stochastic SIS epidemic model on network with Markovian switching}
In this example, analyzed in \cite{SIS},  the dynamics of $(x(t),s(t))$ with initial condition $(x_0,s_0)$ are given by the following switching diffusion: 
\begin{align*}
    dx_i &= [\beta(s)b_i(x,s)(1 - x_i) - \delta(s)x_i] dt + \sigma_i(x_i,s)b_i(x,s)(1 - x_i)dW_i(t) \\
    b_i(x, s) &\coloneqq \sum_{j=1}^N a_{ij}(s)x_j  \\
    \Prb(s(t) = s') &= q_{s_0s'}t + o(t) \quad \text{ for } s_0 \neq s' \,.
\end{align*}
Here, we set $\mcM = [0,1]^N \times S = [0,1]^N \times \{1,\dots,m\}$ for some $N,m \in \N$, $\mcM_0 = \{0\} \times S$, and $\mcM_+ = \mcM_0^c$. 
For $x \in [0,1]^N$ we write $(x_1,\dots,x_n)$ for the components of $x$ and use $(x,s)$ to denote an element of $\mcM$ (where $s \in S$). Also, $x_i$ represents the probability that the $i$th node is infected and $s$ represents a different regime/environment. The functions $\beta, \delta: S \to (0,\infty)$ represent the infection and recovery rates and the parameters $a_{ij}: S \to \{0,1\}$ for $1 \leq i,j \leq N$ characterize the connections in the network so that the matrix $A(s)$ with entries $a_{ij}(s)$ is the adjacency matrix of the network when the environment is in state $s$. The intensity of the noise is given by continuously differentiable functions $\sigma_i: [0,1] \times S \to [0,\infty)$ for $1 \leq i \leq N$. We assume that $\sigma_i(x_i,s) > 0$ for $x_i \in (0,1)$ in order to guarantee that $\mcM_0$ is accessible, and also that $\sigma_i(0,s) = 0$ so that $x_i$ remains nonnegative for all times. The driving  noise is an $N$-dimensional Brownian motion $(W_1(t),\dots,W_N(t))$.

$s(t)$ is an irreducible Markov process on $S$ which is independent of the Brownian motion and $q_{ss'}$ denotes the transition rate from state $s \in S$ to $s' \in S$, meaning that $q_{ss'} \geq 0$ if $s \neq s'$ and $\sum_{s' = 1}^m q_{ss'} = 0$. Since $s(t)$ is irreducible, it has a unique invariant probability measure which we denote as $\rho \in [0,1]^m$ satisfying $\sum_{s = 1}^m \rho_sq_{ss'} = 0$ and $\sum_{s = 1}^m \rho_s = 1$. 

In \cite[Theorem 1 and Remark 1]{SIS} it is shown that the solution $(x(t),s(t))$ exists and is unique, and furthermore $(x(t),s(t))$ remains in $\mcM$ (resp. $\inv$, $\mcM_0$) when the initial condition is in $\mcM$ (resp. $\inv, \mcM_0$) so that \Cref{as1} is satisfied. Since $\mcM$ is compact, it follows from \Cref{switch-is-feller} that the Feller property (\Cref{as2}) holds. Thus, $(\mcM, \mcM_0, \inv, \{(x(t),s(t))\}_{t \geq 0})$ is a Feller quadruple (we suppress the superscript for the initial condition).

In the notation of \eqref{switch-diff} and \eqref{switch-gen} we have
\begin{equation}\label{F-and-Sigma}
\begin{aligned}
    F_i(x,s) &= \beta(s)b_i(x,s)(1 - x_i) - \delta(s)x_i \\
    \Sigma_{ij}(x,s) &= \begin{cases}
   [\sigma_i(x_i,s)b_i(x,s)(1 - x_i)]^2        & \text{if } i = j \\
   0        & \text{otherwise} 
  \end{cases} \,.
\end{aligned}
\end{equation}

Our strategy is to use \Cref{change-of-variables} and \Cref{compact} to conclude that \Cref{main3} applies to the problem. For $(x,s) \in \inv$, we set $V(x,s) = -\frac{1}{2}\log{\|x\|^2}$, where $\|x\| \coloneqq \sqrt{\sum_{i=1}^N x_i^2}$ is the Euclidean norm, and observe that $V$ satisfies \Cref{as4}\ref{4.1}. Since $V$ is independent of $s$, from
\eqref{switch-gen} and \eqref{F-and-Sigma} follows
\begin{equation}\label{LV-SIS}
    \begin{aligned}
    L V(x,s) &= \sum_{i = 1}^n \frac{1}{2}[\sigma_i(x_i,s)b_i(x,s)(1 - x_i)]^2\frac{-\|x\|^2 + 2x_i^2}{\|x\|^4} \\
    &\qquad + \sum_{i=1}^n [\beta(s)b_i(x,s)(1 - x_i) - \delta(s)x_i]\frac{-x_i}{\|x\|^2} \\
    &= \delta(s) + \sum_{i = 1}^n \frac{1}{2}[\sigma_i(x_i,s)b_i(x,s)(1 - x_i)]^2\frac{-\|x\|^2 + 2x_i^2}{\|x\|^4} \\
    &\qquad + \sum_{i=1}^n [\beta(s)b_i(x,s)(1 - x_i)]\frac{-x_i}{\|x\|^2} \\
    \Gamma V(x,s) &= \sum_{i=1}^n [\sigma_i(x_i,s)b_i(x,s)(1 - x_i)]^2\Big(\frac{-x_i}{\|x\|^2}\Big)^2\,.
\end{aligned}
\end{equation} 
 Since $b_i(x,s)$ is a linear function of $x$, it satisfies $|b_i(x,s)| \lesssim \|x\|$ and thus $\Gamma V$ is uniformly bounded on $\inv$. It follows from \Cref{switch-domain} that $V \in \Dme_2(\inv)$ with $\mathcal{L} V = L V$ and $\Gamma V$ as above.
Since
$$\lim_{x \to 0} \sum_{i=1}^n [\beta(s)b_i(x,s)(1 - x_i)]\frac{-x_i}{\|x\|^2}$$ does not exist (it depends on the direction $x$ approaches $0$), we need to enlarge the state space. Let $\pi: S^{N-1} \times [0,\infty) \to \R^N$ denote the polar coordinate map, that is, $\pi(v, r) = rv$ for $v \in S^{N-1} \coloneqq \{v \in \R^N \mid \|v\| = 1\}$, $r \in [0,\infty)$. Let $\mathcal{N} = \pi^{-1}([0,1]^N) \times S$, $\mathcal{N}_0 = \pi^{-1}(\{0\}) \times S$, $\mathcal{N}_+ = \mathcal{N} \setminus \mathcal{N}_0$.

It is routine to show that there exists a Markov process $\{(v(t),r(t),s(t))\}_{t \geq 0}$ on $\mathcal{N}$ such that $(x(t),s(t)) = (\pi(v(t),r(t)),s(t))$, but we provide some details anyway. For $x = \pi(v,r)$ we have $r = \|x\|$, $v = \frac{x}{\|x\|}$, and then It\^ o's Formula (note that $dx_idx_j = 0$) gives us
\begin{equation}\label{eq:srer}
    \begin{aligned}
    dr &= \sum_{i=1}^n \frac{x_idx_i}{r} + \frac{1}{2}\sum_{i=1}^n \frac{r - x_i^2/r}{r^2}(dx_i)^2 \\
    &= r\Big[\sum_{i=1}^n v_i\frac{dx_i}{r} + \frac{1}{2}\sum_{i=1}^n (1 - v_i^2)\Big(\frac{dx_i}{r}\Big)^2\Big] 
\end{aligned}
\end{equation}
and
\begin{equation}\label{eq:srev}
\begin{aligned}
    dv_i &= x_idr^{-1} + r^{-1}dx_i + dx_idr^{-1}  \\
    &= v_i\Big(\frac{-dr}{r} + \Big(\frac{dr}{r}\Big)^2\Big) + \frac{dx_i}{r} - \frac{dx_i}{r}\frac{dr}{r} \,.
\end{aligned}
\end{equation}
Since 
$$
\frac{dx_i}{r} = [\beta(s)b_i(v,s)(1 - rv_i) - \delta(s)v_i] dt + \sigma_i(rv_i,s)b_i(v,s)(1 - rv_i)dW_i(t)
$$ 
we obtain that \eqref{eq:srer} and \eqref{eq:srev} define a Markov process $\{(v(t),r(t),s(t))\}_{t \geq 0}$ (again suppressing the superscript for the initial condition) on $\mathcal{N}$.
It is standard to verify that $(\mathcal{N}, \mathcal{N}_0, \mathcal{N}_+, \{(v(t),r(t),s(t))\}_{t \geq 0})$ is a Feller quadruple and $\tilde{\pi} \coloneqq (v,r,s) \mapsto (rv,s)$ is a quadruple map from $(\mathcal{N}, \mathcal{N}_0, \mathcal{N}_+, \{(v(t),r(t),s(t))\}_{t \geq 0})$ to $(\mcM, \mcM_0, \inv, \{(x(t),s(t))\}_{t \geq 0})$ (see \Cref{quadruple-map}).

Substituting $x = rv$ into the formula for $L V(x,s)$ \eqref{LV-SIS}, we define the continuous extension $H: \mathcal{N} \to \R$ of $\mathcal{L}V \circ \tilde{\pi}$ by \begin{align*}
    H(v,r,s) &\coloneqq \delta(s) + \sum_{i = 1}^n \frac{1}{2}[\sigma_i(rv_i,s)b_i(rv,s)(1 - rv_i)]^2\frac{-r^2 + 2v_i^2r^2}{r^4} \\
    &\qquad + \sum_{i=1}^n [\beta(s)b_i(rv,s)(1 - rv_i)]\frac{-rv_i}{r^2} \\
    &= \delta(s) + \sum_{i = 1}^n \frac{1}{2}[\sigma_i(rv_i,s)b_i(v,s)(1 - rv_i)]^2(-1 + 2v_i^2) \\
    &\qquad - \sum_{i=1}^n [\beta(s)b_i(v,s)(1 - rv_i)]v_i \,,
\end{align*}
where the equality is by the linearity of $b$ in $x$.
Since $\sigma(0, s) = 0$ and $b(v, x) = A(s)v$, for $(v,0,s) \in \ncN_0$ we have
\begin{align*}
    H(v,0,s) &= \delta(s) - \beta(s)\sum_{i=1}^n b_i(v,s)v_i 
    = \delta(s) - \beta(s)v^TA(s)v \,.
\end{align*}
Let $\lambda_1(s)$ denote the largest eigenvalue of $A(s)$. For any $\mu \in P_{inv}(\mathcal{N}_0)$, by using $\|v\|=1$, we have 
$$
\mu H \geq \mu(\delta(s) - \beta(s) \lambda_1(s)) = \sum_{s = 1}^m \rho_s(\delta(s) - \beta(s) \lambda_1(s)) \,,
$$ 
where recall that $(\rho_1, \dots, \rho_m)$ are components of the unique invariant measure for $s(t)$. Thus,
if $\sum_{s = 1}^m \rho_s(\delta(s) - \beta(s) \lambda_1(s)) > 0$ then
there is a constant $\alpha > 0$ such that $\mu H \geq \alpha$ for all $\mu \in P_{inv}(\ncN_0)$, and so \eqref{as4.3-sub} is satisfied.

Recall that $\sigma_i(x_i,s) > 0$ for $x_i \in (0,1)$ and $\sigma_i$ is continuously differentiable in $x_i$, and consequently by the Stroock-Varadhan support theorem (see \cite{SVSupport}), $(0,s)$ is accessible from $(x,s)$ in the sense of \Cref{suff-acc}, and thus \eqref{acc} is satisfied.

Furthermore, by \Cref{compact} we have that \Cref{as3}, \Cref{as4}\ref{4.2}, and \Cref{as5} are satisfied for the Feller quadruple $(\mcM, \mcM_0, \inv, \{X_t^x\}_{x \in \mcM, t \geq 0})$. 

Thus, \Cref{change-of-variables} and \Cref{main3} provide the following result. 

\begin{thm}\label{SIS}
    If $\sum_{s = 1}^m \rho_s(-\delta(s) + \beta(s) \lambda_1(s)) < 0$
    then for any initial condition almost surely 
    $$
    \limsup_{t \to \infty}\frac{\log{\|x\|}}{t} \leq \sum_{s = 1}^m \rho_s(-\delta(s) + \beta(s) \lambda_1(s)) < 0 \,,
    $$ 
    and in particular $x \to 0$ exponentially fast. 
\end{thm}
\begin{rem}
    Observe that \cite[Theorem 2]{SIS} shows the exponential convergence under the stronger assumption that  
    $$
    \sum_{s = 1}^m \rho_s\Big(-\delta(s) + \beta(s) \lambda_1(s) + \frac{M(s)^2\lambda_1(s)^2}{32}\Big) < 0 \,,
    $$ 
    where $M(s) > 0$ is a global bound on the strength of the noise $\frac{\sigma_i(x_i,s)}{x_i}$ for $x \in (0,1)^N$. Our \Cref{SIS} shows that the extra term $\frac{M(s)^2\lambda_1(s)^2}{32}$ is unnecessary. Heuristically, the term $\frac{M(s)^2\lambda_1(s)^2}{32}$ should not be needed,  since it depends on the behavior of the process everywhere on $\mcM$, even at points which are far from $\mcM_0$, which  should not be affecting the rate of convergence to $\mcM_0$, as long as the accessibility of $\mcM_0$ holds true.
\end{rem}

\begin{rem}
    If the adjacency matrix of the network $A(s) = A$ is independent of $s$ and the network is connected, then using the persistence theorem in \cite{persistence} and observing that the only invariant measure of $(v(t),s(t))$ on $\mathcal{N}_0$ is the one where $s(t)$ is distributed according to $\rho$ and $v(t)$ is constantly equal to the (unique) positive eigenvector of $A$ with norm 1, one obtains that the condition $\sum_{s = 1}^m \rho_s(-\delta(s) + \beta(s) \lambda_1) > 0$ gives persistence.
    In the general case, one could also take $V(x,s) = -\log{\sum_{i = 1}^N x_i}$ to obtain that the condition $\sum_{s = 1}^m \rho_s(-\delta(s) + \beta(s) d_{min}(s)) > 0$ gives persistence, where $d_{\min}(s)$ is the minimum degree of nodes of the network with adjacency matrix $A(s)$, an improvement over \cite[Theorem 6]{SIS}. 
\end{rem}

\subsection{Lorenz System} \label{example-lorenz}
In this example taken from \cite{Lorenz}, $\mcM = \R^3$, $\mcM_0 = \{(x,y,z) \in \R^3 \mid x = y = 0\}$ (the $z$-axis), and $\inv = \mcM_0^c$. The Markov process $(X_t,Y_t,Z_t)$ is given by the Lorenz system with constant noise in the $Z$ component. Specifically, there are constants $\sigma, \rho, \beta, \hat{\alpha_0} > 0$ and a (one-dimensional) Brownian motion $W_t$ such that $(X_t,Y_t,Z_t)$ satisfies the SDE \begin{align*}
    dX &= \sigma (Y - X)dt \\
    dY &= [X(\rho - Z) - Y]dt \\
    dZ &= [-\beta Z + XY]dt + \hat{\alpha_0}dW
\end{align*}
(We use $\hat{\alpha_0}$ instead of $\hat{\alpha}$ in order to avoid confusion with the $\alpha$ in \Cref{as4}.)

To consolidate the constants, \cite[(2.1)]{Lorenz} makes a linear change of variables $(x_t,y_t,z_t) = (c_1X_{\chi t}, c_2(Y_{\chi t} - X_{\chi t}), z_* - c_3Z_{\chi t})$, where $c_1,c_2,c_3,z_*,\chi > 0$ are constants so that the system is rewritten as in \cite[(2.2)]{Lorenz}:
\begin{align*}
    dx &= ydt \\
    dy &= [x(z - 2) - 2y]dt \\
    dz &= -[\gamma (z - z_*) + x(x + \eta y)]dt + \alpha_0 dW\,,
\end{align*}
where $\gamma, \eta > 0$ are constants and $\alpha_0 = c_4 \hat{\alpha_0}$ ($c_4 > 0$ is also a constant).

The operator $L$ in \eqref{switch-gen} is given by \cite[(2.5a)]{Lorenz}:
\begin{equation}\label{lorenz-L}
    L = y\partial_x + [x(z - 2) - 2y]\partial_y -[\gamma (z - z_*) + x(x + \eta y)]\partial_z + \frac{\alpha_0^2}{2}\partial_z^2 \,.
\end{equation}
As in \cite[(4.11)]{Lorenz}, we define 
$$
U(x, y, z) := \tilde{U}(X,Y,Z) := \exp(\frac{\beta}{2\hat{\alpha_0}^2}(X^2 + Y^2 + (Z - \sigma - \rho)^2)) 
$$
and we remark that the notation $V_1,\tilde{V_1}$ was used instead of $U,\tilde{U}$ in 
\cite{Lorenz}. By straightforward calculations (see \cite[(4.14)]{Lorenz}) we have $LU \leq K - c(1 + x^2 + y^2 + z^2)U$ for some constants $K, c > 0$.

\begin{rem}\label{indep-of-alpha}
    In fact, using $\tilde{U}$ with $\frac{\beta}{2\hat{\alpha_0}^2}$ replaced by a constant $A$ shows that for fixed $\sigma, \beta$ there are constants $A,K,c,p > 0$ for which $L\tilde{U} \leq K - c(1 + x^2 + y^2 + z^2)\tilde{U}$ as long as $\rho < 1$ and $\alpha_0 \leq p$.
\end{rem}

By standard arguments $\inv, \mcM_0$ are invariant, and consequently
 \Cref{switch-is-feller} with $W = \tilde{U}$ implies that 
 $(\mcM, \mcM_0, \inv, \{(x_t,y_t,z_t)\}_{t \geq 0})$ is a Feller quadruple (see \Cref{quadruple}). In particular, \Cref{as1}--\ref{as2} are satisfied. 
 
Also, \cite[(2.3)]{Lorenz} introduces a ``cylindrical" change of variables via the map $\pi: S^1 \times [0,\infty) \times \R \to \R^3$ given by $\pi(\theta, R, z) = (R\sin \theta, R(\cos \theta - \sin \theta), z)$, where $S^1$ is viewed here as $\R / 2\pi\mathbb{Z}$. With $\mathcal{N} = S^1 \times [0,\infty) \times \R$, $\mathcal{N}_0 = S^1 \times \{0\} \times \R$, and $\mathcal{N}_+ = \mathcal{N} \setminus \mathcal{N}_0$ we see that $\pi$ is a quadruple map (see \Cref{quadruple-map}) between the Feller quadruples $(\ncN, \ncN_0, \ncN_+, \{(\theta_t, R_t, z_t)\}_{t \geq 0})$ and $(\mcM, \mcM_0, \inv, \{(x_t,y_t,z_t)\}_{t \geq 0})$ (we suppress superscripts for the initial condition), where the dynamics of $(\theta_t, R_t, z_t)$ are governed by the SDE \begin{equation}\label{theta-R-z}\begin{aligned}
    d\theta &= [1 - z\sin^2 \theta] dt \\
    dR &= R[-1 + \frac{z}{2}\sin (2\theta)]dt \\
    dz &= -[\gamma (z - z_*) + R^2 \sin \theta(\sin \theta + \eta (\cos \theta - \sin \theta))]dt + \alpha_0 dW
\end{aligned}\end{equation}
We define $V: \inv \to \R$ by $V(x,y,z) = -\frac{1}{2}\log(x^2 + (x+y)^2) = -\log{R}$ and 
note that such $V$ satisfies \Cref{as4}\ref{4.1}. Since
$LV \circ \pi = -\frac{dR}{R} = 1 - \frac{z}{2}\sin(2\theta)$ and $\Gamma V = 0$, then \Cref{as3}, \Cref{as4} \ref{4.2}, and \Cref{as5} follow from \Cref{cts-paths} with our fixed $V$ and $\UU = U$. Also, $\mathcal{L} V \circ \pi$ extends continuously to the function $H(\theta, R, z) = 1 - \frac{z}{2}\sin(2\theta)$ on $\ncN$.

\begin{rem}\label{indep-of-alpha2}
In fact, by \Cref{indep-of-alpha} and the definitions of $W,W',K,U,U'$ in the proof of \Cref{cts-paths}, it follows that \Cref{as3} and \Cref{as5} \ref{5.1}-\ref{5.3} are satisfied using the same $W,W',K,U,U'$ (independent of $\alpha_0$) as long as $\alpha_0$ is small enough.
\end{rem}
 
Concerning the accessibility, \eqref{acc} is satisfied (see \Cref{suff-acc}), 
by \cite[Proposition 3.3]{Lorenz} and the Stroock-Varhadan support theorem \cite{SVSupport}.

\begin{rem}
    It is not hard to show that \eqref{acc} is satisfied, even without using \cite[Proposition 3.3]{Lorenz}. In the notation of \cite[Proposition 3.3]{Lorenz}, we only need $|(x(T),y(T))| < \epsilon$ and $|z(T)| \leq M$ for some $M, T$ that may depend on $(x_0,y_0,z_0)$ but not on $\epsilon > 0$. Finding a control $h$ so that such a condition holds is not difficult since we can choose a fixed $z$ so that the eigenvalues of
    $\begin{bmatrix}
0 & 1 \\
z - 2 & -2 
 \end{bmatrix}  $ have negative real parts and then have $h$ force $z(t)$ to eventually be equal to $z$.
\end{rem}

Finally, we focus on the assumption \eqref{as4.3-sub} from \Cref{change-of-variables} and recall that $H(\theta, R, z) = 1 - \frac{z}{2}\sin(2\theta)$. 
Since for $\alpha_0 > 0$ the set $P_{inv}(\ncN_0)$ consists of a single measure $\mu_{\alpha_0}$ (see the discussion preceding \cite[Theorem 4.1]{Lorenz}), 
the assumption $\mu H \geq \alpha > 0$ for all $\mu \in P_{inv}(\ncN_0)$ is equivalent to 
\begin{equation}\label{lorenz-condition}
   -\lambda_{\alpha_0} \coloneqq \int_{S^1 \times \R} 1 - \frac{z}{2}\sin(2\theta) d\mu_{\alpha_0}(\theta, z) > 0 \,.
\end{equation}
Note that this is precisely the condition considered in \cite[Theorem 4.1]{Lorenz}. 

Thus, assuming \eqref{lorenz-condition},  \Cref{change-of-variables} and \Cref{main3} imply
$$\Prb\Big(\limsup_{t \to \infty} \frac{\frac{1}{2}\log(x_t^2 + (x_t+y_t)^2)}{t} \leq \lambda_{\alpha_0} < 0\Big) = 1$$ for any initial condition in $\mcM_+$.

 If $\alpha_0 = 0$ then  all invariant measures $\mu_0$ on $\ncN_0$ have $z$ constantly equal to $z_*$ and, analyzing the eigenvalues of $ \begin{bmatrix}
0 & 1 \\
z - 2 & -2 \\
 
\end{bmatrix}  $, $$\lambda_0 \coloneqq -\inf_{\mu_0 \in P_{inv}(\ncN_0)} \mu_0 H = \begin{cases}
   \sqrt{z_* - 1} - 1       & \text{if } z_* > 1 \\
   -1       & \text{otherwise} 
  \end{cases} \,.$$
 Fixing $R_t \equiv 0$ and $\alpha_t \equiv \alpha_0$ and noting that the coefficients of \eqref{theta-R-z} are locally Lipschitz in $(\theta,z,\alpha_0)$, we may use standard arguments similar to those in \Cref{switch-is-feller} to verify the Feller property (\Cref{as2}) for the Markov process $(\theta_t,0,z_t,\alpha_t)$. Then it follows from \Cref{indep-of-alpha2} that \Cref{robust} applies to the Markov Processes $(\theta_t, 0, z_t)$ on $\ncN_0$ for $\alpha_0 \in \Theta \coloneqq [0,p]$ (where $p$ is as in \Cref{indep-of-alpha}), and we conclude that for $\rho < 1$ (which implies $z_* < 2$, see \cite[page 5]{Lorenz}) we have  $\limsup_{\alpha_0 \downarrow 0} \lambda_{\alpha_0} \leq \lambda_0 < 0$,
and therefore \eqref{lorenz-condition} is satisfied for all sufficiently small $\alpha_0$.
In particular we recover the following half of \cite[Theorem 1.1]{Lorenz}.

\begin{thm}
    If $\rho < 1$ then there is a constant $\alpha_* > 0$ (depending on $\sigma, \beta$) such that if $0 \leq \hat{\alpha_0} < \alpha_*$ then, regardless of initial condition, $(x_t,y_t) \to 0$ exponentially fast almost surely.
\end{thm}

\begin{rem}
Using the exact same average Lyapunov function $V$, the second half of \cite[Theorem 1.1]{Lorenz} (persistence) follows from \cite{persistence}, \cite[Theorem 3.1]{Lorenz}, and \cite[Theorem 5.2]{Lorenz}. This shows that the only the construction of $V_1$ 
in \cite[Section 4]{Lorenz} is necessary for proving their main theorems. In particular,   there is no need to construct $V_0$.
\end{rem}

\subsection{Discrete Time Ecological Models} \label{example-ecological-discrete}

In this section we consider a general discrete-time ecological model discussed in \cite[(2.1)]{ecologicalGeneral}
\begin{equation}\label{eq:dtem}
\begin{aligned}
    X_i(t+1) &= X_i(t)F_i(Z(t),\xi(t)), \quad i = 1,\dots,n \\
    Y(t+1) &= G(Z(t),\xi(t)) \,.
\end{aligned}
\end{equation}

Here $\mcM = [0,\infty)^n \times \R^{\kappa_0}$, where each coordinate of $X = (X_1, \cdots, X_n) \in [0,\infty)^n$ represents the population of some species and $Y \in \R^{\kappa_0}$ represents auxiliary variables such as ``eco-environmental feedbacks, forcing, the structure of each species or other factors" \cite{ecologicalGeneral}. The stochasticity is represented by iid random variables $\{\xi(t)\}_{t \in \N}$ taking values in a Polish space $\Xi$ which models the state of the environment on the time interval $[t,t+1)$. The functions $F: \mcM \times \Xi \to (0,\infty)^n$ and $G: \mcM \times \Xi \to \R^{\kappa_0}$ are measurable in $(z,\xi) \in \mcM \times \Xi$ and continuous in $z$ for every fixed $\xi$. 

We also assume the existence of a proper function $\Upsilon: \mcM \to [1,\infty)$ and constants $\epsilon > 0, \rho \in (0,1), C > 0$ such that $\Pp \Upsilon \leq \rho^2 \Upsilon + C^2$ and \begin{equation}\label{eco-assumptions}
    \sum_{i=1}^n \E[|\log F_i(z, \xi(0))|^{2 + \epsilon}] \lesssim \Upsilon(z)^{\frac{1}{4}} \leq \Upsilon(z)^{\frac{1}{2}} \,,
\end{equation}
where the last inequality follows since $\Upsilon \geq 1$.

\begin{rem}
Compared to \cite{ecologicalGeneral}, we do not need to assume \cite[Assumption 2.2]{ecologicalGeneral}. Also, the existence of our $\Upsilon$ is weaker than \cite[Assumption 2.1 A3]{ecologicalGeneral} since for any fixed $\gamma > 0$, $|\log x|^4 \lesssim (x \vee \frac{1}{x})^\gamma$ for all $x > 0$.
\end{rem}
\begin{rem}
    By using more delicate arguments in \Cref{discretization} \ref{upsilon}, it suffices to assume $\E[|\log F_i(z, \xi(0))|^{2 + \epsilon}] \lesssim \Upsilon(z)^{\gamma}$ for some $\gamma < \frac{1}{2}$, but we choose $\gamma = \frac{1}{4}$ for ease and concreteness. Also, for \Cref{ecological-discrete-1} it is enough to take $\gamma < 1$.
\end{rem}

For $I \subset \{1,\dots,n\}$ denote 
\begin{align*}
    \mcM^I &\coloneqq \{(x,y) \in \mcM \mid x_i = 0 \quad \forall i \notin I\} \\
    \mcM_0^I &\coloneqq \{(x,y) \in \mcM^I \mid \exists i \in I \text{ such that } x_i = 0\} \\
    \mcM_+^I &\coloneqq \mcM^I \setminus \mcM_0^I
\end{align*}
When $I = \{1,\dots,n\}$ we suppress the superscript and write 
$\mcM$ and $\mcM_0$ instead of $\mcM^I$ and $\mcM^I_0$. Note that  $\mcM_0$ is then the set of states where at least one species is extinct. It is standard to verify (continuity of $F_i$ and the dominated convergence theorem) that \Cref{discretization} \ref{discrete-as1} and \ref{discrete-as2} are satisfied for $(\mcM,\mcM_0,\inv,\{Z(t)\}_{t \in \N})$ so that $(\mcM,\mcM_0,\inv,\{Z_t\}_{t \geq 0})$ is a Feller quadruple (see \Cref{feller-quadruple}).

\begin{rem}\label{notation-discrete}
    We slightly abuse notation by using notation $\{Z_t\}_{t \geq 0}$ also for the continuous-time Markov process $Z_t \coloneqq Z(N(t))$, where $N: [0,\infty) \to \N$ is a Poisson process. We keep this notation for the rest of this section, using $\cdot (t)$ for discrete-time Markov chains and $\cdot_t \coloneqq \cdot(N(t))$ for the continuous-time counterpart.
\end{rem}

For $i \in \{1,\dots,n\}$ we consider the functions $V_i: \inv \to \R$ defined by $V_i(x,y) = -\log{x_i}$. For any $(x,y) = z \in \inv$ we have 
\begin{equation}\label{eq:dovi}
V_i(Z(1)) - V_i(z) = -\log{x_iF_i(z,\xi(0))} + \log{x_i} = -\log{F_i(z,\xi(0))}
\end{equation}
and then by the definition of $\mathcal{L}$ in \eqref{discrete-gen} and our assumption \eqref{eco-assumptions}, it holds that 
\begin{equation}\label{V_i}
    \begin{aligned}
        \mathcal{L}V_i(z) &= -\E[\log F_i(z,\xi(0))] \\
        \E[|V_i(Z(1)) - V_i(z)|^{2 + \epsilon}] &\lesssim \Upsilon(z)^{1/4}
    \end{aligned}
\end{equation}
Recall that $F_i > 0$, and therefore 
$\log(F_i(z, \xi(0)))$ is well defined. 
Also, $\mathcal{L}V_i$ extends to the continuous function $H_i(z) \coloneqq -\E[\log F_i(z,\xi(0))]$ on $\mcM$, where the continuity follows from \eqref{eco-assumptions}, \Cref{generalized-DCT}, and \Cref{as2} (the Feller property requires continuity of $z \mapsto \mu^z_Z$, where $\mu^z_Z$ is the law of $Z_1^z$). Similarly, we define $\tilde{V}_i: \inv \to [0,\infty)$ by $\tilde{V}_i \coloneqq V_i \vee 0$. Since $a \mapsto a \vee 0$ is 1-Lipschitz, by \eqref{V_i}
\begin{equation} \label{V-tilde-bound}
    \E[|\tilde{V}_i(Z(1)) - \tilde{V}_i(z)|^{2 + \epsilon}] \lesssim \Upsilon(z)^{1/4} 
\end{equation}
and, as above, 
 $\mathcal{L}\tilde{V}_i$ extends to a continuous function $\tilde{H}_i$ on $\mcM$. Explicitly, since $\mathcal{L}\tilde{V_i}(z)$ is given by $-\E[\log F_i(z,\xi(0))] + \E[(-\log x_iF_i(z,\xi(0))) \wedge 0] - \E[(-\log x_i) \wedge 0]$, $\tilde{H}_i(z)$ can be defined by formally setting $(-\log 0) \wedge 0 = 0$, so in particular $\tilde{H}_i$ agrees with $H_i$ on $\{x_i = 0\}$.

Let us first focus on the case $I = \{1, \cdots, n\}$. 
If $p_i > 0$ are constants and $V \coloneqq \sum_{i=1}^n p_i V_i$, then by \eqref{V_i} and \Cref{discretization} \ref{upsilon} we have that \Cref{as3}, \Cref{as4} \ref{4.2}, and \Cref{as5} hold for $(\mcM,\mcM_0,\inv,\{Z_t\}_{t \geq 0})$. In addition, \Cref{as4}\ref{4.1} follows from the definition of $V$.

Next, we provide sufficient condition for \Cref{as4}\ref{4.3} to hold.
For ergodic $\mu \in P_{inv}(\mcM)$, \cite[(2.3)]{ecologicalGeneral} defines the ``expected per-capita growth rate of species i" 
\begin{equation}\label{invasion}
     r_i(\mu) \coloneqq -\mu H_i = -\mu \tilde{H}_i \,,
\end{equation} 
which is well defined by \Cref{inv-is-compact}. The second equality follows from \cite[Proposition 2.1]{ecologicalGeneral} and the fact that $\tilde{H}_i$ agrees with $H_i$ on $\{x_i = 0\}$. The following theorem (\cite[Theorem 2.4]{ecologicalGeneral}) is an immediate consequence of \Cref{discretization} \ref{disc-last} and \Cref{main3}:

\begin{thm}\label{ecological-discrete-1}
    Suppose there exist $p_i > 0$ such that for all ergodic $\mu \in P_{inv}(\mcM_0)$ 
    $$
    \sum_{i=1}^n p_ir_i(\mu) < 0 \,,
    $$ 
    so by \Cref{inv-is-compact} $V$ satisfies \Cref{as4}\ref{4.3}. 
    If $\mcM_0$ is accessible in the sense of \eqref{acc} (see \Cref{discretization} \ref{disc-acc}), then there exists an $\alpha > 0$ such that for all initial conditions $z \in \inv$ $$\Prb\Big(\limsup_{t \to \infty}\frac{\log \min_{1 \leq i \leq n} X_i^z(t)}{t} \leq -\alpha \Big) = 1 \,.$$
\end{thm}

\begin{rem}
    Note that the definition of accessibility in  \cite[pages 11-12]{ecologicalGeneral} implies \eqref{acc} as detailed in \Cref{suff-acc}.
\end{rem}

Next we focus on the case $I \subsetneq \{1,\dots,n\}$ and prove the assertions of \cite[Theorem 2.5]{ecologicalGeneral}. Note that as remarked in \cite{ecologicalGeneral}, by using \cite{ecologicalContinuous} it is possible to derive the (nontrivial) corollaries  \cite[Theorem 2.6, Theorem 2.7]{ecologicalGeneral}.

\begin{thm}\label{discrete-eco-thm-2}
    Suppose $I \subsetneq \{1,\dots,n\}$ is such that
    \begin{enumerate}[label=(\roman*)]
        \item \label{M+I} $P_{inv}(\mcM_+^I)$  is nonempty and all ergodic measures $\mu \in P_{inv}(\mcM_+^I)$ satisfy $\max_{i \notin I} r_i(\mu) < 0$.
        \item \label{M0I} For all $\nu \in P_{inv}(\mcM_0^I)$ (not necessarily ergodic), $\max_{i \in I} r_i(\nu) > 0$.
    \end{enumerate}
  Then there exists $\alpha_I > 0$ such that for any compact set $\K_I \subset \inv^I$, $$\lim_{z \to \K_I, z \in \inv} \Prb\Big(\limsup_{t \to \infty} \frac{\log \max_{i \notin I} X_i^z(t)}{t} \leq -\alpha_I\Big) = 1\,.$$
\end{thm}

\begin{rem}
   Recall that $\inv^I$ is the set where the species in $I$ all have positive population while all other species are extinct. As explained in \cite{ecologicalGeneral}, the assumption \ref{M+I} asserts that when all species in $I$ are alive and a potential invading species not in $I$ has small population, on average the invading species will die out. The assumption \ref{M0I} is equivalent to saying that, in the absence of other species, the species in $I$ will coexist (persist). Thus, heuristically, if the initial populations are close to $\inv^I$, then with high probability,  the species not in $I$ disappear exponentially fast.
\end{rem}

\begin{proof}
As above, it is standard to verify that  \Cref{discretization} \ref{discrete-as1} and \ref{discrete-as2} are satisfied for $(\mcM,\mcM^I,\inv,\{Z(t)\}_{t \in \N})$, so $(\mcM,\mcM^I,\inv,\{Z_t\}_{t \geq 0})$ is a Feller quadruple.
    
    First note that by \Cref{inv-is-compact} the sets 
    $\{(r_i(\nu))_{i \in I} \mid \nu \in P_{inv}(\mcM_0^I)\} \subset \R^{|I|}$ and $\{(r_i(\nu))_{i \notin I} \mid \nu \in P_{inv}(\mcM_0^I)\}$ are compact and convex. Thus, by Hahn-Banach separation theorem and \ref{M0I} there are $p_i > 0$ such that 
    \begin{equation}\label{eq:cbozm}
    \inf_{\nu \in P_{inv}(\mcM_0^I)} \sum_{i \in I} p_ir_i(\nu) > 0\,, \qquad 
    \sup_{\nu \in P_{inv}(\mcM_0^I)} \max_{i \notin I} |r_i(\nu)| < \infty \,.
    \end{equation}
    Similarly, $\{(r_i(\mu))_{i \notin I} \mid \mu \in P_{inv}(\inv^I)\}$ is compact and so \cite[Proposition 2.1]{ecologicalGeneral} and \ref{M+I} imply
    \begin{equation}\label{eq:cbopm}
    \sup_{\mu \in P_{inv}(\mcM_+^I)} \max_{i \notin I} r_i(\mu) < 0\,, \qquad 
    r_i(\mu) = 0 \quad \textrm{for any} \quad  i \in I, \mu \in P_{inv}(\mcM_+^I) \,.
    \end{equation}
    Since $\mcM^I = \inv^I \cup \mcM^I_0$, then any $\mu \in P_{inv}(\mcM^I)$
    can be decomposed as a convex combination of invariant measures on $\inv^I$ and $\mcM^I_0$ respectively. Thus, by \eqref{eq:cbozm} and \eqref{eq:cbopm} 
     there is $\hat{p} > 0$ such that
    \begin{equation}\label{ri-mu-pos}
        \inf_{\mu \in P_{inv}(\mcM^I)} \sum_{i \in I} p_ir_i(\mu) - \hat{p}\max_{i \notin I} r_i(\mu) > 0 \,.
    \end{equation}
    
    This motivates the choice of average Lyapunov function 
    \begin{equation}\label{eq:defvi}
    V_p \coloneqq -\sum_{i \in I} p_i\tilde{V}_i - \frac{\hat{p}}{p}\log \sum_{i \notin I} x_i^p
    \end{equation}
    for some $1 > p > 0$ to be determined later. Since $\tilde{V}_i \geq 0$, then \Cref{as4} \ref{4.1} holds and
    $$
\limsup_{t \to \infty} -\frac{V_p(X_i^z(t))}{t} \geq
\limsup_{t \to \infty}\frac{\hat{p}}{p}\log \sum_{i \notin I} (X_i^z(t))^p \geq 
    \hat{p}\limsup_{t \to \infty} \frac{\log \max_{i \notin I} X_i^z(t)}{t}  \,, 
$$
so the proof will be finished once we show that \Cref{main2} applies.

We will use \Cref{change-of-variables} with a suitable change of coordinates for the variables $x_i$ where $i \notin I$. Specifically, we introduce the Feller quadruple $(\ncN, \ncN_0, \ncN_+,$ $\{(v_p(t),r_p(t),w_p(t),y_p(t))\}_{t \geq 0})$ and the quadruple map $\pi_p$ as follows.

   Without loss of generality, otherwise relabel the variables, we assume that $I$ are the last $|I|$ coordinates of $[0,\infty)^n$. Let $\ncN = \bigtriangleup^{n - |I| - 1} \times [0,\infty) \times [0,\infty)^{|I|} \times \R^{\kappa_0}$, where $$
   \bigtriangleup^{n - |I| - 1} = \Big\{v \in [0,1]^{n - |I|} \mid \sum_{i=1}^{n - |I|} v_i = 1\Big\}
   $$ 
   is a simplex. For $v \in \bigtriangleup^{n - |I| - 1}$, $r \in [0,\infty)$, $w \in [0,\infty)^{|I|}$, $y \in \R^{\kappa_0}$ define $\pi_p: \ncN \to \mcM$ by 
   $$
   \pi_p(v,r,w,y) = (x,y) \text{ where } x_I = w \text{ and } x_i = rv_i^{1/p} \text{ for } i \notin I \,,
   $$ 
   where $x_I$ is the element of $[0,\infty)^{|I|}$ with coordinates equal to the last $|I|$ coordinates of $x$. 
   With $\ncN_0 \coloneqq \pi_p^{-1}(\mcM^I) = \{r = 0\}$ and $\ncN_+ \coloneqq \pi_p^{-1}(\inv)$, we obtain a Feller quadruple $(\ncN, \ncN_0, \ncN_+, \{(v_{p,t},r_{p,t},w_{p,t},y_{p,t})\}_{t \geq 0})$, where $(v_{p,t},r_{p,t},w_{p,t},y_{p,t})$ is the Markov process corresponding to the Markov chain $(v_p(t),r_p(t),w_p(t),y_p(t))$ (see \Cref{notation-discrete}) given by setting $(X(t),Y(t)) = \pi_p(v_p(t),r_p(t),w_p(t),y_p(t))$ in \eqref{eq:dtem}:
\begin{equation}\label{vrwy}
      \begin{aligned}
        v_i(t+1) &= v_i(t)\frac{\tilde{F}_i(v(t),r(t),w(t),y(t),\xi(t))^p}{\sum_{i \not \in I} \tilde{F}_i(v(t),r(t),w(t),y(t),\xi(t))^p} \quad i \notin I \\
        r(t+1) &= r(t)\Big(\sum_{i \not \in I} v_i(t) \tilde{F}_i(v(t),r(t),w(t),y(t),\xi(t))^p\Big)^{1/p} \\
        w(t) &= X_I(t) \\
        y(t) &= Y(t)
    \end{aligned}
\end{equation}
where we omit the subscript $p$ and use $\tilde{F}_i$ to denote $F_i \circ \pi_p$.  
    Then clearly $\pi_p$ is a quadruple map (see \Cref{quadruple-map}) from $(\ncN, \ncN_0, \ncN_+, \{(v_{p,t},r_{p,t},w_{p,t},y_{p,t})\}_{t \geq 0})$ to  $(\mcM,\mcM^I,\inv,\{Z_t\}_{t \geq 0})$.

Next we compute $\mathcal{L}V_p$ in the new and old coordinates, where recall $V_p$ was defined in \eqref{eq:defvi}. Since we have already discussed $\mathcal{L}\tilde{V}_i$ above, we provide details for the $\frac{\hat{p}}{p}\log \sum_{i \notin I} x_i^p$ term. Note that for the initial condition $z = (x,y) = \pi_p(\vec{a}) \in \inv$ where $\vec{a} = (v,r,w,y) \in \ncN_+$ we have
\begin{equation}\label{eq:x1-x0}
    \begin{aligned}
        \log \sum_{i \notin I} X_i^z(1)^p -  \log \sum_{i \notin I} X_i^z(0)^p &= 
\log \sum_{i \notin I} x_i^pF_i(x, y, \xi(0))^p -  \log \sum_{i \notin I} x_i^p
\\
&= \log \frac{\sum_{i \notin I} x_i^p F_i(x, y, \xi(0))^p}{\sum_{i \notin I} x_i^p} \\
&= \log \sum_{i \notin I} v_i^p\tilde{F_i}(\vec{a}, \xi(0))^p \,.
    \end{aligned}
\end{equation}
Thus, we have
\begin{align*}
\mathcal{L}V_p(x,y) &= -\sum_{i \in I} p_i\tilde{H}_i(x,y) - \frac{\hat{p}}{p} \E[\log \Big(\frac{\sum_{i \notin I} x_i^pF_i(x,y,\xi(0))^p}{\sum_{i \notin I} x_i^p}\Big)] \\
\mathcal{L}V_p \circ \pi_p(\vec{a}) &= -\sum_{i \in I} p_i\tilde{H}_i \circ \pi_p(\vec{a}) 
- \frac{\hat{p}}{p} \E\Big[\log \sum_{i \notin I} v_i\tilde{F}_i(\vec{a},\xi(0))^p\Big]
\end{align*}
Then for any $v \in \bigtriangleup^{n - |I| - 1}$ and $0 < p < 1$, since $|\log|$ is quasiconvex,
\begin{equation}\label{eq:ublft}
\Big|\log \sum_{i \notin I} v_i\tilde{F}_i(\vec{a},\xi(0))^p\Big|
\leq 
\max_{i \notin I} |\log \tilde{F}_i(\vec{a},\xi(0))^p|
\leq p\sum_{i \notin I} |\log \tilde{F}_i(\vec{a},\xi(0))| \,.
\end{equation}

In particular, \eqref{eco-assumptions} and \Cref{generalized-DCT} imply that $\mathcal{L}V_p \circ \pi_p$ extends to a continuous function $H_p$ on $\ncN$. By combining the definition of $V_p$ in \eqref{eq:defvi} with \eqref{eq:x1-x0}, \eqref{eq:ublft}, \eqref{eco-assumptions}, and \eqref{V-tilde-bound} we obtain
\begin{equation*}
     \E[|V_p(Z(1)) - V_p(z)|^{2 + \epsilon}] \lesssim \Upsilon(z)^{1/4} \,,
\end{equation*}
and therefore by \Cref{discretization} \ref{upsilon} we have that \Cref{as3}, \Cref{as4} \ref{4.2}, and \Cref{as5} hold for $(\mcM,\mcM^I,\inv,\{Z_t\}_{t \geq 0})$.

It remains to choose $1 > p > 0$ so that the assumption \eqref{as4.3-sub} in \Cref{change-of-variables} holds, that is, so that there is a constant $\alpha > 0$ such that $\mu H_p \geq \alpha$ for all $\mu \in P_{inv}(\ncN_0)$. By \eqref{ri-mu-pos} it suffices to show that 
\begin{equation}\label{p-to-0}
\lim_{p \downarrow 0} \inf_{\mu \in P_{inv}^p(\ncN_0)} \mu H_p \geq \inf_{\mu \in P_{inv}(\mcM^I)} \sum_{i \in I} p_ir_i(\mu) - \hat{p}\max_{i \notin I} r_i(\mu) \,,
    \end{equation}
where $P_{inv}^p(\ncN_0)$ is the set of all invariant measures supported on $\ncN_0$ for \eqref{vrwy} (see \Cref{discretization} \ref{discrete-inv-meas}).

To show \eqref{p-to-0}, we will define a function $H_0$ and a Markov Process which are (in some sense) the limits of $H_p$ and \eqref{vrwy} as $p \downarrow 0$. Then, we verify \eqref{p-to-0} with $\lim_{p \downarrow 0}$ removed and $H_p$ replaced by $H_0$. Finally, by \Cref{robust} we will conclude \eqref{p-to-0}. Next we provide details.

In the remainder of the argument we only deal with Markov processes on $\ncN_0 = \{r = 0\}$. Since $\pi_p(v,0,w,y)$, $\tilde{F}_i(v,0,w,y,\xi(0))$, and $\tilde{G}(v,0,w,y,\xi(0))$ are independent of $v$ and $p$, we write $\pi(w,y)$, $\tilde{F}_i(w,y,\xi(0))$, and $\tilde{G}(w,y,\xi(0))$ instead. Similarly, we write $\tilde{\Upsilon}(w,y)$ for $\Upsilon \circ \pi_p(v,0,w,y)$. Additionally, we drop $r$ and use $(v,w,y)$ to denote $(v,0,w,y) \in \ncN_0$.

We define the function $H_0$ on $\ncN_0$ by  
$$
H_0(v,w,y) = -\sum_{i \in I} p_i\tilde{H}_i \circ \pi(w,y) + \hat{p} \sum_{i \notin I} v_i H_i \circ \pi(w,y)
$$
and the Markov chain $\{(v_0(t),w_0(t),y_0(t))\}_{t \in \N}$ on $\ncN_0$ by
\begin{equation}\label{0-chain}
    \begin{aligned}
        v_0(t+1) &= v_0(t) \\
        w_0(t) &= X_I(t) \\
        y_0(t) &= Y(t) \,.
\end{aligned}
\end{equation}
Overall, we  defined a collection of Markov processes $\{\{(v_{p,t},w_{p,t},y_{p,t})\}_{t \geq 0}\}_{p \in [0,1/2]}$  on $\ncN_0$ corresponding to the Markov chains given by (omitting the subscript $p$)
\begin{equation}\label{big-chain}
      \begin{aligned}
        v_{i}(t+1) &= v_i(t)\frac{\tilde{F}_i(w(t),y(t),\xi(t))^p}{\sum_{i\notin I} v_i(t) \tilde{F}_i(w(t),y(t),\xi(t))^p} \quad i \notin I \\
        w_{i-n+|I|}(t+1) &= \tilde{F}_i(w(t),y(t),\xi(t))w_{i-n+|I|}(t) \quad i \in I \\
        y(t+1) &= \tilde{G}(w(t),y(t),\xi(t))
    \end{aligned}
\end{equation}
and a function $H: [0,1/2] \times \ncN_0 \to \R$ given by $$H(p,v,w,y) = \begin{cases}
    -\sum_{i \in I} p_i\tilde{H}_i \circ \pi(w,y) - \frac{\hat{p}}{p} \E[\log \sum_{i \notin I} v_i\tilde{F}_i(w,y,\xi(0))^p] & \text{if } p > 0 \\
    -\sum_{i \in I} p_i\tilde{H}_i \circ \pi(w,y) + \hat{p} \sum_{i \notin I} v_i H_i \circ \pi(w,y) & \text{if } p = 0
\end{cases} $$

To apply \Cref{robust}, we note that \Cref{discretization} \ref{discrete-as2} implies \Cref{as2} for the Markov process $(p_t,v_{p,t},w_{p,t},y_{p,t})$ on $[0,1/2] \times \ncN_0$, where $p_t \equiv p$ (simply look at \eqref{big-chain} and note that the right hand side is continuous in $(p,v,w,y)$). Also, since we have already noted that by \Cref{discretization} \ref{upsilon} that $Z_t$ satisfies \Cref{as3} and \Cref{as5} \ref{5.1}-\ref{5.3} with $U = \sqrt{\Upsilon}, U' = (1 - \rho) U, W = \Upsilon^{1/4}, W' = (1 - \sqrt{\rho})W$, and some $K > 0$, it follows easily from $\pi(w_{p,t}^{\vec{b}},y_{p,t}^{\vec{b}}) = Z_t^z$ for all initial conditions $\vec{b} = (v,w,y) \in \ncN_0$, $z = \pi(w,y) \in M^I$ that each $(v_{p,t},w_{p,t},y_{p,t})$ satisfies \Cref{as3} and \Cref{as5} \ref{5.1}-\ref{5.3} with $U = \sqrt{\tilde{\Upsilon}}, U' = (1 - \rho) \sqrt{\tilde{\Upsilon}}, W = \tilde{\Upsilon}^{1/4}, \text{ and } W' = (1 - \sqrt{\rho})\tilde{\Upsilon}^{1/4}$, and $K$ the same as for $Z_t$. In particular, $W,W',K,U,U'$ are independent of $p$.

To satisfy the rest of the assumptions of \Cref{robust}, we claim that $H$ is a continuous function which vanishes over $(p,v,w,y) \mapsto \tilde{\Upsilon}^{1/4}(w,y)$. The continuity on $(0,1/2) \times \ncN_0$ follows from \eqref{eco-assumptions}, \eqref{eq:ublft}, and \Cref{generalized-DCT}. Also, $H_0$ is clearly continuous. Thus, for continuity of $H$ it suffices to show that $H_p \to H_0$ uniformly on compact sets.

As a preliminary calculation, for fixed $(v,w,y) \in \ncN_0$ and $p \in (0,1/2)$, we set $\tilde{F}_i := \tilde{F}_i(w,y,\xi(0))$ and estimate by the mean value theorem that \begin{equation}\begin{aligned}\label{MVT1}
    \Big|\frac{v_i\tilde{F}_i^{p}}{\sum_{i \not \in I} v_i \tilde{F}_i^{p}} - v_i\Big| &\leq p\sup_{0 < p* < p} \Big| \frac{v_i \tilde{F}_i^{p*} \log \tilde{F}_i}{\sum_{i \not \in I} v_i \tilde{F}_i^{p*}}\Big| 
    + \Big|\frac{v_i \tilde{F}_i^{p*} \sum_{i \not \in I} v_i \tilde{F}_i^{p*} \log \tilde{F}_i}{(\sum_{i \not \in I} v_i \tilde{F}_i^{p*})^2}\Big| \\
    &\leq p|\log \tilde{F}_i| + p\max_{i \not \in I} |\log \tilde{F}_i| \leq 2p\sum_{i \not \in I} |\log \tilde{F}_i|
    \end{aligned}
\end{equation}

Applying the mean value theorem to $p \mapsto \log \Big(\sum_{i \notin I} v_i\tilde{F}_i^p\Big)$ and recalling that $\sum_{i \not \in I} v_i = 1$ and $H_i \circ \pi = -\E[\log \tilde{F_i}]$, we obtain
\begin{align*}
    |H_p(v,w,y) - H_0(v,w,y)| &= \hat{p}\Big|\E\Big[\frac{1}{p}\Big[\log \Big(\sum_{i \notin I} v_i\tilde{F}_i^p\Big) - \log \Big(\sum_{i \notin I} v_i\Big)\Big]  - \sum_{i \notin I} v_i\log \tilde{F_i}\Big]\Big| \\
    &\leq \sup_{0 < p* < p} \hat{p}\E\Big[\Big|\sum_{i \notin I} \Big(\frac{v_i\tilde{F}_i^{p*}}{\sum_{i \not \in I} v_i \tilde{F}_i^{p*}} - v_i\Big)\log \tilde{F_i}\Big|\Big] \\
    &\lesssim \hat{p}p\E\Big[\Big(\sum_{i \not \in I} |\log \tilde{F}_i|\Big)^2\Big] \lesssim \hat{p}p(\tilde{\Upsilon}^{1/4}(w,y))^{2/(2+\epsilon)} \,,
\end{align*}
where the first $\lesssim$ is by \eqref{MVT1} and the second is by \eqref{eco-assumptions}. This shows that $H_p \to H_0$ uniformly on compacts, and therefore $H$ is continuous. Additionally, it shows that $H$ vanishes over $(p,v,w,y) \mapsto \tilde{\Upsilon}^{1/4}(w,y)$ if and only if $H_0$ vanishes over $\tilde{\Upsilon}^{1/4}$. The latter is a consequence of the fact that $H_i$ and $\tilde{H_i}$ vanish over $\Upsilon^{1/4}$ (in fact, over $\Upsilon^{1/2}$), which is shown above (see \eqref{V_i} and \eqref{V-tilde-bound}).

Thus, the assumptions of \Cref{robust} are satisfied and so we conclude that
\begin{equation}\label{exponent-cts-in-p}
    \lim_{p \downarrow 0} \inf_{\mu \in P_{inv}^p(\ncN_0)} \mu H_p \geq \inf_{\nu \in P_{inv}^0(\ncN_0)} \nu H_0 \,,
\end{equation}
where $P_{inv}^0(\ncN_0)$ is the set of all invariant measures for \eqref{0-chain} (see \Cref{discretization} \ref{discrete-inv-meas}). By \Cref{inv-is-compact}, the right hand side of \eqref{exponent-cts-in-p} is unchanged if we restrict $\nu \in P_{inv}^0(\ncN_0)$ to be ergodic, in which case $\nu$ is supported on $\{v = v_0\}$ for some $v_0\in \bigtriangleup^{n - |I| - 1}$ (recall from \eqref{0-chain} that $v_0(t+1) = v_0(t)$). In particular, there is $\mu \in P_{inv}(\mcM^I)$ (equal to the pushforward of $\nu$ by $\pi$) such that
$$\nu H_0 = -\sum_{i \in I} p_i\mu \tilde{H}_i + \hat{p} \sum_{i \notin I} (v_0)_i \mu H_i \geq \inf_{\mu \in P_{inv}(\mcM^I)} \sum_{i \in I} p_ir_i(\mu) - \hat{p}\max_{i \notin I} r_i(\mu) \,.$$ 
Thus, \eqref{exponent-cts-in-p} implies \eqref{p-to-0}, which finishes the proof.
\end{proof}

\subsection{Stochastic Kolmogorov Systems} \label{example-ecological-continuous}
In this subsection we mention how our results can be applied to Stochastic Kolmogorov Equations, that is, SDEs of the form
\begin{equation}\label{Kolmogorov}
    dX_i = X_if_i(X)dt + X_ig_i(X)dE_i \,,
\end{equation}
where $X_t \in [0,\infty)^n$, $f,g: [0,\infty)^n \to \R^n$ are locally Lipschitz functions, and $(E_1(t),\dots,E_n(t))^T = A^T W(t)$, where $A$ is a $d$ by $n$ matrix and $W(t)$ is a $d$-dimensional Brownian motion. We denote $A^T A$ by $\Sigma = (\Sigma_{ij})_{1 \leq i,j \leq n}$. Similarly to the previous example, we could have also included auxiliary variables $Y_t$, but we omit these for brevity. Our technical assumption is that there exists a proper $U: [0,\infty)^n \to [1,\infty)$ in $\C^2([0,\infty)^n)$ and $K,c > 0$ such that $\mathcal{L}U \leq K - cU$ and $\sum_{i=1}^n |f_i(x)| + g_i(x)^2 \lesssim 2K - \frac{\mathcal{L}U}{U} + \frac{\Gamma U}{U^2}$, where \begin{equation*}
    \begin{aligned}
        \mathcal{L}U(x) &= \sum_{i,j = 1}^n \frac{1}{2}\Sigma_{ij}x_ix_jg_i(x)g_j(x)\partial_i \partial_j U(x) + \sum_{i=1}^n x_if_i(x)\partial_iU(x) \\
        \Gamma U(x) &= \sum_{i,j = 1}^n \Sigma_{ij}x_ix_jg_i(x)g_j(x)\partial_i U(x) \partial_j U(x) \,.
    \end{aligned}
\end{equation*}
\begin{rem}
 If \cite[Assumption 1.1 (3)]{ecologicalContinuous} is satisfied then our technical assumption holds with $U = (1 + c^Tx)^\theta$ for some small enough $1 > \theta > 0$. Similarly to \cite[Assumption 3.1]{ecologicalGeneral}, if there exists a proper $U: [0,\infty)^n \to [1,\infty)$ in $\C^2([0,\infty)^n)$ such that $\mathcal{L}U \leq K - cU\Big(1 + \sum_{i=1}^n |f_i(x)| + g_i(x)^2\Big)$, then our technical assumption is satisfied.
 We have no need of \cite[Assumption 1.1 (1)]{ecologicalContinuous} (nondegeneracy), \cite[Assumption 1.4]{ecologicalContinuous}, or \cite[Assumption 3.2]{ecologicalGeneral}, although nondegeneracy is useful for deducing uniqueness of invariant measures on certain sets and for deducing accessibility.
\end{rem}

With $H_i(x) = \frac{1}{2}\Sigma_{ii}g_i(x)^2 - f_i(x)$ and $r_i(\mu)$ defined as in \eqref{invasion}, using \Cref{cts-paths} in place of \Cref{discretization} \ref{upsilon} one can show that the analogues of \Cref{ecological-discrete-1} and \Cref{discrete-eco-thm-2} hold for \eqref{Kolmogorov}. We omit the proofs since they are exactly the same as those given in \Cref{example-ecological-discrete} except that the details are easier (It\^ o's formula gives a closed-form expression for $H_p$ so we can directly compute $H_p - H_0$ as opposed to using the mean value theorem twice to bound $|H_p - H_0|$) and also instead of defining $\tilde{V_i}$ as $V_i \vee 0$ one should choose some smooth function $v: \R \to [0,\infty)$ with bounded first and second derivatives with $v(t) = t$ for $t \geq 1$ and instead set $\tilde{V_i} \coloneqq v \circ V_i$ (as in the proof of \cite[Theorem 5.1ii]{persistence}).

Finally, we remark that the results in \cite{rps} (under the weaker assumptions described above) are corollaries of our theory. For example, \cite[Theorem 6.2]{rps} can be proved by considering $V(x_1,x_2,x_3) = -p_1\ln{x_1} - p_2\ln{x_2} - p_3\ln{x_3} - \frac{p_0}{p}v(-\ln({x_1^p + x_2^p + x_3^p}))$ for small enough $p > 0$ and suitable $p_0,p_1,p_2,p_3 > 0$ (see \cite[Lemma 6.1]{rps}).

\newcommand{\etalchar}[1]{$^{#1}$}
\providecommand{\bysame}{\leavevmode\hbox to3em{\hrulefill}\thinspace}
\providecommand{\MR}{\relax\ifhmode\unskip\space\fi MR }
\providecommand{\MRhref}[2]{%
  \href{http://www.ams.org/mathscinet-getitem?mr=#1}{#2}
}
\providecommand{\href}[2]{#2}

\end{document}